\documentclass[11pt]{article}

\usepackage{verbatim,amsfonts,graphicx,amsmath,amssymb,amsthm,geometry,
hyperref,color,nicefrac,enumitem,relsize,titling}
\usepackage[latin1]{inputenc}
\usepackage{bbm} % for indicator function: \mathbbm{1}
\newcommand{\R}{{\mathbb R}}
\newcommand{\N}{{\mathbb N}}

\newcommand{\Q}{{\mathbb Q}}
\newcommand{\EE}{\mathbb E}
\newcommand{\bEE}[1]{\mathbb E\bigl[ #1 \bigr]}

\newcommand{\bbbEE}[1]{\mathbb E\biggl[ #1 \biggr]}
\newcommand{\bbbbEE}[1]{\mathbb E\Biggl[ #1 \Biggr]}
\newcommand{\pp}[1]{(#1)}
\newcommand{\bp}[1]{\bigl(#1\bigr)}
\newcommand{\bbp}[1]{\Bigl(#1\Bigr)}
\newcommand{\bbbp}[1]{\biggl(#1\biggr)}
\newcommand{\bbbbp}[1]{\Biggl(#1\Biggr)}
\newcommand{\br}[1]{[#1]}
\newcommand{\bbr}[1]{\bigl[#1\bigr]}
\newcommand{\bbbr}[1]{\Bigl[#1\Bigr]}
\newcommand{\bbbbr}[1]{\biggl[#1\biggr]}
\newcommand{\bbbbbr}[1]{\Biggl[#1\Biggr]}
\newcommand{\PP}{{\mathbb P}}

\newcommand{\diff}{\mathrm{d}}
\newcommand{\1}{{\mathbbm 1}}

\newcommand{\eps}{\varepsilon}

\newcommand{\mc}[1]{\mathcal{#1}}
\newcommand{\mf}[1]{\mathfrak{#1}}
\newcommand{\norm}[1]{\lVert #1\rVert}
\newcommand{\bnorm}[1]{\bigl\lVert #1\bigr\rVert}
\newcommand{\bbnorm}[1]{\Bigl\lVert #1\Bigr\rVert}
\newcommand{\bbbnorm}[1]{\biggl\lVert #1\biggr\rVert}
\newcommand{\nnorm}[1]{{\lvert\kern-0.25ex\lvert\kern-0.25ex\lvert #1 
    \rvert\kern-0.25ex\rvert\kern-0.25ex\rvert}}
\newcommand{\ang}[1]{\langle #1 \rangle}
\newcommand{\abs}[1]{\lvert #1 \rvert}
\newcommand{\babs}[1]{\bigl\lvert #1 \bigr\rvert}
\newcommand{\bbabs}[1]{\Bigl\lvert #1 \Bigr\rvert}
\newcommand{\bbbabs}[1]{\biggl\lvert #1 \biggr\rvert}

\def\beq#1\eeq{\begin{equation}#1\end{equation}}
\def\ba#1\ea{\begin{equation}\begin{split}#1\end{split}\end{equation}}
\def\bm#1\em{\begin{multline}#1\end{multline}}

\theoremstyle{plain}
\newtheorem{theorem}{Theorem}[section]
\newtheorem{prop}[theorem]{Proposition}
\newtheorem{lemma}[theorem]{Lemma}
\newtheorem{cor}[theorem]{Corollary}

\theoremstyle{definition}

\newcommand{\opnorm}[1]{\norm{#1}_{\mathrm{op}}}

\newcommand{\qqandqq}{\qquad\text{and}\qquad}

\newcommand{\T}[1]{\marginpar{\textcolor{red}{#1}}}

\newcommand{\intrtype}[1]{#1}
\newcommand{\intrtypen}[1]{}
\newcommand{\ann}[2]{#1}
\newcommand{\com}[1]{\ignorespaces}

\newcommand{\is}{\leftarrow}

\newcommand{\id}{\mathrm{id}}

\newcommand{\extR}{[-\infty,\infty]}

\title{On the strong regularity of degenerate additive\\
noise driven stochastic differential equations \\
with respect to their initial~values}

\author{Arnulf Jentzen$^1$,
  Benno Kuckuck$^2$,\\
  Thomas M\"uller-Gronbach$^3$, 
  and 
  Larisa Yaroslavtseva$^{4}$
  \bigskip
  \\
  \small{$^1$Department of Mathematics, 
    ETH Zurich,}
  \\
    \small{e-mail: 
    arnulf.jentzen@sam.math.ethz.ch}
  \smallskip
  \\
  \small{$^2$Mathematical Institute,  
    Universit\"at D\"usseldorf,}
  \\
  \small{e-mail:
    kuckuck@math.uni-duesseldorf.de}
  \smallskip
  \\
  \small{$^3$Faculty of Computer Science and Mathematics, 
    Universit\"at Passau,}
  \\
    \small{e-mail: 
    thomas.mueller-gronbach@uni-passau.de} 
\smallskip
\\
\small{$^4$Faculty of Computer Science and Mathematics, 
  Universit\"at Passau,}
\\
\small{e-mail: 
  larisa.yaroslavtseva@uni-passau.de} 
}

\begin{document}

\maketitle
\begin{abstract}
 Recently in [M.~Hairer, M.~Hutzenthaler, and A.~Jentzen, \textit{Ann.\ Probab.\ 43}, 2 (2015), 468--527] 
 and [A.~Jentzen, T.~M\"uller-Gronbach, and L.~Yaroslavtseva, \textit{Commun.\ Math.\ Sci.\ 14}, 6 (2016), 1477--1500]
 stochastic differential equations (SDEs) 
 with smooth coefficient functions have  been constructed which have an 
 arbitrarily slowly converging modulus of continuity in the initial value.
 In these SDEs it is crucial that some of  the first order partial derivatives 
 of the drift coefficient functions grow at least exponentially and, 
 in particular, quicker than any polynomial. 
 However, in applications SDEs do typically have coefficient functions 
 whose  first order partial derivatives are polynomially bounded. 
 In this article we study  whether 
 arbitrarily bad regularity phenomena in the initial value may also arise 
 in the latter case
 % of SDEs with coefficient functions whose  first order partial 
 %derivatives grow at most polynomially. 
 and we partially answer this question in the negative.  
 More precisely, we show that every additive noise driven 
 SDE which admits a Lyapunov-type condition (which ensures the existence 
 of a unique solution of the SDE)
 and which has a drift coefficient function whose first order partial derivatives 
 grow at most polynomially
 is at least logarithmically H\"older continuous in the initial value. 
\end{abstract}
\newpage 

\tableofcontents

\section{Introduction} 

The regularity analysis of nonlinear stochastic differential equations (SDEs) 
with respect to their initial values is an active research topic in 
stochastic analysis (cf., e.g., \cite{CL14, CHJ13, FIZ07, HHJ15, HHM19, HHM192, Kry99, Li94, LS11, SS17, Zhang10} and the 
references mentioned therein). In particular, it has recently been revealed in the 
literature that there exist SDEs with smooth coefficient functions 
which have very poor regularity properties in the initial value.  
More precisely, it  has been shown  in~\cite{HHJ15}  that there exist 
additive noise driven SDEs with infinitely often differentiable 
drift coefficient functions which have a modulus of continuity in the 
initial value that converges to zero slower than with any polynomial rate. 
Moreover, in~\cite{JMGY16}  additive noise driven SDEs with infinitely 
often differentiable drift coefficient functions have been constructed 
which even  have an arbitrarily slowly converging modulus of continuity 
in the initial value.
In these SDEs it is crucial that the first order partial derivatives 
of the drift coefficient functions grow at least exponentially and, 
in particular, quicker than any polynomial. However, in applications 
SDEs do typically have coefficient functions whose  first order partial 
derivatives grow at most polynomially 
(cf., e.g.,~\cite{ ans96,  df12, h07c,  Leung1995,  Schmalfuss1997,  Schoeneretal1985, TimmerEtAl2000,  TornatoreBuccellatoVetro2005},  
\cite[Chapter 7]{kp92},  and \cite[Chapter 4]{HJ2015} for examples). 
In particular, in many applications the coefficient functions of the 
SDEs under consideration are polynomials 
(cf., e.g.,~\cite{ ans96,  df12, Schmalfuss1997, TimmerEtAl2000,  TornatoreBuccellatoVetro2005}, 
\cite[Chapter 7]{kp92}, and \cite[Chapter 4]{HJ2015}   for examples). 
In view of this, the natural question arises whether such arbitrarily 
bad regularity phenomena in the initial value may also arise in the case 
of SDEs with coefficient functions whose  first order partial 
derivatives grow at most polynomially. It is the subject of the main result 
of this article to partially answer this question in the negative.  
More precisely, the main result of this article, 
Theorem~\ref{thm:introduction} below, shows that every additive noise driven 
SDE which admits a 
Lyapunov-type condition (which ensures the existence of a unique solution of the SDE)
and which has a drift coefficient function whose  first order partial 
derivatives grow at most polynomially
is at least logarithmically H\"older continuous in the initial value. 

\begin{theorem}
  \label{thm:introduction}
  Let $d,m\in\N$, 
  $T,\kappa\in[0,\infty)$,
  $\alpha\in[0,2)$,
  $\mu\in C^1(\R^d,\R^d)$,
  $\sigma\in\R^{d\times m}$,
  $V\in C^1(\R^d,[0,\infty))$,
  let $\norm{\cdot}\colon\R^d\to[0,\infty)$ and $\nnorm\cdot\colon\R^m\to[0,\infty)$ be norms,
  assume for all 
    $x,h\in\R^d$, 
    $z\in\R^m$ 
  that
    $\norm{\mu'(x)h}\leq \kappa \bp{1+\norm x^\kappa} \norm h$,
    $V'(x)\mu(x+\sigma z)\leq \kappa(1+\nnorm{z}^\alpha)V(x)$,
    and $\norm x\leq V(x)$,
  let $(\Omega,\mc F,\PP)$ be a probability space,
  and let $W\colon[0,T]\times\Omega\to\R^m$ be a standard Brownian motion
    %on $(\Omega,\mc F,\PP)$ 
    with continuous sample paths.
  Then
  \begin{enumerate}[label=(\roman{enumi}),ref=(\roman{enumi})]
  \item \label{enum:thmintroduction:1}
    there exist unique stochastic processes
    $X^x\colon [0,T]\times\Omega\to\R^d$, $x\in\R^d$, with continuous sample paths
    such that for all 
      $x\in\R^d$,
      $t\in[0,T]$,
      $\omega\in\Omega$
    it holds that
    \beq\label{eq:SDEintroduction}
    X^x(t,\omega)=x+\int_0^t\mu(X^x(s,\omega))\,\diff s+\sigma W(t,\omega)
    \eeq
    and
  \item \label{enum:thmintroduction:2}
    it holds for all 
      $R,q\in[0,\infty)$
      that there exists 
        $c\in(0,\infty)$ 
      such that for all 
        $x,y\in\{v\in \R^d\colon \norm v\leq R\}$ with $0<\norm{x-y}\neq 1$
      it holds that
      \beq
      \label{eq:mainresultintroduction}
        \sup_{t\in[0,T]}\bEE{\norm{X^x(t)-X^y(t)}}\leq c\,\babs{\ln(\norm{x-y})}^{-q}
        .
      \eeq
  \end{enumerate}
\end{theorem}

Theorem~\ref{thm:introduction} above is an immediate consequence of 
Theorem~\ref{thm:finalthm} in Subsection~\ref{subsec:subhoelder} below. 
Inequality~\eqref{eq:mainresultintroduction}  proves, roughly speaking, 
only H\"older continuity in the initial value in a logarithmic sense 
but does neither prove local Lipschitz continuity nor prove 
local H\"older continuity in the initial value in the
usual sense. In view of this, the question arises 
whether the statement of Theorem~\ref{thm:introduction} can be 
strengthened to ensure local H\"older continuity in the initial value in 
the usual sense. In a subsequent article we show that this is not the case 
and specify a concrete additive noise driven SDE which satisfies 
the hypotheses of Theorem~\ref{thm:introduction} but whose 
solution  fails for every arbitrarily small $\alpha\in (0,1]$ to be 
locally $\alpha$-H\"older continuous in the initial value.
Even more, we show that under the hypotheses of 
Theorem~\ref{thm:introduction} the upper bound 
in~\eqref{eq:mainresultintroduction} can not be substantially 
improved in general.

In the following we briefly sketch the key ideas of our proof of 
inequality~\eqref{eq:mainresultintroduction} in Theorem~\ref{thm:introduction}. 
A straightforward approach to estimating the expectation of the  
Euclidean distance between two solutions of the SDE~\eqref {eq:SDEintroduction} 
with different initial values (cf.\ the left hand side 
of~\eqref{eq:mainresultintroduction}) would be (i) to apply the fundamental 
theorem of calculus 
to the difference of the two solutions with the derivative 
being  taken with respect to the initial value, thereafter, 
(ii) to employ the triangle inequality to get the Euclidean norm 
inside of the Riemann integral which has appeared due to the application 
of the fundamental theorem of calculus, and, finally, 
(iii) to try to provide a finite upper bound for the expectation of the 
Euclidean operator norm of the derivative processes 
of solutions of~\eqref {eq:SDEintroduction} with respect to the initial value. 
This approach, however, fails to work in general under  the hypotheses 
of Theorem~\ref{thm:introduction}  as the derivative processes of solutions  
may have very poor integrability properties and, in particular, 
may have infinite absolute moments. 
A key idea in this article for overcoming the latter obstacle is to 
estimate the expectation of the Euclidean distance  between the  two solutions 
in terms of  the expectation of a new distance  between the  two solutions, 
which is  induced from a very slowly growing norm-type function.
As in the approach above, we then also apply the fundamental theorem of calculus 
to the difference of the two solutions. However, in the latter approach 
the derivative processes of solutions appear only inside of the argument 
of the very slowly growing norm-type function and the expectation of the 
resulting random variable is finite. We then estimate  the expectation of 
this random variable by employing properties of the derivative  processes 
of solutions  and the assumption that the first order partial derivatives 
of the   drift coefficient function grow at most polynomially and, thereby, 
finally establish
 inequality~\eqref{eq:mainresultintroduction}.

The remainder of this article is organized as follows. 
In Section~\ref{section2} we 
establish  an essentially well-known existence and 
uniqueness result for  perturbed ordinary differential
equations.  In Section~\ref{section3}
 we recall well-known  facts on 
measurability properties of function limits and  
in Section~\ref{section4} we establish  a well-known measurability result for 
solutions of additive noise driven SDEs.    
% with a locally Lipschitz continuous drift coefficient function are stochastic processes.  
  In Section~\ref{section5} we prove   existence, 
uniqueness, and pathwise differentiability with respect to the 
initial value  and   in Section~\ref{section6} we present a few elementary 
integrability properties for solutions of additive noise driven SDEs 
with a drift coefficient function which admits a  Lyapunov-type condition. 
In Section~\ref{section7} we establish
an abstract regularity result for solutions of 
certain additive noise driven SDEs with respect to their initial values.  
This result together with the results of 
Sections \ref{section5} and \ref{section6} is  then used to prove 
the main result of this article, Theorem~\ref{thm:finalthm},  
in Section~\ref{section8}. 

\section{Existence of solutions of perturbed ordinary differential equations (ODEs)}\label{section2}

In this section we employ suitable Lyapunov-type functions to
establish in Lemma~\ref{lem:inteq} in Subsection~\ref{subsec:expodes} 
below an essentially well-known existence and 
uniqueness result for a certain class of perturbed ordinary differential
equations (ODEs).
Our proof of Lemma~\ref{lem:inteq} employs the essentially well-known a priori 
estimate in Lemma~\ref{lem:bndy} in Subsection~\ref{subsec:aprioripodes} below. 
Our proof of Lemma~\ref{lem:bndy} uses a suitable Lyapunov-type function 
(denoted by $V \colon \R^d \to \R $ in Lemma~\ref{lem:bndy} below).

\subsection{A priori estimates for solutions of perturbed ODEs}
\label{subsec:aprioripodes}
\begin{lemma}
  \label{lem:bndy}
  Let $d,m\in\N$,
  $T\in[0,\infty)$, 
  $\xi\in\R^d$, 
  $\mu\in C(\R^d,\R^d)$,
  $\sigma\in\R^{d\times m}$,
  $\varphi\in C(\R^m,[0,\infty))$,
  $V\in C^1(\R^d,[0,\infty))$,
  let $\norm{\cdot}\colon\R^d\to[0,\infty)$ be a norm,
  let $J\subseteq[0,T]$ be an interval, 
  assume that $0\in J$,
  and let $y\in C(J,\R^d)$,
  $w\in C([0,T],\R^m)$
  satisfy for all 
    $x\in\R^d$, 
    $u\in\R^m$,
    $t\in J$
  that
  $
  V'(x)\mu(x+\sigma u)\leq \varphi(u)V(x)
  $,
    $\norm x\leq V(x)$,
  and
  \beq
  \label{eq:assy1}
  y(t)=\xi+\int_0^t \mu(y(s))\,\diff s+\sigma w(t).
  \eeq
  Then it holds that
    $\sup_{t\in J}\bbr{\varphi(w(t))+\norm{\sigma w(t)}}<\infty$
  and
  \beq
  \label{eq:bndy}
  \sup_{t\in J}\,\norm{y(t)}\leq V(\xi)\exp\bbbp{ T\bbbbr{\sup_{s\in J} \varphi(w(s))} }+\bbbbr{\sup_{t\in J}\,\norm{\sigma w(t)}}.
  \eeq
\end{lemma}
\begin{proof}[Proof of Lemma~\ref{lem:bndy}]
  Throughout this proof
    assume w.l.o.g.\ that $\sup J>0$,
  let $I\subseteq[0,T]$ be the set which satisfies $I=(0,\sup J)$,
  let $K\in[0,\infty]$ satisfy
  \beq
    K=\sup_{s\in J}\varphi(w(s)),
  \eeq
  and let $z\colon J\to \R^d$ 
  \intrtype{be the function which satisfies }%
  \intrtypen{satisfy }%
  for all 
    $t\in J$ 
  that
  \beq
  \label{eq:defytilde}
    z(t)=y(t)-\sigma w(t)
    .
  \eeq
  Observe that
    the fact that $\varphi$ and $w$ are continuous functions
  ensures that
  \beq
    \label{eq:Kfin2}
    K
    \leq 
    \sup_{s\in[0,T]}\varphi(w(s))<\infty
    .
  \eeq
    This
    and the hypothesis that $w%\colon[0,T]\to\R^m
      $ is a continuous function
  ensure that
  \beq
    \sup_{t\in J}\bbr{\varphi(w(t))+\norm{\sigma w(t)}}
    \leq
    \bbbbr{\sup_{t\in J}\varphi(w(t))}+\bbbbr{\sup_{t\in J}\,\norm{\sigma w(t)}}
    \leq K+\sup_{t\in [0,T]}\,\norm{\sigma w(t)}
    <
    \infty
    .
  \eeq
  Next note that
    \eqref{eq:assy1}
    and \eqref{eq:defytilde}
  imply that for all
    $t\in J$
  it holds that
  \beq
  \label{eq:zint}
    z(t)=\xi+\int_0^t \mu(y(s))\,\diff s
    .
  \eeq
    The hypothesis that $\mu$ and $y$ are continuous functions
    and the fundamental theorem of calculus
    hence 
  ensure that for all
    $t\in I$ 
  it holds that
    $z|_I\in C^1(I,\R^d)$ 
  and
  \beq
    (z|_I)'(t)=\mu(y(t))
    .
  \eeq
    This,
    the assumption that $V\in C^1(\R^d,[0,\infty))$,
    and the chain rule
  imply that for all
    $t\in I$
  it holds that
    $V\circ (z|_I)\in C^1(I,[0,\infty))$
  and
  \beq
  \label{eq:Vtydiff}
    (V\circ (z|_I))'(t)
    =
    V'(z(t))(z|_I)'(t)
    =
    V'(z(t))\mu(y(t))
    .
  \eeq
  Furthermore, note that
    the hypothesis that $V\in C^1(\R^d,[0,\infty))$
    and the hypothesis that $y$, $w$, and $\mu$ are continuous functions
  establish that
    $J\ni t\mapsto V'(z(t))\mu(y(t))\in\R$
  is a continuous function.
  Combining
    this
    and \eqref{eq:Vtydiff}
  with
    the fundamental theorem of calculus
    and the fact that $z(0)=\xi$
  shows that for all
    $t\in I$ 
  it holds that
  \ba
    V(z(t))&=\bbr{ V(z(s)) }_{s=0}^{s=t} + V(z(0))\\
    %&=\int_0^t (V\circ \tilde y)'(s)\,\diff s + V(\xi)\\
    %&=\int_0^t V'(\tilde y(s))(\tilde y'(s))\,\diff s + V(\xi)\\
    &=\int_0^t V'(z(s))\mu(y(s))\,\diff s + V(\xi)\\
    &=\int_0^t V'(z(s))\mu(z(s)+\sigma w(s))\,\diff s + V(\xi)
    .
  \ea
    The hypothesis that 
      for all 
        $x\in\R^d$, 
        $u\in\R^m$ 
      it holds that
        $V'(x)\mu(x+\sigma u)\leq \varphi(u)V(x)$
    and \eqref{eq:Kfin2}
    hence
  prove that for all
    $t\in I$
  it holds that
  \beq
    \label{eq:Vgw}
    V(z(t))
    \leq\int_0^t \varphi(w(s))V(z(s))\,\diff s + V(\xi)
    \leq \int_0^t KV(z(s))\,\diff s + V(\xi)
    .
  \eeq
%   In addition, note that
%   \beq
%   \label{eq:Vty0}
%     V(z(0))=V(\xi)
%     .
%   \eeq
  %
    The assumption that $\sup J>0$
    and the fact that $J\in t\mapsto V(z(t))\in[0,\infty)$ is a continuous function
    therefore
  imply that for all
    $u\in\{s\in J\colon s=\sup J\}$
  it holds that
  \ba
    V(z(u))
    &=
    \limsup_{t\nearrow u}V(z(t))
    \\&\leq
    \limsup_{t\nearrow u}\bbbbr{\int_0^t KV(z(s))\,\diff s + V(\xi)}
    \\&=
    \int_0^u KV(z(s))\,\diff s + V(\xi)
    .
  \ea
    This,
    \eqref{eq:Vgw},
    and the fact that $V(z(0))=V(\xi)$
  demonstrate that for all
    $t\in J$
  it holds that
  \beq
    \label{eq:VgwJ}
    V(z(t))
    \leq 
    \int_0^t KV(z(s))\,\diff s + V(\xi)
    .
  \eeq
  Combining
    this
    and \eqref{eq:Kfin2}
  with
    Gronwall's integral inequality
      (see, e.g., Grohs et al.~\cite[Lemma~2.11]{GHJW}
        (with
          $\alpha\is V(\xi)$,
          $\beta\is K$,
          $T\is t$,
          $f\is ([0,t]\ni s\mapsto V(z(s))\in[0,\infty))$
        for
        $t\in J$
        in the notation of Grohs et al.~\cite[Lemma~2.11]{GHJW})%
      )
  proves that for all 
    $t\in J$ 
  it holds that
  \beq
    V(z(t))\leq V(\xi)\exp(tK)\leq V(\xi)\exp(TK)
    .
  \eeq
%   Combining
%     this
%   with
%     the fact that $V(z(0))=V(\xi)$
%   shows that for all 
%     $t\in J$ 
%   it holds that
%   \beq
%     V(z(t))\leq V(\xi)\exp(TK)
%     .
%   \eeq
  %
    The triangle inequality
    and the hypothesis that for all $x\in\R^d$ it holds that $\norm{x}\leq V(x)$
    hence 
  establish that
  \ba
  \label{eq:supy}
    \sup_{t\in J}\,\norm{y(t)}
    &= \sup_{t\in J}\,\norm{z(t)-\sigma w(t)}\\
    &\leq \sup_{t\in J} \bbr{\norm{z(t)}+\norm{\sigma w(t)}}\\
    &\leq \bbbbr{\sup_{t\in J}\, \norm{z(t)}}+\bbbbr{\sup_{t\in J}\,\norm{\sigma w(t)}}\\
    &\leq \bbbbr{\sup_{t\in J} V(z(t))}+\bbbbr{\sup_{t\in J}\,\norm{\sigma w(t)}}\\
    &\leq V(\xi)\exp(TK)+\bbbbr{\sup_{t\in J}\,\norm{\sigma w(t)}}
    .
  \ea
%   Combining 
%     this 
%   with
%     \eqref{eq:Kfin2}
%     and \eqref{eq:supy}
%   proves that
%   \beq
%     \sup_{t\in J}\,\norm{y(t)}\leq V(\xi)\exp(TK)+\bbbbr{\sup_{t\in J}\,\norm{\sigma w(t)}}<\infty
%     .
%   \eeq
%   This establishes \eqref{eq:bndy}. 
  The proof of Lemma~\ref{lem:bndy} is thus completed.
\end{proof}

\subsection{Existence of solutions of perturbed ODEs}
\label{subsec:expodes}

\begin{lemma}
  \label{lem:inteq}
  Let $d,m\in\N$,
  $T\in[0,\infty)$, 
  $\xi\in\R^d$, 
  $\sigma\in\R^{d\times m}$,
  $\varphi\in C(\R^m,[0,\infty))$,
  $V\in C^1(\R^d,[0,\infty))$,
  $w\in C([0,T],\R^m)$,
  let $\norm{\cdot}\colon\R^d\to[0,\infty)$ be a norm,
  let $\mu\colon\R^d\to\R^d$ be a locally Lipschitz continuous function,
  and assume for all $x\in\R^d$, $z\in\R^m$ that
  $
  V'(x)\mu(x+\sigma z)\leq \varphi(z)V(x)
  $
  and $\norm x\leq V(x)$%
  .
  Then there exists a unique $y\in C([0,T],\R^d)$
  such that for all $t\in[0,T]$ it holds that
  \beq
  \label{eq:inteqconc}
  y(t)=\xi+\int_0^t \mu(y(s))\,\diff s+\sigma w(t).
  \eeq
\end{lemma}
\begin{proof}[Proof of Lemma~\ref{lem:inteq}]
  Throughout this proof 
    assume w.l.o.g.\ that $T>0$.
%   Note that 
%     the hypothesis that $\mu\in C^1(\R^d,\R^d)$
%     and the fact that all norms on $\R^d$ are equivalent
%   ensure that for all
%     $r\in(0,\infty)$
%   it holds that
%   \beq
%   \label{eq:muloclip}
%     \sup_{\substack{x,y\in\R^d,\,x\neq y,\\\norm x+\norm y\leq r}}\frac{\norm{\mu(x)-\mu(y)}}{\norm{x-y}}<\infty.
%   \eeq
  Note that 
    the hypothesis that $\mu$ is a locally Lipschitz continuous function,
    the hypothesis that $w$ is a continuous function,
    and~\cite[Theorem~8.3]{JentzenMazzonettoSalimova}
      (with
        $(V,\norm{\cdot}_V)\is (\R^d,\norm{\cdot})$,
        $(W,\norm{\cdot}_W)\is (\R^d,\norm{\cdot})$,
        $T\is T$,
        $F\is \mu$,
        $S\is \pp{(0,T)\ni t\mapsto \id_{\R^d}\in L(\R^d,\R^d)}$,
        $\mc S\is \pp{[0,T]\ni t\mapsto \id_{\R^d}\in L(\R^d,\R^d)}$,
        $o\is \pp{[0,T]\ni t\mapsto \xi+\sigma w(t)\in\R^d}$,
        $\phi\is \pp{(0,T)\ni t\mapsto t\in(0,\infty)}$
      in the notation of~\cite[Theorem~8.3]{JentzenMazzonettoSalimova})
  ensure that
    there exists an interval $J\subseteq[0,T]$ 
    with 
      $0\in J$
      and $\sup J>0$
      such that 
        there exists a unique $x\in C(J,\R^d)$ which satisfies
          for all 
            $t\in J$
          that
          \beq
          \label{eq:defy}
          %\int_0^t\norm{\mu(x(s))}\,\diff s<\infty,\qquad 
            x(t)=\xi+\int_0^t \mu(x(s))\,\diff s+\sigma w(t)
            \qqandqq 
            \limsup_{s\nearrow\sup J} \bbr{ (T-s)^{-1} + \norm{x(s)} }=\infty
            .
          \eeq
    Lemma~\ref{lem:bndy}
    hence 
  proves that $\sup_{t\in J}[\varphi(w(s))+\norm{\sigma w(t)}]<\infty$ and
  \beq
  \label{eq:supxfin}
    \sup_{t\in J}\,\norm{x(t)}\leq V(\xi)\exp\bbbp{ T\bbbbr{\sup_{s\in J} \varphi(w(s))} }+\bbbbr{\sup_{t\in J}\,\norm{\sigma w(t)}}<\infty.
  \eeq
  Combining 
    this
  with 
    \eqref{eq:defy} 
  ensures that
    $\sup J=T$.
  Therefore, we obtain that 
    $J=[0,T)$ or $J=[0,T]$.
    This,
    the hypothesis that $\mu$ is a locally Lipschitz continuous function,
    \eqref{eq:supxfin},
    and~\cite[Lemma~8.1]{JentzenMazzonettoSalimova}
      (with
        $(V,\norm{\cdot}_V)\is (\R^d,\norm{\cdot})$,
        $(W,\norm{\cdot}_W)\is (\R^d,\norm{\cdot})$,
        $T\is T$,
        $\tau\is T$,
        $x\is x|_{[0,T)}$,
        $o\is \pp{[0,T]\ni t\mapsto \xi+\sigma w(t)\in\R^d}$,
        $F\is\mu$,
        $S\is \pp{(0,T)\ni t\mapsto \id_{\R^d}\in L(\R^d)}$,
        $\phi\is \pp{(0,T)\ni t\mapsto t\in(0,\infty)}$
      in the notation of~\cite[Lemma~8.1]{JentzenMazzonettoSalimova})
  prove that 
    there exists a continuous function $y\colon[0,T]\to\R^d$
      such that it holds for all 
        $t\in[0,T]$ 
      that
      \beq
      \label{eq:yex}
        y|_{[0,T)}=x\qqandqq y(t)=\xi+\int_0^t\mu(y(s))\,\diff s+\sigma w(t).
      \eeq
  In the next step we observe that 
    \eqref{eq:defy}
    and the fact that $\sup J=T$
  show that for all 
    $z\in\bigl\{u\in C([0,T],\R^d)\colon \bp{\forall\, t\in[0,T]\colon u(t)=\xi+\int_0^t\mu(u(s))\,\diff s+\sigma w(t)}\bigr\}$
  it holds that 
    $z|_J=x$.
    The fact that $y$ is a continuous function 
    and the fact that $y|_{[0,T)}=x$ 
    hence 
  demonstrate that for all
    $z\in\bigl\{u\in C([0,T],\R^d)\colon \bp{\forall\, t\in[0,T]\colon u(t)=\xi+\int_0^t\mu(u(s))\,\diff s+\sigma w(t)}\bigr\}$
  it holds that 
    $z=y$.
  Combining
    this
  with
    \eqref{eq:yex}
  establishes 
    \eqref{eq:inteqconc}. 
  The proof of Lemma~\ref{lem:inteq} is thus completed.
\end{proof}

\section{Measurability properties}\label{section3}

In this section we recall in Lemmas~\ref{lem:sup_measurability}--\ref{lem:diffmeas} 
in Subsection~\ref{subsec:measfunc} and in
Lemmas~\ref{lem:supmeas1} and \ref{lem:supmeas} in Subsection~\ref{subsec:measproc}
below a few well-known facts on 
measurability properties of suitable function limits. 
For completeness we also include in this section proofs 
for Lemmas~\ref{lem:sup_measurability}--\ref{lem:supmeas}.

\subsection{Measurability properties for functions}
\label{subsec:measfunc}

\begin{lemma}
\label{lem:sup_measurability}
Let $ ( \Omega, \mathcal{F} ) $
be a measurable space,
let $I$ be a non-empty and at most countable set,
let $ Y \colon \Omega \to \extR $
be a function, and
let
$ X_i \colon \Omega \to \extR $,
$ i \in I $,
be
$ \mathcal{F} 
$/$ \mathcal{B}( \extR ) 
$-measurable
functions which satisfy for all
$ \omega \in \Omega $
that
\begin{equation}
  Y( \omega ) =
  \sup_{ i \in I } X_i( \omega )
  .
\end{equation}
Then it holds that 
$ Y $ is an
$ \mathcal{F} $/$ \mathcal{B}( \extR ) 
$-measurable function.
\end{lemma}

\begin{proof}[Proof
of Lemma~\ref{lem:sup_measurability}]
Note that 
  the hypothesis that 
    for all 
      $i\in I$
    it holds that
      $X_i$ is an $\mc F$/$\mc B(\extR)$-measurable function
  and the hypothesis that $I$ is at most countable
establish that
for all $ c \in \R $
it holds that
\begin{equation}
  \left\{ 
    Y \leq c
  \right\}
=
  \left\{ 
    \sup_{ i \in I } X_i \leq c
  \right\}
=
  \bigcap_{ i \in I }
  \underbrace{
  \left\{ 
    X_i \leq c
  \right\}
  }_{ \in \mathcal{F} }
\in 
  \mathcal{F}
  .
\end{equation}
The proof of
Lemma~\ref{lem:sup_measurability}
is thus completed.
\end{proof}

\begin{lemma}
\label{lem:convergence_measurability}
Let $ ( \Omega, \mathcal{F} ) $
be a measurable space,
let $ Y \colon \Omega \to \R $
be a function, and let
$ X_n \colon \Omega \to \R $,
$ n \in \N $,
be a sequence of 
$ \mathcal{F} 
$/$ \mathcal{B}( \R ) $-measurable
functions which satisfies 
for all
$ \omega \in \Omega $ 
that
$
  \limsup_{ n \to \infty } 
  | 
    X_n( \omega )
  - 
    Y( \omega )
  |
  = 0
$.
Then it holds that 
$ Y $ is an
$ \mathcal{F} $/$ \mathcal{B}( \R ) 
$-measurable function.
\end{lemma}

\begin{proof}[Proof
of Lemma~\ref{lem:convergence_measurability}]
First, observe that
  the assumption that
    for all
      $\omega\in\Omega$
    it holds that
      $\limsup_{n\to\infty}\abs{X_n(\omega)-Y(\omega)}=0$
implies that
  for all 
    $\omega\in\Omega$ 
  it holds that
    $\N\ni n\mapsto X_n(\omega)\in\R$ is a convergent sequence
    and
    \beq
    \label{eq:lim}
    \lim_{n\to\infty} X_n(\omega)=Y(\omega)
    .
    \eeq
Moreover, note that 
  Lemma~\ref{lem:sup_measurability}
ensures that
for all $n\in\N$ it holds that
$\Omega\ni\omega\mapsto \sup_{m\in\{n,n+1,\dots\}}X_m(\omega)\in\extR$
is an $\mc F$/$\mc B(\extR)$-measurable function.
  This
  and \eqref{eq:lim}
  show that
for all $ c \in \R $
it holds that
\begin{equation}
\begin{split}
&
  \left\{
    Y \geq c
  \right\}
=
  \left\{
    \lim_{ n \to \infty } X_n \geq c
  \right\}
=
  \biggl\{
    \limsup_{ n \to \infty } X_n \geq c
  \biggr\}
\\ & =
  \biggl\{
    \lim_{ n \to \infty }
    \bbbbr{
      \sup_{ m \in \{ n, n + 1 , \dots \} }
      X_m 
    }
    \geq c
  \biggr\}
=
  \bigcap_{ n \in \N }
  \underbrace{
  \biggl\{
    \bbbbr{
      \sup_{ m \in \{ n, n + 1 , \dots \} }
      X_m 
    }
    \geq c
  \biggr\}
  }_{
    \in \mathcal{F}
  }
  \in \mathcal{F}
  .
\end{split}
\end{equation}
The proof of Lemma~\ref{lem:convergence_measurability}
is thus completed.
\end{proof}

\begin{lemma}
\label{lem:convergence_measurability_d}
Let $ ( \Omega, \mathcal{F} ) $
be a measurable space,
let $d\in\N$,
let $\norm{\cdot}\colon\R^d\to [0,\infty)$ be a norm,
let $ Y \colon \Omega \to \R^d $
be a function, and let
$ X_n \colon \Omega \to \R^d $,
$ n \in \N $,
be a sequence of 
$ \mathcal{F} 
$/$ \mathcal{B}( \R^d ) $-measurable
functions which satisfies 
for all
$ \omega \in \Omega $ 
that
$
  \limsup_{ n \to \infty } 
  \norm{
    X_n( \omega )
  - 
    Y( \omega )
  }
  = 0
$.
Then it holds that 
$ Y $ is an
$ \mathcal{F} $/$ \mathcal{B}( \R^d ) 
$-measurable function.
\end{lemma}
\begin{proof}[Proof of Lemma~\ref{lem:convergence_measurability_d}]
  Throughout this proof
  let $K\in[0,\infty]$ satisfy
    \beq
    \label{eq:normq}
    K
    =
    \sup_{v=(v_1,v_2,\dots,v_d)\in\R^d\setminus\{0\}}
    \bbbbp{\frac{\bp{\sum_{j=1}^d\,\abs{v_j}}}{\norm{v}}},
    \eeq
  let $X_{n,i}\colon\Omega\to\R$, $n\in\N$, $i\in\{1,2,\dots,d\}$, 
    \intrtype{be the functions which }%
    satisfy
    for all 
      $n\in\N$,
      $\omega\in\Omega$
    that
    \beq
      X_n(\omega)=(X_{n,1}(\omega),X_{n,2}(\omega),\dots,X_{n,d}(\omega)),
    \eeq
  and let $Y_i\colon \Omega\to\R$, $i\in\{1,2,\dots,d\}$, 
    \intrtype{be the functions which }%
    satisfy
    for all 
      $\omega\in\Omega$
    that
    \beq
    Y(\omega)=(Y_1(\omega),Y_2(\omega),\dots,Y_d(\omega))
    .
    \eeq
  Observe that
    the fact that all norms on $\R^d$ are equivalent
  ensures that
    $K<\infty$.
    This 
    %and \eqref{eq:normq} 
  implies that for all
    $n\in\N$,
    $i\in\{1,2,\dots,d\}$, 
    $\omega\in\Omega$ 
  it holds that
  \beq
  \label{eq:compbnd}
    \abs{X_{n,i}(\omega)-Y_i(\omega)}
    \leq 
    \sum_{j=1}^d\,\abs{X_{n,j}(\omega)-Y_j(\omega)}
    \ann{\leq}{\eqref{eq:normq}}
    K\,\norm{X_n(\omega)-Y(\omega)}.
  \eeq
    The assumption that 
      for all 
        $\omega\in\Omega$ 
      it holds that
        $\limsup_{ n \to \infty } \norm{X_n(\omega)-Y(\omega)}=0$
    and the fact that $K<\infty$
    hence
  show that for all 
    $i\in\{1,2,\dots,d\}$, 
    $\omega\in\Omega$
  it holds that
  \beq
  \label{eq:compconv}
    \limsup_{ n \to \infty }\,\abs{X_{n,i}(\omega)-Y_i(\omega)}=0.
  \eeq
  Furthermore, observe that 
    the assumption that
      for all
        $n\in\N$
      it holds that
        $X_n$ is an $\mc F$/$\mc B(\R^d)$-measurable function
  implies that for all 
    $n\in\N$, 
    $i\in\{1,2,\dots,d\}$ 
  it holds that
    $X_{n,i}%\colon \Omega\to\R
    $ is an $\mathcal F$/$\mathcal B(\R)$-measurable function.
  Combining 
    this 
    and \eqref{eq:compconv} 
  with
    Lemma~\ref{lem:convergence_measurability}
  establishes that for all 
    $i\in\{1,2,\dots,d\}$
  it holds that 
    $Y_i%\colon \Omega\to\R
    $ is an $\mc F$/$\mc B(\R)$-measurable function.
    The fact that $\mc B(\R^d)=[\mc B(\R)]^{\otimes d}$
    hence
  shows that 
    $Y%\colon \Omega\to\R^d
    $ is an $\mc F$/$\mc B(\R^d)$-measurable function.
  The proof of Lemma~\ref{lem:convergence_measurability_d}
  is thus completed.
\end{proof}

\begin{lemma}
\label{lem:diffmeas}
  Let $d\in\N$,
  $x,h\in\R^d$,
  let $(\Omega,\mc F)$ be a measurable space,
  and let $y^z\colon \Omega\to\R^d$, $z\in\R^d$,
    be $\mc F$/$\mc B(\R^d)$-measurable functions
    which satisfy for all 
      $\omega\in\Omega$ 
    that 
      $(\R^d\ni z\mapsto y^z(\omega)\in\R^d)\in C^1(\R^d,\R^d)$.
  Then it holds that
    $\Omega\ni \omega\mapsto \bp{\tfrac\partial{\partial x} y^x(\omega)}(h)\in\R^d$
    is an $\mc F$/$\mc B(\R^d)$-measurable function.
\end{lemma}
\begin{proof}[Proof of Lemma~\ref{lem:diffmeas}]
  Throughout this proof let 
    $D_n\colon \Omega\to \R^d$, 
      $n\in\N$,
      \intrtype{be the sequence of functions which satisfies }%
      \intrtypen{satisfy }%
      for all
        $n\in\N$,
        $\omega\in\Omega$
      that
      \beq
        D_{n}(\omega)=\frac{y^{x+n^{-1}h}(\omega)-y^x(\omega)}{n^{-1}}
      \eeq
    and let $\norm{\cdot}\colon\R^d\to[0,\infty)$ be the $d$-dimensional Euclidean norm.
  Note that for all
    $\omega\in\Omega$
  it holds that
  \beq
  \label{eq:difflim}
    \limsup_{n\to\infty}\,\bnorm{D_{n}(\omega)-\bp{\tfrac\partial{\partial x} y^x(\omega)}(h)}=0
    .
  \eeq
  Furthermore, observe that
    the assumption that
      for all
        $z\in\R^d$
      it holds that
        $y^z$ is an $\mc F$/$\mc B(\R^d)$-measurable function
  ensures that for all
    $n\in\N$
  it holds that
    $D_n$ is an $\mc F$/$\mc B(\R^d)$-measurable function.
  Combining
    this
    and \eqref{eq:difflim}
  with
    Lemma~\ref{lem:convergence_measurability_d} 
      (with
        $(\Omega,\mc F)\is (\Omega,\mc F)$,
        $d\is d$,
        $\norm\cdot\is\norm\cdot$,
        $Y\is \bp{\Omega\ni \omega\mapsto\bp{\frac\partial{\partial x}y^x(\omega)}(h)\in\R^d}$,
        $(X_n)_{n\in\N}\is (D_n)_{n\in\N}$
      in the notation of Lemma~\ref{lem:convergence_measurability_d})
  implies that
    $\Omega\ni \omega\mapsto \bp{\tfrac\partial{\partial x} y^x(\omega)}(h)\in\R^d$ 
      is an $\mc F$/$\mc B(\R^d)$-mea\-sur\-able function.
  The proof of Lemma~\ref{lem:diffmeas} is thus completed.
\end{proof}

\subsection{Measurability properties for stochastic processes}
\label{subsec:measproc}
 
\begin{lemma}
\label{lem:supmeas1}
  Let $T\in[0,\infty)$,
  let $(\Omega,\mc F,\PP)$ be a probability space,
  and let $Y\colon [0,T]\times\Omega\to \R$ be a stochastic process
    with continuous sample paths.
  Then 
  \begin{enumerate}[label=(\roman{enumi})]
    \item 
      it holds for all 
        $\omega\in\Omega$ 
      that
      $
        \sup_{t\in[0,T]} Y(t,\omega)=\sup_{t\in[0,T]\cap\Q} Y(t,\omega)
      $
      and
    \item
      it holds that
        $\Omega\ni\omega\mapsto \sup_{t\in[0,T]} Y(t,\omega)\in\R$
        is an $\mc F$/$\mc B(\R)$-measurable function.
  \end{enumerate}
\end{lemma}
\begin{proof}[Proof of Lemma~\ref{lem:supmeas1}]
  Observe that 
    the hypothesis that
      for all
        $\omega\in\Omega$
      it holds that
        $[0,T]\ni t\mapsto Y(t,\omega)\in\R$ is a continuous function
    and the fact that $[0,T]\cap\Q$ is dense in $[0,T]$
  imply that for all 
    $\omega\in\Omega$
  it holds that
  \beq
%     \sup_{t\in[0,T]}Y(t,\omega)
%     <\infty
%     \qqandqq
    \sup_{t\in[0,T]} Y(t,\omega)
    =
    \sup_{t\in[0,T]\cap\Q} Y(t,\omega)
    .
  \eeq
  Combining this with Lemma~\ref{lem:sup_measurability}
    (with
      $(\Omega,\mc F)\is(\Omega,\mc F)$,
      $I\is[0,T]\cap\Q$,
      $Y\is\bp{\Omega\ni\omega\mapsto\sup_{t\in[0,T]} Y(t,\omega)\in\R}$,
      $(X_t)_{t\in[0,T]\cap\Q}\is (\Omega\ni\omega\mapsto Y(t,\omega)\in\R)_{t\in[0,T]\cap\Q}$
    in the notation of Lemma~\ref{lem:sup_measurability}) 
  shows that $\Omega\ni\omega\mapsto\sup_{t\in[0,T]} Y(t,\omega)\in\R$
    is an $\mc F$/$\mc B(\R)$-measurable function.
  The proof of Lemma~\ref{lem:supmeas1} is thus completed.
\end{proof}

\begin{lemma}
\label{lem:supmeas}
  Let $d\in\N$,
  $T,R\in[0,\infty)$,
  let $\norm\cdot\colon\R^d\to[0,\infty)$ be a norm,
  let $(\Omega,\mc F,\PP)$ be a probability space,
  and let $Y^x\colon[0,T]\times \Omega\to[0,\infty)$, $x\in\R^d$,
    be stochastic processes with continuous sample paths
    which satisfy 
      for all 
        $t\in[0,T]$,
        $\omega\in\Omega$
      that
        $(\R^d\ni x\mapsto Y^x(t,\omega)\in[0,\infty))\in C(\R^d,[0,\infty))$.
        %and $\sup_{x\in\{z\in\R^d\colon \norm z\leq R\}}\sup_{s\in[0,T]}(\norm{X^x(s,\omega)}^r)<\infty$.
  Then
  \begin{enumerate}[label=(\roman{enumi}),ref=(\roman{enumi})]
    \item \label{item:lemsupmeas:1}
    it holds for all
      $\omega\in\Omega$
    that
    \beq
      \bbbbr{
      \sup_{x\in\{z\in\R^d\colon\norm{z}\leq R\}}\,
        \sup_{t\in[0,T]}
          Y^x(t,\omega)}
      =
      \bbbbr{
      \sup_{x\in\{z\in\Q^d\colon\norm{z}\leq R\}}\,
        \sup_{t\in[0,T]\cap\Q}
          Y^x(t,\omega)}
    \eeq
    and
    \item \label{item:lemsupmeas:2}
    it holds that
      \beq
      \Omega\ni\omega\mapsto 
        \bbbbr{\sup_{x\in\{z\in\R^d\colon\norm{z}\leq R\}}\,
          \sup_{t\in[0,T]}
            Y^x(t,\omega)}
      \in[0,\infty]
      \eeq
      is an $\mc F$/$\mc B([0,\infty])$-measurable function%
      .
    \em
  \end{enumerate}
\end{lemma}

\begin{proof}[Proof of Lemma~\ref{lem:supmeas}]
Throughout this proof 
  let $I\subseteq \Q^d\times([0,T]\cap\Q)$ 
    be the set which satisfies
    \beq
    I=\{(x,t)\in\Q^d\times([0,T]\cap\Q)\colon \norm x\leq R\}
    .
    \eeq
Observe that 
  Lemma~\ref{lem:supmeas1}
%   the assumption that for all
%     $x\in\R^d$,
%     $\omega\in\Omega$
%   it holds that 
%     $[0,T]\ni t\mapsto Y^x(t,\omega)\in[0,\infty)$ is a continuous function
%   and the fact that $[0,T]\cap \Q$ is dense in $[0,T]$ 
implies that for all 
  $x\in\R^d$,
  $\omega\in\Omega$
it holds that
\beq
  \sup_{t\in[0,T]}Y^x(t,\omega)
  =\sup_{t\in[0,T]\cap\Q}Y^x(t,\omega)
  .
\eeq
  \com{a} The assumption that 
    for all 
      $t\in[0,T]$,
      $\omega\in\Omega$
    it holds that
      $\R^d\ni x\mapsto Y^x(t,\omega)\in[0,\infty)$ is a continuous function
  and \com{b} the fact that
    $\{z\in\Q^d\colon \norm z\leq R\}$ is dense in $\{z\in\R^d\colon \norm z\leq R\}$ 
  \com{c} therefore 
show that for all 
  $\omega\in\Omega$
it holds that
\ba
\label{eq:supcomb}
  \bbbbr{\sup_{x\in\{z\in\R^d\colon\norm{z}\leq R\}}\,\sup_{t\in[0,T]}Y^x(t,\omega)}
  &\ann{=}{c}
  \bbbbr{\sup_{x\in\{z\in\R^d\colon\norm{z}\leq R\}}\,\sup_{t\in[0,T]\cap\Q}Y^x(t,\omega)}\\
  &=
  \bbbbr{\sup_{(x,t)\in\{(z,s)\in\R^d\times([0,T]\cap \Q)\colon\norm z\leq R\}}Y^x(t,\omega)}\\
  &=
  \bbbbr{\sup_{t\in[0,T]\cap\Q}\,\sup_{x\in\{z\in\R^d\colon\norm{z}\leq R\}}Y^x(t,\omega)}\\
  &\ann{=}{ab}
  \bbbbr{\sup_{t\in[0,T]\cap\Q}\,\sup_{x\in\{z\in\Q^d\colon\norm{z}\leq R\}}Y^x(t,\omega)}\\
  &=
  \bbbbr{\sup_{(x,t)\in\{(z,s)\in\Q^d\times([0,T]\cap \Q)\colon\norm z\leq R\}}Y^x(t,\omega)}
  .
\ea
Hence, we obtain that for all
  $\omega\in\Omega$ 
it holds that
\beq
  \bbbbr{
  \sup_{x\in\{z\in\R^d\colon\norm{z}\leq R\}}\,\sup_{t\in[0,T]}Y^x(t,\omega)
  }
  =
  \bbbbr{
  \sup_{x\in\{z\in\Q^d\colon\norm{z}\leq R\}}\,\sup_{t\in[0,T]\cap\Q}Y^x(t,\omega)
  }
  .
\eeq
This establishes \ref{item:lemsupmeas:1}.
In the next step we combine
  \eqref{eq:supcomb}
  and the fact that $I$ is an at most countable set
with Lemma~\ref{lem:sup_measurability}
  (with
    $(\Omega,\mc F)\is(\Omega,\mc F)$,
    $I\is I$,
    $Y\is \bp{\Omega\ni\omega\mapsto \sup_{x\in\{z\in\R^d\colon\norm{z}\leq R\}}\,\sup_{t\in[0,T]}Y^x(t,\omega)\in[0,\infty]}$,
    $(X_{(x,t)})_{(x,t)\in I}\is(\Omega\ni\omega\mapsto Y^x(t,\omega)\in[0,\infty))_{(x,t)\in I}$
  in the notation of Lemma~\ref{lem:sup_measurability})
to obtain~\ref{item:lemsupmeas:2}.
The proof of Lemma~\ref{lem:supmeas} is thus completed.
\end{proof}

\section{Measurability properties for solutions of SDEs}\label{section4}

In this section we establish in Lemma~\ref{lem:intmeas} in Subsection~\ref{subsec:meassdes}
below the well-known fact that pathwise 
solutions of certain additive noise driven SDEs are stochastic processes. 
Our proof of Lemma~\ref{lem:intmeas} exploits the fact that Euler approximations
converge pathwise to solutions of such SDEs (cf.~Lemmas~\ref{lem:approxa} and
\ref{lem:approxa2} in 
Subsection~\ref{subsec:euler} and Lemma~\ref{lem:interpol} in Subsection~\ref{subsec:conteuler} 
below) 
as well as the elementary fact that the Euler approximations, in turn, 
are indeed stochastic processes.
Our proof of the convergence statement for the Euler approximations in Lemma~\ref{lem:approxa} 
exploits the familiar time-discrete Gronwall inequality in Lemma~\ref{lem:discr_gronwall} in 
Subsection~\ref{subsec:euler} below. For completeness we also include here detailed proofs 
for Lemmas~\ref{lem:discr_gronwall}--%
%, \ref{lem:approxa}, \ref{lem:approxa2}, \ref{lem:interpol}, and 
\ref{lem:intmeas}.

\subsection{Time-discrete approximations for deterministic differential equations (DEs)}
\label{subsec:euler}

\begin{samepage}
\begin{lemma}
\label{lem:discr_gronwall}
Let $ N \in \N $,
$ \beta \in [0,\infty) $,
$ \alpha \in \R $, 
$ f_0, f_1, \dots, f_N \in \R \cup \{ \infty \} $ 
satisfy 
for all
$ n \in \left\{ 0, 1, \dots, N \right\} $
that
\begin{equation}
\label{gronwall2need}
  f_n
  \leq
  \alpha
  +
  \beta
  \left(
    \sum^{ n - 1 }_{ k = 0 }
    f_k
  \right)
  .
\end{equation}
Then it holds for all
$ n \in \left\{ 0, 1, \dots, N \right\} $
that
\begin{equation}
\label{eq:discr_gronwall_c}
  f_n
  \leq 
  \alpha 
  \left(
    1 + \beta 
  \right)^n
  \leq
  \left| \alpha \right|
  e^{
    \beta n
  }
  < \infty
  .
\end{equation}
\end{lemma}
\end{samepage}
\begin{proof}[Proof of
Lemma \ref{lem:discr_gronwall}]
Throughout this proof
  let
$ 
  u_0, u_1, \dots, u_N
  \in \R 
$
be the real numbers 
which satisfy
for all 
$ n \in \{ 0, 1, 2, \dots, N \} $
that
\begin{equation}
\label{eq:def_un}
  u_n 
=
  \alpha
  +
  \beta
  \left(
    \sum^{ n - 1 }_{ k = 0 }
    u_k
  \right)
  .
\end{equation}
% Observe that
% induction and \eqref{gronwall2need} prove that
% for all $ n \in \{ 0, 1, 2, \dots, N \} $
% it holds that
% $
%   f_n \in \R
% $.
Hence, we obtain that
for all
$ 
  n \in \{ 0, 1, \dots, N - 1 \}
$
it holds that
\begin{equation}
  u_{ n + 1 }
  =
  \alpha
  +
  \beta
  \left(
    \sum^n_{ k = 0 }
    u_k
  \right)
  =
  \underbrace{
  \alpha
  +
  \beta
  \left(
  \sum^{ n - 1 }_{ k = 0 }
  u_k
  \right)
  }_{
    = u_n
  }
  +
  \beta 
  u_n
  =
  \left( 1 + \beta \right) u_n
  .
\end{equation}
This implies that
for all $ n \in \{ 0, 1, \dots, N \} $
it holds that
\begin{equation}
\label{eq:un_satisfies}
  u_n
  =
  \alpha 
  \left( 1 + \beta \right)^{
    n
  }
  .
\end{equation}
Moreover, observe that induction shows that
for all
$ n \in \{ 0, 1, \dots, N \} $
it holds that
\begin{equation}
  f_n \leq u_n  
  .
\end{equation}
Combining this with \eqref{eq:un_satisfies}
establishes \eqref{eq:discr_gronwall_c}.
This
completes the proof
of Lemma~\ref{lem:discr_gronwall}.
\end{proof}

\begin{lemma}
  \label{lem:approxa}
  Let $d\in\N$, 
  $T\in[0,\infty)$,
  $f\in C([0,T]\times\R^d,\R^d)$, 
  let $\norm{\cdot}\colon\R^d\to[0,\infty)$ be a norm,
  assume for all $r\in(0,\infty)$ that
  \beq
    \sup_{t\in[0,T]}\,\sup_{\substack{x,y\in\R^d,\,x\neq y,\\ \norm x+\norm y\leq r}} \frac{\norm{f(t,x)-f(t,y)}}{\norm{x-y}}<\infty,
  \eeq
  let $Y\in C([0,T],\R^d)$ satisfy
  for all 
    $t\in[0,T]$ 
  that
  $
    Y(t)=Y(0)+\int_0^tf(s,Y(s))\,\diff s
  $,
  and let $\mc Y^N\colon\{0,1,\dots,N\}\to\R^d$, $N\in\N$, satisfy for all
    $N\in\N$, 
    $n\in\{0,1,\dots,N-1\}$ 
  that
  \beq
    \mc Y^N(0)=Y(0)
    \qqandqq
    \mc Y^N(n+1)=\mc Y^N(n)+ \tfrac TN f\bp{\tfrac{nT}{N},\mc Y^N(n)}
    .
  \eeq
  Then it holds that
  \beq
    \label{eq:approxaconc}
    \limsup_{N\to\infty}\bbbbr{\sup_{n\in\{0,1,\dots,N\}}\bnorm{\mc Y^N(n)-Y\bp{\tfrac{nT}N}}}
    =0
    .
  \eeq
\end{lemma}
\begin{proof}[Proof of Lemma~\ref{lem:approxa}]
  Throughout this proof 
    let $R\in[0,\infty)$ be the real number which satisfies
    \beq
      R=2\bbbbr{\sup_{t\in[0,T]}\norm{Y(t)}}+1,
    \eeq
    let $L\in[0,\infty)$ be the real number which satisfies
    \beq
      L
      =
      \bbbbr{\sup_{t\in[0,T]}\,\sup_{\substack{x,y\in\R^d,\,x\neq y,\\ \norm x+\norm y\leq R}} \frac{\norm{f(t,x)-f(t,y)}}{\norm{x-y}}}
      +
      \bbbbr{\sup_{s,t\in[0,T]} f(s,Y(t))}
      ,
    \eeq
    let $\tau_N\in\{0,1,\dots,N\}$, $N\in\N$, be the numbers which satisfy
      for all
        $N\in\N$
      that
      \beq
      \label{eq:approxatauN}
        \tau_N
        =
        \min\bp{
          \{N\}
          \cup
          \bigl\{n\in\{0,1,\dots,N\}\colon \bnorm{\mc Y^N(n)-Y\bp{\tfrac{nT}N}}>1\bigr\}
        }
        ,
      \eeq
%     let $\floor{\cdot}_N\colon[0,\infty)\to[0,\infty)$, $N\in\N$, 
%       be the functions which satisfy for all
%         $N\in\N$,
%         $s\in[0,\infty)$
%       that
%       \beq
%         \floor{s}_N
%         =
%         \max\bp{\bigl\{0,\tfrac 1N,\tfrac2N,\dots\bigr\}\cap[0,s]}
%         ,
%       \eeq
    and let $\alpha_{N}\in[0,\infty)$, $N\in\N$,
      be the real numbers which satisfy for all
        $N\in\N$
      that
      \beq
        \alpha_{N}
        =
        \sum_{k=0}^{N-1}\int_{\tfrac{kT}N}^{\tfrac{(k+1)T}N} \bnorm{f(s,Y(s))-f\bp{\tfrac{kT}N,Y(s)}}\,\diff s
%         \int_{0}^{T} \norm{f(s,Y(s))-f\bp{\floor{s}_N,Y(s)}}\,\diff s
        .
      \eeq
    Note that for all
      $N\in\N$,
      $n\in\{0,1,\dots,N\}$
    it holds that
    \ba
      &\bnorm{Y\bp{\tfrac{nT}N}-\mc Y^N(n)}
      \\&=
      \bbbnorm{
        \bbbbr{Y(0)+\int_0^{\tfrac{nT}N} f(s,Y(s))\,\diff s}
        -
        \bbbbr{Y(0)+\frac TN\sum_{k=0}^{n-1}f\bp{\tfrac{kT}N,\mc Y^N(k)}}
      }
      \\&=
      \bbbnorm{
        \int_0^{\tfrac{nT}N} f(s,Y(s))\,\diff s
        -
        \frac TN\sum_{k=0}^{n-1}f\bp{\tfrac{kT}N,\mc Y^N(k)}
      }
      .
    \ea
    Hence, we obtain that for all
      $N\in\N$,
      $n\in\{0,1,\dots,N\}$
    it holds that
    \ba
    \label{eq:approxa1}
      &\norm{Y\bp{\tfrac{nT}N}-\mc Y^N(n)}
      \\&=
      \bbbnorm{
        \sum_{k=0}^{n-1}\int_{\tfrac{kT}N}^{\tfrac{(k+1)T}N} f(s,Y(s))\,\diff s
        -
        \sum_{k=0}^{n-1}\int_{\tfrac{kT}N}^{\tfrac{(k+1)T}N}f\bp{\tfrac{kT}N,\mc Y^N(k)}\,\diff s
      }
      \\&=
      \bbbnorm{
        \sum_{k=0}^{n-1}\int_{\tfrac{kT}N}^{\tfrac{(k+1)T}N} f(s,Y(s))-f\bp{\tfrac{kT}N,\mc Y^N(k)}\,\diff s
      }
      \\&\leq
      \sum_{k=0}^{n-1}\int_{\tfrac{kT}N}^{\tfrac{(k+1)T}N} \bnorm{f(s,Y(s))-f\bp{\tfrac{kT}N,\mc Y^N(k)}}\,\diff s
      \\&\leq
      \sum_{k=0}^{n-1}\int_{\tfrac{kT}N}^{\tfrac{(k+1)T}N} \bnorm{f(s,Y(s))-f\bp{\tfrac{kT}N,Y(s)}}
        +
        \bnorm{f\bp{\tfrac{kT}N,Y(s)}-f\bp{\tfrac{kT}N,\mc Y^N(k)}}\,\diff s
      \\&\leq
%       \int_{0}^{\tfrac{nT}N} 
%         \bnorm{f(s,Y(s))-f\bp{\floor{s}_N,Y(s)}}\,\diff s
%         \\&\qquad
      \alpha_N
      +
      \sum_{k=0}^{n-1}\int_{\tfrac{kT}N}^{\tfrac{(k+1)T}N} 
        \bnorm{f\bp{\tfrac{kT}N,Y(s)}-f\bp{\tfrac{kT}N,\mc Y^N(k)}}\,\diff s
      .
    \ea
    Moreover, note that for all
      $N\in\N$,
      $k\in\{0,1,\dots,N-1\}\cap[0,\tau_N)$,
      $s\in\bbr{\frac{kT}N,\frac{(k+1)T}N}$
    it holds that
    \ba
    \label{eq:approxa2}
      \norm{Y(s)-\mc Y^N(k)}
      &\leq
      \bnorm{Y(s)-Y\bp{\tfrac{kT}N}}+\bnorm{Y\bp{\tfrac{kT}N}-\mc Y^N(k)}
      \\&\leq
      \norm{Y(s)}+\bnorm{Y\bp{\tfrac{kT}N}}+\bnorm{Y\bp{\tfrac{kT}N}-\mc Y^N(k)}
      \\&\leq R
      .
    \ea
    Furthermore, note that for all
      $N\in\N$,
      $k\in\{0,1,\dots,N-1\}$,
      $s\in\bbr{\frac{kT}N,\frac{(k+1)T}N}$
    it holds that
    \ba
      \bnorm{Y(s)-Y\bp{\tfrac{kT}N}}
      &=
      \bbbnorm{\bbbbr{Y(0)+\int_0^sf(u,Y(u))\,\diff u}-\bbbbr{Y(0)+\int_0^{\frac{kT}N}f(u,Y(u))\,\diff u}}
      \\&=
      \bbbnorm{\int_{\frac{kT}N}^sf(u,Y(u))\,\diff u}
      \\&\leq
      \int_{\frac{kT}N}^s\norm{f(u,Y(u))}\,\diff u
      \\&\leq
      L\bp{s-\tfrac{kT}N}
      \\&\leq
      \frac{LT}N
      .
    \ea
      This 
      and \eqref{eq:approxa2}
    imply that for all
      $N\in\N$,
      $k\in\{0,1,\dots,N-1\}\cap[0,\tau_N)$,
      $s\in\bbr{\frac{kT}N,\frac{(k+1)T}N}$
    it holds that
    \ba
      \bnorm{f\bp{\tfrac{kT}N,Y(s)}-f\bp{\tfrac{kT}N,\mc Y^N(k)}}
      &\leq
      L\norm{Y(s)-\mc Y^N(k)}
      \\&\leq
      L\bnorm{Y(s)-Y\bp{\tfrac{kT}N}}+L\bnorm{Y\bp{\tfrac{kT}N}-\mc Y^N(k)}
      \\&\leq
      \frac{L^2T}N+L\bnorm{Y\bp{\tfrac{kT}N}-\mc Y^N(k)}
      .
    \ea
    Combining
      this
    with
      \eqref{eq:approxa1}
    shows that
      $N\in\N$,
      $n\in\{0,1,\dots,\tau_N\}$
    it holds that
    \ba
      \bnorm{Y\bp{\tfrac{nT}N}-\mc Y^N(n)}
%       &\sum_{k=0}^{n-1}\int_{\tfrac{kT}N}^{\tfrac{(k+1)T}N} 
%         \bnorm{f\bp{\tfrac{kT}N,Y(s)}-f\bp{\tfrac{kT}N,\mc Y^N(k)}}\,\diff s
      &\leq
      \alpha_{N}
      +
      \sum_{k=0}^{n-1}\bbbbr{\int_{\tfrac{kT}N}^{\tfrac{(k+1)T}N} 
        \frac{L^2T}N+L\bnorm{Y\bp{\tfrac{kT}N}-\mc Y^N(k)}\,\diff s}
      \\&=
      \alpha_{N}
      +
      \sum_{k=0}^{n-1}
        \bbbbr{\frac{L^2T^2}{N^2}+\frac{LT}N\bnorm{Y\bp{\tfrac{kT}N}-\mc Y^N(k)}}
      \\&=
      \alpha_{N}
      +
      \frac{L^2T^2n}{N^2}+\frac{LT}N\bbbp{\sum_{k=0}^{n-1}\bnorm{Y\bp{\tfrac{kT}N}-\mc Y^N(k)}}
      \\&\leq
      \alpha_{N}
      +
      \frac{L^2T^2}{N}+\frac{LT}N\bbbp{\sum_{k=0}^{n-1}\bnorm{Y\bp{\tfrac{kT}N}-\mc Y^N(k)}}
      .
    \ea
      Lemma~\ref{lem:discr_gronwall}
      (with
        $N\is \tau_N$,
        $\beta\is \frac{LT}N$,
        $\alpha\is \alpha_{N}+\frac{L^2T^2}{N}$,
        $(f_n)_{n\in\{0,1,\dots,N\}}\is \bp{\bnorm{Y\bp{\tfrac{nT}N}-\mc Y^N(n)}}_{n\in\{0,1,\dots,\tau_N\}}$
        for
        $N\in\N$
      in the notation of Lemma~\ref{lem:discr_gronwall})
      hence 
    establishes that for all
      $N\in\N$,
      $n\in\{0,1,\dots,\tau_N\}$
    it holds that
    \beq
    \label{eq:discgronc}
      \bnorm{Y\bp{\tfrac{nT}N}-\mc Y^N(n)}
      \leq
      \bbbabs{\alpha_{N}+\frac{L^2T^2}{N}}\exp\bbbp{\frac{LTn}N}
      \leq
      \bbbabs{\alpha_{N}+\frac{L^2T^2}{N}}\exp(LT)
      .
    \eeq
    In the next step we observe that
      the fact that $f$ is a continuous function
    ensures that there exist
      $\delta_\eps\in(0,\infty)$, $\eps\in(0,\infty)$,
    such that for all
      $\eps\in(0,\infty)$,
      $s,t\in[0,T]$,
      $x\in\{z\in\R^d\colon\norm{z}\leq R\}$
      with $\abs{s-t}\leq \delta_{\eps}$
    it holds that
    \beq
      \norm{f(s,x)-f(t,x)}<\eps
      .
    \eeq
    Hence, it holds for all
      $\eps\in(0,\infty)$,
      $N\in\N\cap[T/\delta_\eps,\infty)$
    that
    \beq
      \alpha_N
      \leq
      \sum_{k=0}^{N-1}\int_{\tfrac{kT}N}^{\tfrac{(k+1)T}N} \eps\,\diff s
      =
      T\eps
      .
    \eeq
    Therefore, we obtain that
    \beq
    \label{eq:limalphaN0}
      \limsup_{N\to\infty} \alpha_N=0.
    \eeq
      This 
      and \eqref{eq:discgronc}
    prove that there exists 
      $M\in\N$ 
    such that for all
      $N\in\N\cap[M,\infty)$
    it holds that
    \beq
      \bnorm{Y\bp{\tfrac{\tau_NT}N}-\mc Y^N(\tau_N)}
      <
      1
      .
    \eeq
    Combining
      this 
    with
      \eqref{eq:approxatauN}
    shows that for all
      $N\in\{M,M+1,\dots\}$
    it holds that
    \beq
      \tau_N=N
      .
    \eeq
      This
      and \eqref{eq:discgronc}
    show that for all
      $N\in\{M,M+1,\dots\}$
    it holds that
    \beq
      \sup_{n\in\{0,1,\dots,N\}}\bnorm{Y\bp{\tfrac{nT}N}-\mc Y^N(n)}
      \leq
      \bbbabs{\alpha_{N}+\frac{L^2T^2}{N}}\exp(LT)
      .
    \eeq
    Combining
      this
    with
      \eqref{eq:limalphaN0}
    proves
      \eqref{eq:approxaconc}.
    The proof of Lemma~\ref{lem:approxa} is thus completed.
\end{proof}

\begin{lemma}
  \label{lem:approxa2}
  Let $d\in\N$, 
  $T\in[0,\infty)$,
  $f\in C([0,T]\times\R^d,\R^d)$,
  $w\in C([0,T],\R^d)$,
  $\xi\in\R^d$,
  let $\norm{\cdot}\colon\R^d\to[0,\infty)$ be a norm,
  assume for all $r\in(0,\infty)$ that
  \beq
    \label{eq:flip3}
    \sup_{t\in[0,T]}\,\sup_{\substack{x,y\in\R^d,\,x\neq y,\\ \norm x+\norm y\leq r}} \frac{\norm{f(t,x)-f(t,y)}}{\norm{x-y}}<\infty,
  \eeq
  let $Y\in C([0,T],\R^d)$ satisfy
  for all 
    $t\in[0,T]$ 
  that
  $
  %\label{eq:Yinteq2}
    Y(t)=\xi+\int_0^tf(s,Y(s))\,\diff s+w(t)
  $,
  and let $\mc Y^N\colon\{0,1,\dots,N\}\to\R^d$, $N\in\N$, satisfy for all
    $N\in\N$, 
    $n\in\{0,1,\dots,N-1\}$ 
  that $\mc Y^N(0)=Y(0)$ and
  \beq
  \label{eq:recYN3}
    \mc Y^N(n+1)
    =
    \mc Y^N(n)+ \tfrac TN f\bp{\tfrac{nT}{N},\mc Y^N(n)}
    -w\bp{\tfrac{nT}N}+w\bp{\tfrac{(n+1)T}N}
    .
  \eeq
  Then it holds that
  \beq
    \limsup_{N\to\infty}\bbbbr{\sup_{n\in\{0,1,\dots,N\}}\bnorm{\mc Y^N(n)-Y\bp{\tfrac{nT}N}}}=0
    .
  \eeq
\end{lemma}
\begin{proof}[Proof of Lemma~\ref{lem:approxa2}]
  Throughout this proof
    let $Z\colon [0,T]\to\R^d$
      \intrtype{be the function which satisfies }%
      \intrtypen{satisfy }%
      for all
        $t\in[0,T]$
      that
      \beq
      \label{eq:defZ3}
        Z(t)=Y(t)-w(t),
      \eeq
    let $g\colon [0,T]\times\R^d\to\R^d$
      \intrtype{be the function which satisfies }%
      \intrtypen{satisfy }%
      for all
        $t\in[0,T]$,
        $x\in\R^d$
      that
      \beq
      \label{eq:defg}
        g(t,x)=f(t,x+w(t)),
      \eeq
  let $\mc Z^N\colon\{0,1,\dots,N\}\to\R^d$, $N\in\N$,
    \intrtype{be the functions which satisfy }%
    \intrtypen{satisfy }%
    for all
      $N\in\N$,
      $n\in\{0,1,\dots,N\}$
    that
    \beq
    \label{eq:defmcZN}
      \mc Z^N(n)=\mc Y^N(n)-w\bp{\tfrac{nT}N}
      ,
    \eeq
  let $C_{t,r}\in[0,\infty]$, $t\in[0,T]$, $r\in(0,\infty)$,
    satisfy 
    for all 
      $t\in[0,T]$, 
      $r\in(0,\infty)$ 
    that
    \beq
      C_{t,r}
      =
      \sup_{\substack{x,y\in\R^d,\,x\neq y,\\\norm x+\norm y\leq r}}\frac{\norm{f(t,x)-f(t,y)}}{\norm {x- y}},
    \eeq
  and let $K\in[0,\infty]$ satisfy
    \beq
      K=\sup_{t\in[0,T]}\norm{w(t)}
      .
    \eeq
  Note that 
    \eqref{eq:flip3} 
    and the hypothesis that $w%\in C([0,T],\R^m)
      $ is a continuous function
  imply that for all 
    $r\in(0,\infty)$ 
  it holds that
  \beq
  \label{eq:bndCK2}
    \sup_{t\in[0,T]} C_{t,r}<\infty
    \qqandqq 
    K<\infty
    .
  \eeq
  In addition, observe that 
    the hypothesis that $Y$, $w$, and $f$ are continuous functions
  shows that
  \beq
  \label{eq:Zgcont}
    Z\in C([0,T],\R^d)
    \qqandqq
    g\in C([0,T]\times\R^d,\R^d)
    .
  \eeq
  Moreover, note that 
    the triangle inequality 
  ensures that for all 
    $t\in[0,T]$,
    $r\in(0,\infty)$, 
    $x,y\in\R^d$
      with $\norm x+\norm y\leq r$
  it holds that 
  \beq
    \norm{x+w(t)}+\norm{y+w(t)}
    \leq \norm x+\norm{w(t)}+\norm y+\norm{ w(t)}
    %&\leq \norm x+\norm y+2K\\
    \leq r+2K
    .
  \eeq
    This
  demonstrates that for all
    $t\in[0,T]$,
    $r\in(0,\infty)$,
    $x,y\in\R^d$ with $x\neq y$ and $\norm x+\norm y\leq r$
  it holds that
  \ba
    \frac{\norm{g(t,x)-g(t,y)}}{\norm{x-y}}
    &=
    \frac{\norm{f(t,x+w(t))-f(t,y+w(t))}}{\norm{(x+w(t))-(y+w(t))}}
    \\&\leq
    \sup_{\substack{u,v\in\R^d,\,u\neq v,\\\norm{u}+\norm{v}\leq r+2K}}
      \frac{\norm{f(t,u)-f(t,v)}}{\norm{u-v}}
    \\&=
    C_{t,r+2K}.
  \ea
  Combining
    this
    with \eqref{eq:bndCK2} 
  ensures that for all
    $r\in(0,\infty)$
  it holds that
  \beq
  \label{eq:glip}
    \sup_{t\in[0,T]}\,\sup_{\substack{x,y\in\R^d,\,x\neq y,\\\norm x+\norm y\leq r}}
      \frac{\norm{g(t,x)-g(t,y)}}{\norm{x-y}}
    \leq
    \sup_{t\in[0,T]}
      C_{t,r+2K}
    <\infty
    .
  \eeq
  Moreover, observe that
    the hypothesis that
      for all
        $t\in[0,T]$
      it holds that
        $Y(t)=\xi+\int_0^tf(s,Y(s))\,\diff s+w(t)$,
    \eqref{eq:defZ3},
    \eqref{eq:defg},
    and \eqref{eq:Zgcont}
%     and the assumption that
%       for all
%         $t\in[0,T]$
%       it holds that
%         $Y(t)=x+\int_0^t f(s,Y(s))\,\diff s+w(t)$
  show that for all
    $t\in[0,T]$
  it holds that
  \beq
  \label{eq:Zint4}
    Z(t)
    =
    \xi+\int_0^t f(s,Y(s))\,\diff s
    =
    \xi+\int_0^t f(s,Z(s)+w(s))\,\diff s
    =
    Z(0)+\int_0^t g(s,Z(s))\,\diff s
    .
  \eeq
  Next we combine
    \eqref{eq:recYN3},
    \eqref{eq:defg},
    and \eqref{eq:defmcZN}
  to obtain that for all
    $N\in\N$,
    $n\in\{0,1,\dots,N-1\}$
  it holds that
  \ba
  \label{eq:recZN3}
    \mc Z^N(n+1)
    &=
    \mc Y^N(n+1)-w\bp{\tfrac{(n+1)T}N}
    \\&=
    \bbbr{\mc Y^N(n)+ \tfrac TN f\bp{\tfrac{nT}{N},\mc Y^N(n)}
      -w\bp{\tfrac{nT}N}+w\bp{\tfrac{(n+1)T}N}}
      -w\bp{\tfrac{(n+1)T}N}
    \\&=
    \mc Y^N(n)-w\bp{\tfrac{nT}N} + \tfrac TN f\bp{\tfrac{nT}{N},\mc Y^N(n)}
    \\&=
    \mc Z^N(n) + \tfrac TN f\bp{\tfrac{nT}{N},\mc Z^N(n)+w\bp{\tfrac{nT}N}}
    \\&=
    \mc Z^N(n) + \tfrac TN g\bp{\tfrac{nT}{N},\mc Z^N(n)}
    .
  \ea
  Furthermore, observe that 
    the assumption that for all
      $N\in\N$
    it holds that
      $\mc Y^N(0)=Y(0)$
  ensures that for all
    $N\in\N$
  it holds that
  \beq
    \mc Z^N(0)
    =
    \mc Y^N(0)-w(0)
    =
    Y(0)-w(0)
    =
    Z(0)
    .
  \eeq
  Combining
    this,
    \eqref{eq:Zgcont},
    \eqref{eq:glip},
    \eqref{eq:Zint4},
    and \eqref{eq:recZN3}
  with
    Lemma~\ref{lem:approxa}
    (with
      $d\is d$,
      $T\is T$,
      $f\is g$,
      $\norm\cdot\is\norm\cdot$,
      $Y\is Z$,
      $(\mc Y^N)_{N\in\N}\is\mc (\mc Z^N)_{N\in\N}$
    in the notation of Lemma~\ref{lem:approxa})
  shows that
  \beq
  \label{eq:limsupZN}
    \limsup_{N\to\infty}\bbbbr{\sup_{n\in\{0,1,\dots,N\}}\bnorm{\mc Z^N(n)-Z\bp{\tfrac{nT}N}}}
    =
    0
    .
  \eeq
  Moreover, note that for all
    $N\in\N$,
    $n\in\{0,1,\dots,N\}$
  it holds that
  \beq
    \bnorm{\mc Y^N(n)-Y\bp{\tfrac{nT}N}}
    =
    \bnorm{\bbr{\mc Z^N(n)+w\bp{\tfrac{nT}N}}-\bbr{Z\bp{\tfrac{nT}N}+w\bp{\tfrac{nT}N}}}
    =
    \bnorm{\mc Z^N(n)-Z\bp{\tfrac{nT}N}}
    .
  \eeq
  Combining 
    this 
  with
    \eqref{eq:limsupZN}
  establishes that
  \beq
    \limsup_{N\to\infty}\bbbbr{\sup_{n\in\{0,1,\dots,N\}}\bnorm{\mc Y^N(n)-Y\bp{\tfrac{nT}N}}}
    =0
    .
  \eeq
  This completes the proof of Lemma~\ref{lem:approxa2}.
\end{proof}

\subsection{Time-continuous approximations for deterministic differential equations}
\label{subsec:conteuler}
\begin{lemma}
\label{lem:interpol}
  Let $d\in\N$,
  $T\in(0,\infty)$,
  $Y\in C([0,T],\R^d)$,
  let $\norm{\cdot}\colon \R^d\to[0,\infty)$ be a norm,
  let $\mc Y^N\colon[0,T]\to\R^d$, $N\in\N$, satisfy
    for all
      $N\in\N$,
      $n\in\{0,1,\dots,N-1\}$,
      $t\in\bbr{\tfrac{nT}N,\tfrac{(n+1)T}N}$
    that
    \beq
    \label{eq:interpYN}
      \mc Y^N(t)=\bp{1-\bp{\tfrac{tN}T-n}}\mc Y^N\bp{\tfrac{nT}N}+\bp{\tfrac{tN}T-n}\mc Y^N\bp{\tfrac{(n+1)T}N},
    \eeq
    and assume that
    \beq
    \label{eq:pointconv}
      \limsup_{N\to\infty}\bbbbr{\sup_{n\in\{0,1,\dots,N\}}\bnorm{\mc Y^N\bp{\tfrac{nT}N}-Y\bp{\tfrac{nT}N}}}
      =
      0
      .
    \eeq
  Then it holds that
    \beq
    \label{eq:YNconv}
      \limsup_{N\to\infty}\bbbbr{\sup_{t\in[0,T]}\norm{\mc Y^N(t)-Y(t)}}
      =
      0
      .
    \eeq
\end{lemma}
\begin{proof}[Proof of Lemma~\ref{lem:interpol}]
  First, note that 
    \eqref{eq:pointconv}
  ensures that there exist
    $M_\eps\in\N$, $\eps\in(0,\infty)$, 
  such that for all
    $N\in\N\cap[M_\eps,\infty)$
  it holds that
  \beq
  \label{eq:limsupZ}
    \sup_{n\in\{0,1,\dots,N\}}\bnorm{\mc Y^N\bp{\tfrac{nT}N}-Y\bp{\tfrac{nT}N}}<\eps
    .
  \eeq
  Next observe that 
    the hypothesis that $Y\colon[0,T]\to\R^d%
      $ is a continuous function
  implies that
    $Y$ is a uniformly continuous function.
    This
  ensures that there exist
    $\delta_\eps\in(0,\infty)$, $\eps\in(0,\infty)$,
  such that for all
    $\eps\in(0,\infty)$,
    $s,t\in[0,T]$ with $\abs{s-t}\leq\delta_\eps$
  it holds that
  \beq
    \label{eq:unifcont}
    \norm{Y(s)-Y(t)}
    <
    \eps
    .
  \eeq
  In the next step we observe that for all
    $\eps\in(0,\infty)$,
    $N\in\N\cap[T/\delta_{\eps/2},\infty)$,
    $n\in\{0,1\dots,N-1\}$,
    $t\in\bigl[\tfrac{nT}N,\tfrac{(n+1)T}N\bigr]$,
    $k\in\{n,n+1\}$
  it holds that
  \beq
    \babs{t-\tfrac{kT}{N}}
    \leq
    \tfrac TN
    \leq
    \delta_{\eps/2}
    .
  \eeq
    This,
    \eqref{eq:limsupZ},
    \eqref{eq:unifcont},
    and the triangle inequality
  show that for all
    $\eps\in(0,\infty)$,
    $N\in\N\cap[\max\{T/\delta_{\eps/2},M_{\eps/2}\},\infty)$,
    $n\in\{0,1,\dots,N-1\}$,
    $t\in\bigl[\tfrac{nT}N,\tfrac{(n+1)T}N\bigr]$,
    $k\in\{n,n+1\}$
  it holds that
  \beq
    \bnorm{\mc Y^N\bp{\tfrac{kT}N}-Y(t)}
    %&=
    %\bnorm{\mc Y^N\bp{\tfrac{kT}N}-Y\bp{\tfrac{kT}{N}}+Y\bp{\tfrac{kT}{N}}-Y(t)}
    \leq
    \bnorm{\mc Y^N\bp{\tfrac{kT}N}-Y\bp{\tfrac{kT}{N}}}+\bnorm{Y\bp{\tfrac{kT}{N}}-Y(t)}
    <
    \tfrac{\eps}2+\tfrac\eps2
    =
    \eps
    .
  \eeq
    The triangle inequality
    and \eqref{eq:interpYN}
    hence
  demonstrate that for all
    $\eps\in(0,\infty)$,
    $N\in\N\cap[\max\{T/\delta_{\eps/2},M_{\eps/2}\},\infty)$,
    $n\in\{0,1,\dots,N-1\}$,
    $t\in\bigl[\tfrac{nT}N,\tfrac{(n+1)T}N\bigr]$
  it holds that
  \ba
    &\norm{\mc Y^N(t)-Y(t)}\\
    &=
    \bnorm{\bbr{\bp{1-\bp{\tfrac{tN}T-n}}\mc Y^N\bp{\tfrac{nT}N}+\bp{\tfrac{tN}T-n}\mc Y^N\bp{\tfrac{(n+1)T}N}}-Y(t)}\\
    &=
    \bnorm{\bp{1-\bp{\tfrac{tN}T-n}}\bp{\mc Y^N\bp{\tfrac{nT}N}-Y(t)}+\bp{\tfrac{tN}T-n}\bp{\mc Y^N\bp{\tfrac{(n+1)T}N}-Y(t)}}\\
    &\leq
    \bnorm{\bp{1-\bp{\tfrac{tN}T-n}}\bp{\mc Y^N\bp{\tfrac{nT}N}-Y(t)}}+\bnorm{\bp{\tfrac{tN}T-n}\bp{\mc Y^N\bp{\tfrac{(n+1)T}N}-Y(t)}}\\
    &=
    \bp{1-\bp{\tfrac{tN}T-n}}\bnorm{\mc Y^N\bp{\tfrac{nT}N}-Y(t)}+\bp{\tfrac{tN}T-n}\bnorm{\mc Y^N\bp{\tfrac{(n+1)T}N}-Y(t)}\\
    &<
    \bp{1-\bp{\tfrac{tN}T-n}}\eps+\bp{\tfrac{tN}T-n}\eps
    \\&=
    \eps
    .
  \ea
  Therefore, we obtain that for all
    $\eps\in(0,\infty)$,
    $N\in\N\cap[\max\{T/\delta_{\eps/2},M_{\eps/2}\},\infty)$
  it holds that
  \beq
    \sup_{t\in[0,T]}\norm{\mc Y^N(t)-Y(t)}
    \leq
    \eps
    .
  \eeq
  This establishes \eqref{eq:YNconv}. 
  The proof of Lemma~\ref{lem:interpol} is thus completed.
\end{proof}

\subsection{Measurability properties for solutions of SDEs}
\label{subsec:meassdes}

\begin{lemma}
\label{lem:intmeas}
  Let $d\in\N$,
  $T\in[0,\infty)$,
  $f\in C([0,T]\times\R^d,\R^d)$,
  $\xi\in\R^d$,
  let $(\Omega,\mc F,\PP)$ be a probability space,
  let $W\colon [0,T]\times \Omega\to\R^d$ be a stochastic process 
    with continuous sample paths,
  let $\norm{\cdot}\colon\R^d\to[0,\infty)$ be a norm,
  assume for all $r\in(0,\infty)$ that
  \beq
  \label{eq:flip2}
    \sup_{t\in[0,T]}\,\sup_{\substack{x,y\in\R^d,\,x\neq y,\\ \norm x+\norm y\leq r}} \frac{\norm{f(t,x)-f(t,y)}}{\norm{x-y}}<\infty,
  \eeq
  and let $Y\colon[0,T]\times \Omega\to \R^d$ satisfy for all 
    $t\in[0,T]$,
    $\omega\in\Omega$
  that
    $\bp{[0,T]\ni s\mapsto Y(s,\omega)\in\R^d}\in C([0,T],\R^d)$
  and
  \beq
  \label{eq:Yint2}
    Y(t,\omega)=\xi+\int_0^t f(s,Y(s,\omega))\,\diff s+W(t,\omega).
  \eeq
  Then it holds that $Y$ is a stochastic process.
\end{lemma}
\begin{proof}[Proof of Lemma~\ref{lem:intmeas}]
  Throughout this proof
    assume w.l.o.g.\ that $T>0$,
    let $\mc Z^N\colon \{0,1,\dots,N\}\times\Omega\to\R^d$, $N\in\N$,
      be the sequence of functions which satisfies
      for all
        $N\in\N$,
        $n\in\{0,1,\dots,N-1\}$,
        $\omega\in\Omega$
      that
        $\mc Z^N(0,\omega)=Y(0,\omega)$ and
        \beq
        \label{eq:recZN}
          \mc Z^N(n+1,\omega)
          =
          \mc Z^N(n,\omega)+\tfrac TNf\bp{\tfrac{nT}N,\mc Z^N(n,\omega)}
            -W\bp{\tfrac{nT}N,\omega}
            +W\bp{\tfrac{(n+1)T}N,\omega},
        \eeq
    and let $\mc Y^N\colon[0,T]\times\Omega \to\R^d$, $N\in\N$, 
      be the sequence of functions which satisfies 
      for all
        $N\in\N$, 
        $n\in\{0,1,\dots,N-1\}$,
        $t\in\bigl(\frac{nT}N,\frac{(n+1)T}N\bigr]$,
        $\omega\in\Omega$
      that $\mc Y^N(0,\omega)=\mc Z^N(0,\omega)$ and
      \beq
      \label{eq:defYN}
        \mc Y^N(t,\omega)=\bp{1-\bp{\tfrac{tN}T-n}}\mc Z^N(n,\omega)+\bp{\tfrac{tN}T-n}\mc Z^N(n+1,\omega)
        .
      \eeq
  Note that
    \eqref{eq:flip2},
    \eqref{eq:Yint2},
    \eqref{eq:recZN},
    and Lemma~\ref{lem:approxa2}
    (with
      $d\is d$,
      $T\is T$,
      $f\is f$,
      $w\is \bp{[0,T]\ni t\mapsto W(t,\omega)\in\R^d}$,
      $\xi\is\xi$,
      $\norm\cdot\is\norm\cdot$,
      $Y\is\bp{[0,T]\ni t\mapsto Y(t,\omega)\in\R^d}$,
      $(\mc Y^N)_{N\in\N}\is\bp{\{0,1,\dots,N\}\ni n\mapsto\mc Z^N(n,\omega)\in\R^d}{}_{N\in\N}$
      for $\omega\in\Omega$
    in the notation of Lemma~\ref{lem:approxa2})
  show that for all
    $\omega\in\Omega$
  it holds that
  \beq
  \label{eq:ZNconv}
    \limsup_{N\to\infty}\bbbbr{
      \sup_{n\in\{0,1,\dots,N\}}
        \bnorm{\mc Z^N(n,\omega)-Y\bp{\tfrac{nT}N,\omega}}
    }
    =0
    .
  \eeq
  In the next step we observe that
    \eqref{eq:defYN}
  implies that for all
    $N\in\N$,
    $n\in\{0,1,\dots,N-1\}$,
    $\omega\in\Omega$
  it holds that
  \ba
    \mc Y^N\bp{\tfrac{(n+1)T}N,\omega}
    &=
    \bp{1-\bp{(n+1)-n}}\mc Z^N(n,\omega)+\bp{(n+1)-n}\mc Z^N(n+1,\omega)
    \\&=
    \mc Z^N(n+1,\omega)
    .
  \ea
  Combining
    this
  with
    the fact that
      for all
        $N\in\N$,
        $\omega\in\Omega$
      it holds that
        $\mc Y^N(0,\omega)=\mc Z^N(0,\omega)$
  shows that for all
    $N\in\N$,
    $n\in\{0,1,\dots,N\}$,
    $\omega\in\Omega$
  it holds that
  \beq
  \label{eq:YNZN}
    \mc Y^N\bp{\tfrac{nT}N,\omega}
    =
    \mc Z^N(n,\omega)
    .
  \eeq
    This
    and \eqref{eq:defYN}
  prove that for all
    $N\in\N$, 
    $n\in\{0,1,\dots,N-1\}$,
    $t\in\bigl[\frac{nT}N,\frac{(n+1)T}N\bigr]$,
    $\omega\in\Omega$
  it holds that 
    $\mc Y^N(0,\omega)=\mc Z^N(0,\omega)$ 
  and
  \beq
    \mc Y^N(t,\omega)=\bp{1-\bp{\tfrac{tN}T-n}}\mc Y^N\bp{\tfrac{nT}N,\omega}+\bp{\tfrac{tN}T-n}\mc Y^N\bp{\tfrac{(n+1)T}N,\omega}
    .
  \eeq
  Combining
    this,
    \eqref{eq:ZNconv},
    and \eqref{eq:YNZN},
  with
    Lemma~\ref{lem:interpol}
    (with
      $d\is d$,
      $T\is T$,
      $Y\is([0,T]\ni t\mapsto Y(t,\omega)\in\R^d)$,
      $\norm\cdot\is\norm\cdot$,
      $(\mc Y^N)_{N\in\N}\is([0,T]\ni t\mapsto \mc Y^N(t,\omega)\in\R^d)_{N\in\N}$
      for $\omega\in\Omega$
    in the notation of Lemma~\ref{lem:interpol})
  establishes that for all
    $\omega\in\Omega$
  it holds that
  \beq
  \label{eq:limYNY}
    \limsup_{N\to\infty}\bbbbr{\sup_{t\in[0,T]}\norm{\mc Y^N(t,\omega)-Y(t,\omega)}}
    =0
    .
  \eeq
%   %
%     This,
%     \eqref{eq:defYN},
%     and \eqref{eq:recZN}
%   prove that for all
%     $N\in\N$,
%     $n\in\{0,1,\dots,N-1\}$,
%     $t\in\bigl(\frac{nT}N,\frac{(n+1)T}N\bigr]$,
%     $\omega\in\Omega$
%   it holds that
%   \ba
%     &\mc Y^N(t,\omega)
%     \\&=
%     \bp{1-\bp{\tfrac{tN}T-n}}\mc Y^N\bp{\tfrac{nT}N,\omega}+\bp{\tfrac{tN}T-n}\Bigl[ \mc Y^N\bp{\tfrac{nT}N,\omega}+\tfrac TN f\bp{\tfrac{nT}N,\mc Y^N\bp{\tfrac{nT}N,\omega}}
%       \\&\qquad
%       +\sigma\bp{W\bp{\tfrac{(n+1)T}N,\omega}-W\bp{\tfrac{nT}N,\omega}} \Bigr]
%     \\&=
%     \mc Y^N\bp{\tfrac{nT}N,\omega}
%       +\bp{t-\tfrac{nT}N}f\bp{\tfrac{nT}N,\mc Y^N\bp{\tfrac{nT}N,\omega}}
%       +\bp{\tfrac{tN}T-n}\bp{W\bp{\tfrac{(n+1)T}N,\omega}-W\bp{\tfrac{nT}N,\omega}}
%     .
%   \ea
%   Combining
%     this
%   with 
%     Lemma~\ref{lem:approxd}
%   shows that for all
%     $\omega\in\Omega$
%   it holds that
%   \beq
%     \limsup_{N\to\infty}\sup_{t\in[0,T]}\norm{Y(t,\omega)-\mc Y^N(t,\omega)}
%     =
%     0
%     .
%   \eeq
  %
    Next observe that 
    the assumption that 
      $\Omega\ni\omega\mapsto W(0,\omega)\in\R^d$ is an $\mc F$/$\mc B(\R^d)$-measurable function
    and the fact that
      for all
        $N\in\N$,
        $\omega\in\Omega$
      it holds that
        $\mc Z^N(0,\omega)=\xi+W(0,\omega)$
    imply that for all
      $N\in\N$
    it holds that 
      \beq
      \label{eq:ZNnmeasind0}
        \Omega\ni\omega\mapsto \mc Z^N(0,\omega)\in\R^d
      \eeq
    is an $\mc F$/$\mc B(\R^d)$-measurable function.
  Furthermore, observe that
    \eqref{eq:recZN},
    the hypothesis that $f$ is a continuous function,
    and the hypothesis that 
      for all 
        $t\in[0,T]$
      it holds that 
        $\Omega\ni\omega\mapsto W(t,\omega)\in\R^d$ is an $\mc F$/$\mc B(\R^d)$-measurable function
    show that for all 
      $N\in\N$,
      $n\in\{m\in\{0,1,\dots,N-1\}\colon(\Omega\ni\omega\mapsto \mc Z^N(m,\omega)\in\R^d)\text{ is an $\mc F$/$\mc B(\R^d)$-measurable function}\}$
    it holds that
      \beq
        \Omega\ni\omega\mapsto \mc Z^N(n+1,\omega)\in\R^d
      \eeq 
      is an $\mc F$/$\mc B(\R^d)$-measurable function.
  Combining
    this
    and \eqref{eq:ZNnmeasind0}
  with
    the induction principle
  proves that for all
    $N\in\N$,
    $n\in\{0,1,\dots,N\}$
  it holds that
  \beq
    \Omega\ni\omega\mapsto \mc Z^N(n,\omega)\in\R^d
  \eeq
  is an $\mc F$/$\mc B(\R^d)$-measurable function.
    This
    and \eqref{eq:defYN}
  demonstrate that for all
    $N\in\N$,
    $t\in[0,T]$
  it holds that
  $
    \Omega\ni\omega\mapsto \mc Y^N(t,\omega)\in\R^d
  $
  is an $\mc F$/$\mc B(\R^d)$-measurable function.
  Combining
    this
    and \eqref{eq:limYNY}
  with
    Lemma~\ref{lem:convergence_measurability_d}
    (with
      $(\Omega,\mc F)\is(\Omega,\mc F)$,
      $d\is d$,
      $\norm\cdot\is\norm\cdot$,
      $Y\is(\Omega\ni\omega\mapsto Y(t,\omega)\in\R^d)$,
      $(X_n)_{n\in\N}\is(\Omega\ni\omega\mapsto \mc Y^N(t,\omega)\in\R^d)_{N\in\N}$
      for $t\in[0,T]$
    in the notation of Lemma~\ref{lem:convergence_measurability_d})
  ensures that for all
    $t\in[0,T]$
  it holds that
  \beq
    \Omega\ni\omega\mapsto Y(t,\omega)\in\R^d
  \eeq
  is an $\mc F$/$\mc B(\R^d)$-measurable function.
  The proof of Lemma~\ref{lem:intmeas} is thus completed.
\end{proof}

\section{Differentiability with respect to the initial value for SDEs}\label{section5}

In this section we establish in Lemma~\ref{lem:existX} in Subsection~\ref{subsec:diffsdes} 
below an existence, 
uniqueness, and regularity result for solutions of certain additive 
noise driven SDEs. Our proof of Lemma~\ref{lem:existX} exploits the related regularity 
results for solutions of certain ODEs in Lemmas~\ref{lem:solnlip}--\ref{lem:diffy} below. 
For the reader's convenience we include in this section also detailed proofs for
Lemmas~\ref{lem:solnlip}--\ref{lem:existX}.

\subsection{Local Lipschitz continuity for deterministic DEs}
\begin{lemma}
  \label{lem:solnlip}
  Let $d\in\N$,
  $w\in\R^d$,
  $T\in[0,\infty)$,
  $f\in C^{0,1}([0,T]\times\R^d,\R^d)$,
  let $\norm\cdot\colon\R^d\to[0,\infty)$ be a norm,
  and let $y^x\in C([0,T],\R^d)$, $x\in\R^d$, 
    be functions which
    satisfy for all
      $x\in\R^d$,
      $t\in[0,T]$
    that 
      \beq
      \label{eq:diffeqyx}
      y^x(t)=x+\int_0^t f(s,y^x(s))\,\diff s.
      \eeq
  Then 
    there exist
      $r,L\in(0,\infty)$
    such that for all
      $v\in\{u\in\R^d\colon\norm{u-w}\leq r\}$,
      $t\in[0,T]$
    it holds that
    \beq
      \norm{y^v(t)-y^w(t)}
      \leq 
      L\norm{v-w}
      .
    \eeq
\end{lemma}
\begin{proof}[Proof of Lemma~\ref{lem:solnlip}]
  Throughout this proof
    assume w.l.o.g.\ that $T>0$,
    let $D\colon [0,T]\times\R^d\to\R^{d\times d}$ 
      be the function which satisfies for all
        $t\in[0,T]$,
        $x\in\R^d$
      that
      \beq
        D(t,x)=\tfrac{\partial}{\partial x} f(t,x),
      \eeq
    let $R\in(0,\infty)$
      be the real number which satisfies
    \beq
      R=1+\bbbbr{\sup_{t\in[0,T]} \norm{y^w(t)-w}},
    \eeq
    let $K\in[0,\infty]$ satisfy
    \beq
      K
      =
      \bbbbr{
      \sup_{x\in\{u\in\R^d\colon\norm{u-w}\leq R\}}\,
        \sup_{t\in[0,T]}\,
          \sup_{h\in\R^d\setminus\{0\}}
          \bbbp{
            \frac{\norm{D(t,x)h}}{\norm h}
          }
      }
      ,
    \eeq
    and let $\tau_{x}\in[0,T]$, $x\in\R^d$,
      be the real numbers which satisfy for all
      $x\in\R^d$
    that
    \beq
    \label{eq:deftaux}
      \tau_{x}=\inf(\{T\}\cup\{t\in[0,T]\colon \norm{y^x(t)-w}\geq R\}).
    \eeq
  Note that
  \ba
    \label{eq:KR}
    K
    &=
    \bbbbr{
    \sup_{x\in\{u\in\R^d\colon \norm {u-w}\leq R\}}\,
      \sup_{t\in[0,T]}\,
        \sup_{h\in\R^d\setminus\{0\}}\,
          \bbbnorm{D(t,x)\bbbp{\frac{h}{\norm h}}}
    }
    \\&=
    \sup_{\substack{(x,t,v)\in\R^d\times[0,T]\times\R^d,\\\norm {x-w}\leq R,\,\norm v=1}}\,
      \norm{D(t,x)v}
    .
  \ea
  In addition, observe that
    the hypothesis that $f\in C^{0,1}([0,T]\times\R^d,\R^d)$
  implies that
    $\R^d\times[0,T]\times\R^d\ni(x,t,v)\mapsto D(t,x)v\in\R^d$
      is a continuous function.
  Combining 
    this 
  with
    \eqref{eq:KR}
  establishes that
  \beq
  \label{eq:KRfin}
    K<\infty
    .
  \eeq
    The fundamental theorem of calculus 
    hence
  implies that for all
    $t\in[0,T]$,
    $u,v\in\{z\in\R^d\colon \norm{z-w}\leq R\}$
  it holds that
  \ba
  \label{eq:fllip}
    \norm{f(t,u)-f(t,v)}
    &=
    \bnorm{\bbr{f(t,v+s(u-v))}_{s=0}^{s=1}}
    \\&=
    \bbbnorm{\int_{0}^1 D(t,v+s(u-v))(u-v)\,\diff s}
    \\&\leq
    \int_{0}^1 \norm{D(t,v+s(u-v))(u-v)}\,\diff s
    \\&\leq
    \int_{0}^1 K\norm{u-v}\,\diff s
    \\&=
    K\norm{u-v}
    .
  \ea
  Moreover, note that 
    \eqref{eq:deftaux}
  ensures that for all
    $v\in\R^d$,
    $t\in[0,T]\cap(-\infty,\tau_{v})$
  it holds that
  \beq
    \norm{y^v(t)-w}\leq R.
  \eeq
    The triangle inequality,
    \eqref{eq:diffeqyx},
    and \eqref{eq:fllip}
    hence
  show that for all
    $v\in\R^d$,
    $t\in[0,T]\cap(-\infty,\tau_v)$
  it holds that
  \ba
    \norm{y^v(t)-y^w(t)}
    &=
    \bbbnorm{
      \bbbbr{v+\int_0^t f(s,y^v(s))\,\diff s}
      -
      \bbbbr{w+\int_0^t f(s,y^w(s))\,\diff s}
    }
    \\&=
    \bbbnorm{
      v-w+\int_0^t f(s,y^v(s))-f(s,y^w(s))\,\diff s
    }
    \\&\leq
    \norm{v-w}+\bbbnorm{\int_0^t f(s,y^v(s))-f(s,y^w(s))\,\diff s}
    \\&\leq
    \norm{v-w}+\int_0^t \norm{f(s,y^v(s))-f(s,y^w(s))}\,\diff s
    \\&\leq
    \norm{v-w}+\int_0^t K\norm{y^v(s)-y^w(s)}\,\diff s
    .
  \ea
    Gronwall's integral inequality
      (see, e.g., Grohs et al.~\cite[Lemma~2.11]{GHJW}
        (with
          $\alpha\is \norm{v-w}$,
          $\beta\is K$,
          $T\is t$,
          $f\is ([0,t]\ni s\mapsto \norm{y^v(s)-y^w(s)}\in[0,\infty))$
          for
          $v\in\R^d$,
          $t\in[0,T]\cap(-\infty,\tau_v)$
        in the notation of Grohs et al.~\cite[Lemma~2.11]{GHJW})%
      )
    and \eqref{eq:KRfin}
    hence
  establish that for all
    $v\in\R^d$,
    $t\in[0,T]\cap(-\infty,\tau_v)$
  it holds that
  \beq
    \label{eq:ycontx}
    \norm{y^v(t)-y^w(t)}
    \leq
    \norm{v-w}\exp(Kt)
    \leq
    \norm{v-w}\exp(KT)
    .
  \eeq
%   %
%     This 
%     and the fact that 
%       for all 
%         $x\in\R^d$ 
%       it holds that
%         $y^x(0)=x$
%   imply that for all 
%     $x\in\R^d$, 
%     $t\in[0,T]\cap(-\infty,\tau_x)$
%   it holds that
%   \beq
%     \norm{y^x(t)-y^z(t)}
%     \leq
%     \norm{x-z}\exp(KT)
%     .
%   \eeq
  The triangle inequality therefore proves that for all
    $v\in\R^d$,
    $t\in[0,T]\cap(-\infty,\tau_v)$
      with $\norm{v-w}<(2\exp(KT))^{-1}$
  it holds that
  \ba
  \label{eq:bndyz}
    \norm{y^v(t)-w}
    &\leq
    \norm{y^v(t)-y^w(t)}+\norm{y^w(t)-w}
    \\&\leq
    \norm{v-w}\exp(KT)+R-1
    \\&\leq
    R-\frac12
    .
  \ea
  In addition, observe that for all
    $v\in\R^d$
      with $\norm{v-w}<(2\exp(KT))^{-1}$
  it holds that
  \beq
    \norm{y^v(0)-w}=\norm{v-w}<1\leq R
    .
  \eeq
    This,
    the assumption that $T>0$,
    and the hypothesis that for all $x\in\R^d$ it holds that $y^x$ is a continuous function
  ensure that for all 
    $v\in\R^d$
      with $\norm{v-w}<(2\exp(KT))^{-1}$
  it holds that
  \beq
    \tau_v>0.
  \eeq
  Combining
    this 
    and the hypothesis that for all $x\in\R^d$ it holds that $y^x$ is a continuous function
  with
    \eqref{eq:bndyz}
  implies that for all
    $v\in\R^d$
      with $\norm{v-w}<(2\exp(KT))^{-1}$
  it holds that 
  \beq
  \label{eq:taueqT}
    \tau_{v}=T.
  \eeq
    The fact that for all $x\in\R^d$ it holds that $y^x$ is a continuous function
    and \eqref{eq:ycontx}
    therefore
  ensure that for all 
    $t\in[0,T]$,
    $v\in\R^d$
      with $\norm{v-w}<(2\exp(KT))^{-1}$
  it holds that
  \beq
    \norm{y^v(t)-y^w(t)}
    \leq
    \norm{v-w}\exp(KT)
    .
  \eeq
  The proof of Lemma~\ref{lem:solnlip} is thus completed.
\end{proof}

\subsection{Differentiability with respect to the initial value for deterministic DEs}

\begin{lemma}
  \label{lem:auxdiff}
  Let $d\in\N$, 
  $T\in[0,\infty)$, 
  $f\in C^{0,1}([0,T]\times\R^d,\R^d)$,
  let $D\colon [0,T]\times\R^d\to\R^{d\times d}$ be a function,
  and let $y^x\in C([0,T],\R^d)$, $x\in\R^d$, 
    be functions which
    satisfy for all
      $t\in[0,T]$,
      $x\in\R^d$
    that 
      $D(t,x)=\frac{\partial}{\partial x} f(t,x)$
      and
      \beq
      \label{eq:yxdiffeq}
      y^x(t)=x+\int_0^t f(s,y^x(s))\,\diff s.
      \eeq
  Then
  \begin{enumerate}[label=(\roman{enumi}),ref=(\roman{enumi})]
    \item \label{enum:lemauxdiff:1}
    it holds that 
      $([0,T]\times \R^d\ni(t,x)\mapsto y^x(t)\in\R^d)\in C^{0,1}([0,T]\times\R^d,\R^d)$ 
    and
    \item \label{enum:lemauxdiff:2}
    it holds for all 
      $x,h\in\R^d$, 
      $t\in[0,T]$ 
    that
    \beq
      \bp{\tfrac{\partial}{\partial x}y^x(t)}(h)
      =
      h+\int_0^t D(s,y^x(s))\bp{\bp{\tfrac\partial{\partial x} y^x(s)}(h)}\,\diff s
      .
    \eeq
  \end{enumerate}
\end{lemma}
\begin{proof}[Proof of Lemma~\ref{lem:auxdiff}]
  Throughout this proof 
    let $\norm\cdot\colon\R^d\to[0,\infty)$ be the $d$-dimensional Euclidean norm,
    let $C_r\in[0,\infty)$, $r\in[0,\infty)$,
      be the real numbers which satisfy for all
      $r\in[0,\infty)$
    that
    \beq
      C_r
      =
      \bbbbr{
      \sup_{t\in[0,T]}\,
        \sup_{x\in\{z\in\R^d\colon\norm{z}\leq r\}}
          \norm{f(t,x)}
      },
    \eeq
    let $R_x\in[0,\infty)$, $x\in\R^d$,
      be the real numbers which
      satisfy
      for all 
        $x\in\R^d$
    that
    \beq
      R_x=\sup_{t\in[0,T]}\norm{y^x(t)}
      ,
    \eeq
    and let $K_r\in[0,\infty]$, $r\in[0,\infty)$,
      satisfy for all
        $r\in[0,\infty)$
      that
      \beq
        K_r
        =
        \bbbbr{
        \sup_{x\in\{z\in\R^d\colon\norm{z}\leq r\}}\,
          \sup_{t\in[0,T]}\,
            \sup_{h\in\R^d\setminus\{0\}}
              \bbbp{\frac{\norm{D(t,x)h}}{\norm h}}
        }
        .
      \eeq
  Observe that for all 
    $r\in[0,\infty)$
  it holds that
  \ba
  \label{eq:Kr}
    K_r
    &=
    \bbbbr{
    \sup_{x\in\{z\in\R^d\colon\norm{z}\leq r\}}\,
      \sup_{t\in[0,T]}\,
        \sup_{h\in\R^d\setminus\{0\}}
          \bbbnorm{D(t,x)\bbbp{\frac h{\norm h}}}
    }
    \\&=
    \sup_{\substack{(x,t,w)\in\R^d\times[0,T]\times\R^d,\\\norm{x}\leq r,\,\norm w=1}}\,
      \norm{D(t,x)w}
    .
  \ea
  In addition, observe that
    the fact that $f\in C^{0,1}([0,T]\times\R^d,\R^d)$
  implies that
    $\R^d\times[0,T]\times\R^d\ni(x,t,w)\mapsto D(t,x)w\in\R^d$
      is a continuous function.
  Combining 
    this 
  with
    \eqref{eq:Kr}
  establishes that for all
    $r\in[0,\infty)$
  it holds that
  \beq
  \label{eq:Krfindy}
    K_r<\infty.
  \eeq
  Note that 
    Lemma~\ref{lem:solnlip}
  proves that there exist
    $L_w,r_w\in(0,\infty)$, $w\in\R^d$,
  such that for all
    $v,w\in\R^d$,
    $t\in[0,T]$
      with $\norm{v-w}<r_w$
  it holds that
  \beq
    \label{eq:ycontx2}
    \norm{y^v(t)-y^w(t)}
    \leq
    L_w\norm{v-w}
    .
  \eeq
%   In the next step, we observe that for all
%     $x\in\R^d$,
%     $t\in[0,T]$,
%     $u\in[0,t]$
%   it holds that
%   \ba
%   \label{eq:yxtcont2}
%     \norm{y^x(t)-y^x(u)}
%     &=
%     \bbbnorm{
%       \bbbbr{x+\int_0^t f(s,y^x(s))\,\diff s}
%       -
%       \bbbbr{x+\int_0^u f(s,y^x(s))\,\diff s}
%     }
%     \\&=
%     \bbbnorm{\int_u^t f(s,y^x(s))\,\diff s}
%     \\&\leq
%     \int_u^t \norm{f(s,y^x(s))}\,\diff s
%     .
%   \ea
%   Furthermore, observe that for all
%     $x,z\in\R^d$
%     $s\in[0,T]$
%     with $\norm{x-z}<\min\{1,r_z\}$
%   it holds that
%   \beq
%     \norm{y^x(s)}
%     \leq
%     \norm{y^z(s)}+\norm{y^x(s)-y^z(s)}
%     \leq 
%     R_z+L_z\norm{x-z}
%     \leq 
%     R_z+L_z
%     .
%   \eeq
%   Combining 
%     this 
%   with 
%     \eqref{eq:yxtcont2}
%   shows that  for all
%     $x,z\in\R^d$,
%     $t\in[0,T]$,
%     $u\in[0,t]$
%     with $\norm{x-z}<\min\{1,r_z\}$
%   it holds that
%   \beq
%     \norm{y^x(t)-y^x(u)}
%     \leq
%     C_{R_z+L_z}(t-u)
%     .
%   \eeq
  In the next step we observe that 
    \eqref{eq:yxdiffeq}
  implies that for all
    $w\in\R^d$,
    $t\in[0,T]$,
    $u\in[0,t]$
  it holds that
  \ba
  \label{eq:yxtcont3}
    \norm{y^w(t)-y^w(u)}
    &=
    \bbbnorm{
      \bbbbr{w+\int_0^t f(s,y^w(s))\,\diff s}
      -
      \bbbbr{w+\int_0^u f(s,y^w(s))\,\diff s}
    }
    \\&=
    \bbbnorm{\int_u^t f(s,y^w(s))\,\diff s}
    \\&\leq
    \int_u^t \norm{f(s,y^w(s))}\,\diff s
    \\&\leq
    \int_u^t C_{R_w}\,\diff s
    \\&=
    (t-u)C_{R_w}
    .
  \ea
    This,
    \eqref{eq:ycontx2},
    and the triangle inequality
    hence 
  prove that for all
    $v,w\in\R^d$,
    $t,u\in[0,T]$
      with $\norm{v-w}<r_w$
  it holds that
  \beq
    \norm{y^v(t)-y^w(u)}
    \leq
    \norm{y^v(t)-y^w(t)}+\norm{y^w(t)-y^w(u)}
    \leq
    L_w\norm{v-w}+C_{R_w}\abs{t-u}
    .
  \eeq
  Therefore, we obtain that for all
    $v,w\in\R^d$,
    $t,u\in[0,T]$,
    $\eps\in(0,\infty)$
      with $\norm{v-w}<\min\{r_w,(2L_w)^{-1}\eps\}$
      and $\abs{t-u}<(2C_{R_w}+1)^{-1}\eps$
    it holds that
  \beq
    \norm{y^v(t)-y^w(u)}
    \leq
    \frac\eps2+\frac\eps2
    =
    \eps
    .
  \eeq
  This establishes that
  \beq
  \label{eq:yxtcont}
    [0,T]\times\R^d\ni (t,x)\mapsto y^x(t)\in\R^d
  \eeq
  is a continuous function.
  Next note that there exist unique
    $v^{x,h}\in C([0,T],\R^d)$, $x,h\in\R^d$,
  such that for all
    $x,h\in\R^d$,
    $t\in[0,T]$
  it holds that
  \beq
  \label{eq:uniquesol}
    v^{x,h}(t)
    =
    h+\int_0^t D(s,y^x(s))(v^{x,h}(s))\,\diff s
    .
  \eeq
    This
  implies that for all
    $x,h,k\in\R^d$,
    $\lambda,\mu\in\R$,
    $t\in[0,T]$
  it holds that
  \ba
  \label{eq:linaux}
    &\lambda v^{x,h}(t)+\mu v^{x,k}(t)
    \\&=
    \lambda \bbbbr{h+\int_0^t D(s,y^x(s))(v^{x,h}(s))\,\diff s}+\mu \bbbbr{k+\int_0^t D(s,y^x(s))(v^{x,k}(s))\,\diff s}
    \\&=
    \lambda h+\mu k+\bbbbr{\int_0^t \lambda D(s,y^x(s))(v^{x,h}(s))+\mu D(s,y^x(s))(v^{x,k}(s))\,\diff s}
    \\&=
    \lambda h+\mu k+\bbbbr{\int_0^t D(s,y^x(s))\bp{\lambda v^{x,h}(s)+\mu v^{x,k}(s)}\,\diff s}
    .
  \ea
  Combining
    this
  with
    \eqref{eq:uniquesol}
  proves that for all
    $x,h,k\in\R^d$,
    $\lambda,\mu\in\R$,
    $t\in[0,T]$
  it holds that
  \beq
    v^{x,\lambda h+\mu k}(t)
    =
    \lambda v^{x,h}(t)+\mu v^{x,k}(t)
    .
  \eeq
  This shows that for all
    $x\in\R^d$,
    $t\in[0,T]$
  it holds that
  \beq
  \label{eq:vlinh}
    \R^d\ni h\mapsto v^{x,h}(t)\in\R^d
  \eeq
  is a linear function.
  Next observe that
    the fact that
      $[0,T]\times\R^d\ni (t,x)\mapsto D(t,x)\in\R^{d\times d}$ 
      is a continuous function
  implies that there exist
    $\delta^{\rho}_{\eps}\in(0,\infty)$, $\rho,\eps\in(0,\infty)$,
  such that for all
    $\rho,\eps\in(0,\infty)$,
    $t\in[0,T]$,
    $\theta\in\{x\in\R^d\colon \norm{x}\leq \rho\}$,
    $\vartheta\in\{x\in\R^d\colon\norm{x-\theta}\leq \delta^\rho_\eps\}$,
    $h\in\R^d$
  it holds that
  \beq
  \label{eq:Dcontx}
    \norm{D(t,\vartheta)h-D(t,\theta)h}
    \leq\eps\norm{h}
    .
  \eeq
  In addition, note that
    \eqref{eq:ycontx2}
  implies that for all
    $\rho,\eps\in(0,\infty)$,
    $z,x\in\R^d$,
    $t\in[0,T]$,
    $u\in[0,1]$
    with $\norm{x-z}< \min\{r_z,(L_z)^{-1}\delta^\rho_{\eps}\}$
  it holds that
  \beq
    \bnorm{\bbr{y^z(t)+u(y^x(t)-y^z(t))}-y^z(t)}
    =
    u\norm{y^x(t)-y^z(t)}
    \leq
    uL_z\norm{x-z}
    \leq
    \delta^\rho_{\eps}
    .
  \eeq
  Combining
    this
  with
    \eqref{eq:Dcontx}
  shows that for all
    $\eps\in(0,\infty)$,
    $z,x,h\in\R^d$,
    $t\in[0,T]$,
    $u\in[0,1]$
    with $\norm{x-z}<\min\{r_z,(L_z)^{-1}\delta^{R_z}_{\eps}\}$
  it holds that
  \ba
    \bnorm{D\bp{t,y^z(t)+u(y^x(t)-y^z(t))}h-D(t,y^z(t))h}
    &\leq
    \eps\norm{h}
    .
  \ea
    The triangle inequality
    %and \eqref{eq:Krfindy}
    therefore
  implies that for all
    $\eps\in(0,\infty)$,
    $z,x,h,\mathfrak h\in\R^d$,
    $t\in[0,T]$,
    $u\in[0,1]$
    with $\norm{x-z}< \min\{r_z,(L_z)^{-1}\delta^{R_z}_{\eps}\}$
  it holds that
  \ba
  \label{eq:Dineq}
    &\bnorm{D\bp{t,y^z(t)+u(y^x(t)-y^z(t))}h-D(t,y^z(t))\mf h}
    \\&\leq
    \bnorm{D\bp{t,y^z(t)+u(y^x(t)-y^z(t))}h-D(t,y^z(t))h}+\norm{D(t,y^z(t))h-D(t,y^z(t))\mf h}
    \\&\leq
    \eps\norm{h}+\norm{D(t,y^z(t))(h-\mf h)}
    \\&\leq
    \eps\norm{h}+K_{R_z}\norm{h-\mf h}
    .
  \ea
    The fundamental theorem of calculus
    and \eqref{eq:ycontx2}
    hence
  prove that for all
    $\eps\in(0,\infty)$,
    $z,k\in\R^d$,
    $t\in[0,T]$
    with $\norm{k}<\min\{r_z,(L_z)^{-1}\delta^{R_z}_{\eps}\}$
  it holds that
  \ba
    &\bnorm{f(t,y^{z+k}(t))-f(t,y^z(t))-D(t,y^z(t))(v^{z,k}(t))}
    \\&=
    \bnorm{
      \bbr{f\bp{t,y^{z}(t)+u(y^{z+k}(t)-y^{z}(t))}}_{u=0}^{u=1}
      -D(t,y^z(t))(v^{z,k}(t))
    }
    \\&=
    \biggl\lVert
      \bbbbr{\int_0^1 D\bp{t,y^{z}(t)+u(y^{z+k}(t)-y^{z}(t))}(y^{z+k}(t)-y^{z}(t))\,\diff u}
      \\&\qquad -D(t,y^z(t))(v^{z,k}(t))
    \biggr\rVert
    \\&=
    \bbbnorm{
      \int_0^1 D\bp{t,y^{z}(t)+u(y^{z+k}(t)-y^{z}(t))}(y^{z+k}(t)-y^{z}(t))-D(t,y^z(t))(v^{z,k}(t))\,\diff u
    }
    \\&\leq
    \int_0^1 \bnorm{D\bp{t,y^{z}(t)+u(y^{z+k}(t)-y^{z}(t))}(y^{z+k}(t)-y^{z}(t))-D(t,y^z(t))(v^{z,k}(t))}\,\diff u
    \\&\leq
    \int_0^1 \eps\norm{y^{z+k}(t)-y^{z}(t)}+K_{R_z}\norm{y^{z+k}(t)-y^{z}(t)-v^{z,k}(t)}\,\diff u
    \\&=
    \eps\norm{y^{z+k}(t)-y^{z}(t)}+K_{R_z}\norm{y^{z+k}(t)-y^{z}(t)-v^{z,k}(t)}
    \\&\leq
    \eps L_z\norm{k}+K_{R_z}\norm{y^{z+k}(t)-y^{z}(t)-v^{z,k}(t)}
    .
  \ea
  Combining
    this
  with
    \eqref{eq:yxdiffeq}
    and \eqref{eq:uniquesol}
  shows that for all
    $\eps\in(0,\infty)$,
    $z,k\in\R^d$,
    $t\in[0,T]$
    with $\norm{k}< \min\{r_z,(L_z)^{-1}\delta^{R_z}_{\eps}\}$
  it holds that
  \ba
    &\norm{y^{z+k}(t)-y^z(t)-v^{z,k}(t)}
    \\&=
    \biggl\lVert
      \bbbbr{z+k+\int_0^t f(s,y^{z+k}(s))\,\diff s}
      -\bbbbr{z+\int_0^t f(s,y^z(s))\,\diff s}
      \\&\qquad
      -\bbbbr{k+\int_0^t D(s,y^z(s))(v^{z,k}(s))\,\diff s}
    \biggr\rVert
    \\&=
    \bbbnorm{
      \int_0^t f(s,y^{z+k}(s))-f(s,y^z(s))-D(s,y^z(s))(v^{z,k}(s))\,\diff s
    }
    \\&\leq
    \int_0^t \bnorm{f(s,y^{z+k}(s))-f(s,y^z(s))-D(s,y^z(s))(v^{z,k}(s))}\,\diff s
    \\&\leq
    \int_0^t \eps L_z\norm{k}+K_{R_z}\norm{y^{z+k}(s)-y^{z}(s)-v^{z,k}(s)}\,\diff s
    \\&=
    t\eps L_z\norm{k}+\int_0^tK_{R_z}\norm{y^{z+k}(s)-y^{z}(s)-v^{z,k}(s)}\,\diff s
    \\&\leq
    T\eps L_z\norm{k}+\int_0^tK_{R_z}\norm{y^{z+k}(s)-y^{z}(s)-v^{z,k}(s)}\,\diff s
    .
  \ea
    Gronwall's integral inequality
      (see, e.g., Grohs et al.~\cite[Lemma~2.11]{GHJW}
        (with
          $\alpha\is T\eps L_z\norm{k}$,
          $\beta\is K_{R_z}$,
          $T\is T$,
          $f\is ([0,T]\ni s\mapsto\norm{y^{z+k}(s)-y^z(s)-v^{z,k}(s)}\in[0,\infty))$
          for
          $\eps\in(0,\infty)$,
          $z\in\R^d$,
          $k\in\{h\in\R^d\colon \norm h<\min\{r_z,(L_z)^{-1}\delta^{R_z}_{\eps}\}\}$
        in the notation of Grohs et al.~\cite[Lemma~2.11]{GHJW})%
      )
    and \eqref{eq:Krfindy}
    therefore
  show that for all
    $\eps\in(0,\infty)$,
    $z,k\in\R^d$,
    $t\in[0,T]$
    with $\norm{k}< \min\{r_z,(L_z)^{-1}\delta^{R_z}_{\eps}\}$
  it holds that
  \ba
  \label{eq:diffgron}
    \norm{y^{z+k}(t)-y^z(t)-v^{z,k}(t)}
    &\leq
    T\eps L_z\norm{k}\exp(K_{R_z}t)
    .
  \ea
%   In addition, note that for all
%     $z,k\in\R^d$
%   it holds that
%   \beq
%     \norm{y^{z+k}(0)-y^z(0)-v^{z,k}(0)}
%     =
%     \norm{z+k-z-k}
%     =
%     0
%     .
%   \eeq
  %
    This
%     and \eqref{eq:diffgron}
  establishes that for all
    $\eps\in(0,\infty)$,
    $z\in\R^d$,
    $k\in\R^d\setminus\{0\}$,
    $t\in[0,T]$
    with $\norm{k}< \min\bigl\{r_z,(L_z)^{-1}\delta^{R_z}_{(TL_z\exp(K_{R_z}t))^{-1}\eps}\bigr\}$
  it holds that
  \beq
    \frac{\norm{y^{z+k}(t)-y^z(t)-v^{z,k}(t)}}{\norm k}
    \leq
    \eps
    .
  \eeq
%   This implies that for all
%     $z\in\R^d$,
%     $t\in[0,T]$,
%     $\eps\in(0,\infty)$
%   there exists
%     $\eta\in(0,\infty)$
%   such that for all
%     $k\in\R^d\setminus\{0\}$
%     with $\norm{k}< \eta$
%   it holds that
%   \ba
%   \label{eq:diffquot}
%     \frac{\norm{y^{z+k}(t)-y^z(t)-v^{z,k}(t)}}{\norm k}
%     &\leq
%     \eps
%     .
%   \ea
  Therefore, we obtain that for all
    $z\in\R^d$,
    $t\in[0,T]$
  it holds that
  \ba
    \limsup_{\substack{k\to 0,\\k\in\R^d\setminus\{0\}}}\bbbbr{\frac{\norm{y^{z+k}(t)-y^z(t)-v^{z,k}(t)}}{\norm k}}
    =0
    .
  \ea
  Combining
    this
  with
    \eqref{eq:vlinh}
  shows that for all
    $t\in[0,T]$,
    $x,h\in\R^d$
  it holds that
    $\R^d\ni z\mapsto y^z(t)\in\R^d$
    is a differentiable function
    and
  \beq
  \label{eq:yztdiffble}
    \bp{\tfrac{\partial}{\partial x}y^x(t)}(h)
    =
    v^{x,h}(t)
    .
  \eeq
    This
    and \eqref{eq:uniquesol}
  establish~\ref{enum:lemauxdiff:2}.
  Next note that
    the triangle inequality
    and \eqref{eq:uniquesol}
  imply that for all
    $x,h\in\R^d$,
    $t\in[0,T]$
  it holds that
  \ba
    \norm{v^{x,h}(t)}
    &\leq
    \norm{h}+\bbbnorm{\int_0^t D(s,y^x(s))(v^{x,h}(s))\,\diff s}
    \\&\leq
    \norm{h}+\int_0^t \bnorm{D(s,y^x(s))(v^{x,h}(s))}\,\diff s
    \\&\leq
    \norm{h}+\int_0^t K_{R_x}\norm{v^{x,h}(s)}\,\diff s
    .
  \ea
    Gronwall's integral inequality
      (see, e.g., Grohs et al.~\cite[Lemma~2.11]{GHJW}
        (with
          $\alpha\is \norm{h}$,
          $\beta\is K_{R_x}$,
          $T\is T$,
          $f\is ([0,T]\ni s\mapsto\norm{v^{x,h}(s)}\in[0,\infty])$
          for
          $x,h\in\R^d$
        in the notation of Grohs et al.~\cite[Lemma~2.11]{GHJW})%
      )
    and \eqref{eq:Krfindy}
    hence
  ensure that for all
    $x,h\in\R^d$,
    $t\in[0,T]$
  it holds that
  \beq
    \label{eq:vxhbnd}
    \norm{v^{x,h}(t)}
    \leq
    \norm{h}\exp(K_{R_x}t)
    \leq
    \norm{h}\exp(K_{R_x}T)
    .
  \eeq
  In addition, observe that 
    \eqref{eq:ycontx2}
    and the triangle inequality
  imply that for all
    $x,z\in\R^d$,
    $t\in[0,T]$
    with $\norm{x-z}<\min\{1,r_z\}$
  it holds that
  \beq
    \norm{y^x(t)}
    %&=
    %\norm{y^x(t)-y^z(t)+y^z(t)}
    \leq
    \norm{y^x(t)-y^z(t)}+\norm{y^z(t)}
    \leq
    L_z\norm{x-z}+R_z
    \leq
    L_z+R_z
    .
  \eeq
  This ensures that for all
    $x,z\in\R^d$
    with $\norm{x-z}<\min\{1,r_z\}$
  it holds that
  \beq
    R_x
    \leq
    R_z+L_z.
  \eeq
  Combining
    this
  with
    \eqref{eq:vxhbnd}
  proves that for all
    $x,z,h\in\R^d$,
    $t\in[0,T]$
    with $\norm{x-z}<\min\{1,r_z\}$
  it holds that
  \beq
  \label{eq:vxhbnd2}
    \norm{v^{x,h}(t)}\leq \norm{h}\exp(K_{R_z+L_z}T).
  \eeq
  Next note that
    \eqref{eq:Dineq}
  implies that for all
    $\eps\in(0,\infty)$,
    $x,z,h\in\R^d$,
    $t\in[0,T]$
    with $\norm{x-z}<\min\{1,r_z,(L_z)^{-1}\delta_\eps^{R_z}\}$
  it holds that
  \ba
    \norm{v^{x,h}(t)-v^{z,h}(t)}
    &=
    \biggl\lVert
      \bbbbr{h+\int_0^t D(s,y^x(s))(v^{x,h}(s))\,\diff s}
      \\&\qquad-
      \bbbbr{h+\int_0^t D(s,y^z(s))(v^{z,h}(s))\,\diff s}
    \biggr\rVert
    \\&=
    \bbbnorm{
      \int_0^t D(s,y^x(s))(v^{x,h}(s))-D(s,y^z(s))(v^{z,h}(s))\,\diff s
    }
    \\&\leq
    \int_0^t \bnorm{D(s,y^x(s))(v^{x,h}(s))-D(s,y^z(s))(v^{z,h}(s))}\,\diff s
    \\&\leq
    \int_0^t \eps\norm{v^{x,h}(s)}+K_{R_z}\norm{v^{x,h}(s)-v^{z,h}(s)}\,\diff s
    \\&\leq
    \int_0^t \eps\norm{h}\exp(K_{R_z+L_z}T)+K_{R_z}\norm{v^{x,h}(s)-v^{z,h}(s)}\,\diff s
    \\&=
    \eps\norm{h}\exp(K_{R_z+L_z}T)+\int_0^t K_{R_z}\norm{v^{x,h}(s)-v^{z,h}(s)}\,\diff s
    .
  \ea
    This,
    Gronwall's integral inequality
      (see, e.g., Grohs et al.~\cite[Lemma~2.11]{GHJW}
        (with
          $\alpha\is \eps\norm{h}\exp(K_{R_z+L_z}T)$,
          $\beta\is K_{R_z}$,
          $T\is T$,
          $f\is ([0,T]\ni t\mapsto\norm{v^{x,h}(t)-v^{z,h}(t)}\in[0,\infty))$
        for
        $\eps\in(0,\infty)$,
        $z,h\in\R^d$,
        $x\in\{w\in\R^d\colon \norm{w-z}<\min\{1,r_z,(L_z)^{-1}\delta_\eps^{R_z}\}\}$
        in the notation of Grohs et al.~\cite[Lemma~2.11]{GHJW})%
      ),
    and \eqref{eq:Krfindy}
  show that for all
    $\eps\in(0,\infty)$,
    $z,x,h\in\R^d$,
    $t\in[0,T]$
    with $\norm{x-z}<\min\{1,r_z,(L_z)^{-1}\delta_\eps^{R_z}\}$
  it holds that
  \ba
  \label{eq:vxhcont1}
    \norm{v^{x,h}(t)-v^{z,h}(t)}
    &\leq
    \eps\norm{h}\exp(K_{R_z+L_z}T)\exp(K_{R_z}t)
    \\&\leq
    \eps\norm{h}\exp(K_{R_z+L_z}T)\exp(K_{R_z}T)
    \\&\leq
    \eps\norm{h}\exp(2K_{R_z+L_z}T)
    .
  \ea
  Moreover,
    \eqref{eq:uniquesol}
    and \eqref{eq:vxhbnd}
  show that for all
    $z,h\in\R^d$,
    $s\in[0,T]$,
    $t\in[0,s]$
  it holds that
  \ba
  \label{eq:vzhtcont}
     \norm{v^{z,h}(s)-v^{z,h}(t)}
     &=
     \biggl\lVert
       \bbbbr{h+\int_0^s D(u,y^{z}(u))(v^{z,h}(u))\,\diff u}
       \\&\qquad-
       \bbbbr{h+\int_0^t D(u,y^{z}(u))(v^{z,h}(u))\,\diff u}
     \biggr\rVert
     \\&=
     \bbbnorm{
       \int_t^s D(u,y^{z}(u))(v^{z,h}(u))\,\diff u
     }
     \\&\leq
     \int_t^s \bnorm{D(u,y^{z}(u))(v^{z,h}(u))}\,\diff u
     \\&\leq
     \int_t^s K_{R_z}\norm{v^{z,h}(u)}\,\diff u
     \\&\leq
     \int_t^s K_{R_z}\norm h\exp(K_{R_z}T)\,\diff u
     \\&=
     (s-t)K_{R_z}\norm h\exp(K_{R_z}T)
     .
  \ea
  Combining
    this
  with
    \eqref{eq:vxhcont1}
  proves that for all
    $\eps\in(0,\infty)$,
    $z,x,h\in\R^d$,
    $s,t\in[0,T]$
    with $\norm{x-z}<\min\bigl\{1,r_z,(L_z)^{-1}\delta_{\exp(-2K_{R_z+L_z}T)2^{-1}\eps}^{R_z}\bigr\}$
      and $\abs{s-t}<(2K_{R_z}\exp(K_{R_z}T)+1)^{-1}\eps$
  it holds that
  \ba
    \norm{v^{x,h}(s)-v^{z,h}(t)}
    &\leq
    \norm{v^{x,h}(s)-v^{z,h}(s)}+\norm{v^{z,h}(s)-v^{z,h}(t)}
    \\&\leq
    \frac{\eps\norm{h}}2+\frac{\eps\norm h}{2}
    \\&=
    \eps\norm h
    .
  \ea
  Combining
    this 
  with
    \eqref{eq:yztdiffble}
  establishes that
  \beq
    [0,T]\times\R^d\ni (t,x)\mapsto \bp{\tfrac{\partial}{\partial x}y^{x}(t)}\in\R^{d\times d}
  \eeq
    is a continuous function.
    This
    and \eqref{eq:yxtcont}
  prove \ref{enum:lemauxdiff:1}.
  The proof of Lemma~\ref{lem:auxdiff} is thus completed.
\end{proof}

\begin{lemma}
  \label{lem:diffy}
  Let $d,m\in\N$, 
  $T\in[0,\infty)$,
  $\mu\in C^1(\R^d,\R^d)$,
  %$\varphi\in C(\R^d,[0,\infty))$,
  $\sigma\in\R^{d\times m}$,
  %$V\in C^1(\R^d,[0,\infty))$,
  $w\in C([0,T],\R^m)$
%   satisfy for all $x,h\in\R^d$, $z\in\R^m$ that
%   \T{(probably, none of these conditions are needed)}
%   \begin{gather}
%   w(0)=0,\qquad  \norm{\mu'(x)(h)}\leq \kappa \bp{1+\norm x^p}\norm h,\qquad
%   V'(x)(\mu(x+\sigma z))\leq \varphi(z)V(x), \\\norm x\leq V(x),
%   \end{gather}
  and let $y^x\in C([0,T],\R^d)$, $x\in\R^d$, 
  satisfy for all $x\in\R^d$, $t\in[0,T]$ that
  \beq
  \label{eq:assy}
  y^x(t)=x+\int_0^t \mu(y^x(s))\,\diff s+\sigma w(t).
  \eeq
%   and assume for all $x\in\R^d$ that
%   $[0,T]\ni t\mapsto y^x(t)\in\R^d$ is continuous.
  Then 
  \begin{enumerate}[label=(\roman{enumi}),ref=(\roman{enumi})]
  \item[(i)] \label{enum:lemdiffy:1}
    it holds that
      $([0,T]\times \R^d\ni (t,x)\mapsto y^x(t)\in\R^d)\in C^{0,1}([0,T]\times \R^d,\R^d)$
    and 
  \item[(ii)] \label{enum:lemdiffy:2}
    it holds for all 
      $x,h\in\R^d$, 
      $t\in[0,T]$ 
    that
    \beq
    \label{eq:diffyconc}
      \bp{ \tfrac{\partial}{\partial x} y^x(t) }(h)=h+\int_0^t \mu'(y^x(s))\bp{\bp{\tfrac{\partial}{\partial x} y^x(s) }(h)}\,\diff s.
    \eeq
  \end{enumerate}
\end{lemma}
\begin{proof}[Proof of Lemma~\ref{lem:diffy}]
  Throughout this proof 
  let $f\colon[0,T]\times\R^d\to \R^d$ 
    \intrtype{be the function which satisfies }%
    \intrtypen{satisfy }%
    for all
      $t\in[0,T]$,
      $x\in\R^d$
    that
    \beq 
      \label{eq:deff}
      f(t,x)=\mu(x+\sigma w(t))
    \eeq
  and let $z^x\colon[0,T]\to\R^d$, $x\in\R^d$, 
    \intrtype{be the functions which satisfy }%
    \intrtypen{satisfy }%
    for all
      $x\in\R^d$, 
      $t\in [0,T]$ 
    that
    \beq 
      \label{eq:defzxb}
      z^x(t)=y^x(t)-\sigma w(t).
    \eeq
  Observe that 
    the hypothesis that $\mu\in C^1(\R^d,\R^d)$
    and the hypothesis that $w$ is a continuous function 
  show that for all 
    $t\in[0,T]$, 
    $x\in\R^d$ 
  it holds that
  \beq
  \begin{gathered}
    f\in C([0,T]\times\R^d,\R^d),\qquad
    (\R^d\ni v\mapsto f(t,v)\in\R^d)\in C^1(\R^d,\R^d),\\
    \text{and}\qquad \tfrac\partial{\partial x} f(t,x)=\mu'(x+\sigma w(t)).
    \label{eq:fC1}
  \end{gathered}
  \eeq
    The hypothesis that $\mu\in C^1(\R^d,\R^d)$
    and the hypothesis that $w$ is a continuous function 
    hence
  imply that
    $[0,T]\times\R^d\ni(t,x)\mapsto \tfrac\partial{\partial x}f(t,x)\in\R^{d\times d}$ is a continuous function.
  This and \eqref{eq:fC1} ensure that 
  \beq
  \label{eq:fC01}
    f\in C^{0,1}([0,T]\times \R^d,\R^d).
  \eeq
  Next we combine
    \eqref{eq:assy}
    and \eqref{eq:defzxb}
  to obtain that for all 
    $x\in\R^d$, 
    $t\in[0,T]$ 
  it holds that
  \beq
    \label{eq:zxinteq}
    z^x(t)=x+\int_0^t \mu(y^x(s))\,\diff s
    =x+\int_0^t \mu(z^x(s)+\sigma w(s))\,\diff s
    =x+\int_0^t f(s,z^x(s))\,\diff s.
  \eeq
  In addition, note that 
    the assumption that 
      for all 
        $x\in\R^d$ 
      it holds that 
        $y^x$ is a continuous function
    and the assumption that $w$ is a continuous function 
  imply that for all 
    $x\in\R^d$ 
  it holds that
  $z^x$ is a continuous function.
  Combining 
    this,
    \eqref{eq:fC1},
    \eqref{eq:fC01},
    and \eqref{eq:zxinteq}
  with 
    Lemma~\ref{lem:auxdiff}
    (with
      $d\is d$,
      $T\is T$,
      $f\is f$,
      $D\is \bp{[0,T]\times \R^d\ni(t,x)\mapsto \mu'(x+\sigma w(t))\in\R^{d\times d}}$,
      $(y^x)_{x\in\R^d}\is (z^x)_{x\in\R^d}$
    in the notation of Lemma~\ref{lem:auxdiff})
  shows 
  \begin{enumerate}[label=(\alph{enumi})]
  \item that 
  \beq
  \label{eq:zxC01}
    ([0,T]\times\R^d\ni(t,x)\mapsto z^x(t)\in\R^d)\in C^{0,1}([0,T]\times\R^d,\R^d)
  \eeq
  and 
  \item that for all 
    $x,h\in\R^d$, 
    $t\in[0,T]$ 
  it holds that
  \beq
  \label{eq:zxint3}
    \bp{\tfrac\partial{\partial x}z^x(t)}(h)
    =h+\int_0^t\mu'(z^x(s)+\sigma w(s))\bp{\bp{\tfrac{\partial}{\partial x} z^x(s)}(h)}\,\diff s
    .
  \eeq
  \end{enumerate}
  Observe that
    \eqref{eq:defzxb} 
    and \eqref{eq:zxC01}
  imply that
    $([0,T]\times\R^d\ni (t,x)\mapsto y^x(t)\in\R^d)\in C^{0,1}([0,T]\times\R^d,\R^d)$.
    This
    and \eqref{eq:defzxb}
  establish that for all
    $x\in\R^d$,
    $t\in[0,T]$
  it holds that
  \beq
    \tfrac\partial{\partial x}y^x(t)
    =
    \tfrac\partial{\partial x}z^x(t)
    .
  \eeq
  Combining
    this
  with
    \eqref{eq:zxint3}
  proves that for all 
    $x,h\in\R^d$, 
    $t\in[0,T]$ 
  it holds that
  \beq
    \bp{ \tfrac\partial{\partial x}y^x(t) }(h)
    =
    h+\int_0^t\mu'(y^x(s))\bp{\bp{\tfrac{\partial}{\partial x} y^x(s)}(h)}\,\diff s
    .
  \eeq
  %This establishes \eqref{eq:diffyconc}.
  The proof of Lemma \ref{lem:diffy} is thus completed.
\end{proof}

\subsection{Differentiability with respect to the initial value for SDEs}
\label{subsec:diffsdes}

\begin{lemma}
  \label{lem:existX}
  Let $d,m\in\N$,
  $T\in[0,\infty)$,
  $\mu\in C^1(\R^d,\R^d)$,
  $\sigma\in\R^{d\times m}$,
  $\varphi\in C(\R^m,[0,\infty))$,
  $V\in C^1(\R^d,[0,\infty))$,
  let $\norm{\cdot}\colon\R^d\to[0,\infty)$ be a norm,
  assume 
    for all 
      $x\in\R^d$, 
      $z\in\R^m$ 
    that
      $V'(x)\mu(x+\sigma z)\leq \varphi(z)V(x)$ and $\norm x\leq V(x)$,
  let $(\Omega,\mc F,\PP)$ be a probability space,
  and let $W\colon[0,T]\times\Omega\to\R^m$ be a stochastic process
    %on $(\Omega,\mc F,\PP)$ 
    with continuous sample paths.
  Then
  \begin{enumerate}[label=(\roman{enumi}),ref=(\roman{enumi})]
  \item \label{enum:lemexistX:1}
    there exist unique stochastic processes
      $X^x\colon [0,T]\times\Omega\to\R^d$, $x\in\R^d$, with continuous sample paths
    which satisfy for all 
      $x\in\R^d$,
      $t\in[0,T]$,
      $\omega\in\Omega$
    that
    \beq
      X^x(t,\omega)=x+\int_0^t\mu(X^x(s,\omega))\,\diff s+\sigma W(t,\omega),
    \eeq
  \item \label{enum:lemexistX:2a}
    it holds for all $\omega\in\Omega$ that
      $([0,T]\times\R^d\ni (t,x)\mapsto X^x(t,\omega)\in\R^d)\in C^{0,1}([0,T]\times\R^d,\R^d)$,
    and
  \item \label{enum:lemexistX:2b}
    it holds 
      for all 
        $x,h\in\R^d$,
        $t\in[0,T]$,
        $\omega\in\Omega$ 
      that
        \beq
          \bp{ \tfrac{\partial}{\partial x} X^x(t,\omega) }(h)=h+\int_0^t \mu'(X^x(s,\omega))\bp{\bp{\tfrac{\partial}{\partial x} X^x(s,\omega) }(h)}\,\diff s
          .
        \eeq
  \end{enumerate}
\end{lemma}
\begin{proof}[Proof of Lemma~\ref{lem:existX}]
  First, observe that 
    Lemma~\ref{lem:inteq} 
    (with
      $d\is d$,
      $m\is m$,
      $T\is T$,
      $\xi\is x$,
      $\mu\is\mu$,
      $\sigma\is\sigma$,
      $\varphi\is\varphi$,
      $V\is V$,
      $w\is([0,T]\ni t\mapsto W(t,\omega)\in\R^m)$,
      $\norm\cdot\is\norm\cdot$
      for $x\in\R^d$, $\omega\in\Omega$
    in the notation of Lemma~\ref{lem:inteq})
  proves that 
    there exist
      unique $y^x_\omega\in C([0,T],\R^d)$, $x\in\R^d$, $\omega\in\Omega$,
    such that for all
      $x\in\R^d$,
      $\omega\in\Omega$,
      $t\in[0,T]$
    it holds that
    \beq
      \label{eq:yxom}
      y^x_\omega(t)=x+\int_0^t\mu(y^x_\omega(s))\,\diff s+\sigma W(t,\omega).
    \eeq
  In addition, note that 
    the hypothesis that $\mu\in C^1(\R^d,\R^d)$
    %and the fact that all norms on $\R^d$ are equivalent
  ensures that 
%     $\mu$ is a locally Lipschitz continuous function.
  for all
    $r\in(0,\infty)$
  it holds that
  \beq
    \sup_{\substack{x,y\in\R^d,\,x\neq y,\\\norm x+\norm y\leq r}}\frac{\norm{\mu(x)-\mu(y)}}{\norm{x-y}}<\infty.
  \eeq
  Combining 
    this
    and \eqref{eq:yxom}
  with
    Lemma~\ref{lem:intmeas}
    (with
      $d\is d$,
      $T\is T$,
      $f\is \pp{[0,T]\times\R^d\ni(t,y)\mapsto\mu(y)\in\R^d}$,
      $\xi\is x$,
      $(\Omega,\mc F,\PP)\is(\Omega,\mc F,\PP)$,
      $W\is \pp{[0,T]\times\Omega\ni(t,\omega)\mapsto \sigma W(t,\omega)\in\R^d}$,
      $\norm\cdot\is\norm\cdot$,
      $Y\is\pp{[0,T]\times\Omega\ni (t,\omega)\mapsto y^x_\omega(t)\in\R^d}$
      for
      $x\in\R^d$
    in the notation of Lemma~\ref{lem:intmeas})
  shows that for all 
    $x\in\R^d$
  it holds that
    $[0,T]\times \Omega\ni(t,\omega)\mapsto y^x_\omega(t)\in\R^d$ 
    is a stochastic process.
    This 
    and \eqref{eq:yxom} 
  establish~\ref{enum:lemexistX:1}.
  Next note that
    \eqref{eq:yxom}
    and Lemma~\ref{lem:diffy}
    (with
      $d\is d$,
      $m\is m$,
      $T\is T$,
      $\mu\is\mu$,
      $\sigma\is\sigma$,
      $w\is([0,T]\ni t\mapsto W(t,\omega)\in\R^m)$,
      $(y^x)_{x\in\R^d}=(y^x_\omega)_{x\in\R^d}$
      for $\omega\in\Omega$
    in the notation of Lemma~\ref{lem:diffy})
  establishes~\ref{enum:lemexistX:2a}
  and~\ref{enum:lemexistX:2b}.
  The proof of Lemma~\ref{lem:existX} is thus completed.
\end{proof}

\section{Integrability properties for stochastic differential equations (SDEs)}\label{section6}

In this section we present in Lemma~\ref{lem:expbrownaux} in
Subsection~\ref{subsec:intbrown} below, in
Lemmas
\ref{lem:expbrown}%,
%\ref{lem:regcrithelp}, \ref{lem:expWr}, and 
--\ref{lem:expWr} in Subsection~\ref{subsec:intndimbrown} below,
and in Lemma~\ref{lem:regcrit} in Subsection~\ref{subsec:intsdes} below
a few elementary 
integrability properties for standard Brownian motions 
(see Lemmas~\ref{lem:expbrownaux} and Lemmas \ref{lem:expbrown}%,
%\ref{lem:regcrithelp}, and 
--\ref{lem:expWr}) and solutions of 
certain additive noise driven stochastic differential equations 
(see Lemma~\ref{lem:regcrit}). Lemma~\ref{lem:expbrownaux} establishes 
exponential integrability properties for 
one-dimensional standard Brownian motions and is a straightforward 
consequence of Ledoux-Talagrand~\cite[Corollary 3.2]{LedouxTalagrand}. 
Lemmas~\ref{lem:expbrown} and \ref{lem:regcrithelp}
establish exponential integrability properties for multi-dimensional 
standard Brownian motions. 
Our proof of Lemma~\ref{lem:expbrown} uses Lemma~\ref{lem:expbrownaux} 
and an application of the 
well-known inequality for real numbers in Lemma~\ref{lem:powsum} below. 
Lemma~\ref{lem:regcrithelp}, in turn, is an immediate consequence of 
Lemma~\ref{lem:expbrown}. 
Lemma~\ref{lem:expWr} establishes polynomial integrability properties for multi-dimensional 
standard Brownian motions and is a direct consequence of
Lemma~\ref{lem:regcrithelp}. Lemmas~\ref{lem:expbrownaux}
%, 
%and \ref{lem:expbrown}%, 
%\ref{lem:regcrithelp}, and 
--\ref{lem:expWr}
are essentially well-known and for the reader's convenience, 
we include in this section full proofs for these lemmas.

\subsection{Integrability properties for scalar Brownian motions}
\label{subsec:intbrown}

\begin{lemma}
  \label{lem:expbrownaux}
  Let $T,c\in[0,\infty)$,
  $\alpha\in[0,2)$,
  let $(\Omega,\mc F,\PP)$ be a probability space,
  and let $W\colon[0,T]\times\Omega\to\R$ be a Brownian motion
    %on $(\Omega,\mc F,\PP)$ 
    with continuous sample paths.
  Then 
  \begin{enumerate}[label=(\roman{enumi}),ref=(\roman{enumi})]
  \item 
    it holds that $\Omega\ni\omega\mapsto \sup_{t\in[0,T]}(\abs{W(t,\omega)}^\alpha)\in\R$ 
    is an $\mc F$/$\mc B(\R)$-measurable function and
  \item it holds that
  $
    \bEE{\exp\bp{c\bbr{\sup\nolimits_{t\in[0,T]}(\abs{W(t)}^\alpha)}}}
    <\infty
  $.
  \end{enumerate}
\end{lemma}
\begin{proof}[Proof of Lemma~\ref{lem:expbrownaux}]
  Throughout this proof
    assume w.l.o.g.\ that $T>0$ and $\alpha>0$
    and let $K\in[0,\infty)$ satisfy
    \beq
      K=\bbp{\frac{2-\alpha}2}\bbp{\frac2{4\alpha T}}^{\frac\alpha{\alpha-2}}c^{\frac2{2-\alpha}}
      .
    \eeq
  Note that
    the fact that $\frac{2-\alpha}2+\frac\alpha2=1$
    and the fact that 
      for all
        $a,b\in[0,\infty)$,
        $p,q\in(0,\infty)$ with $\frac1p+\frac1q=1$
      it holds that $ab\leq \frac{a^p}p+\frac{b^q}q$
      (Young inequality)
  implies that for all 
    $\omega\in\Omega$ 
  it holds that
  \ba
  \label{eq:young}
    &c\bbbbr{\sup_{t\in[0,T]}\bp{\abs{W(t,\omega)}^\alpha}}
    \\&=
    c\bbp{\frac{2}{4\alpha T}}^{-\frac\alpha2}\bbbp{\bbp{\frac{2}{4\alpha T}}^{\frac\alpha2}\bbbbr{\sup_{t\in[0,T]}\bp{\abs{W(t,\omega)}^\alpha}}}
    \\&\leq 
    \frac{2-\alpha}2\bbbbr{c\bbp{\frac{2}{4\alpha T}}^{-\frac\alpha2}}^{\frac2{2-\alpha}}
      +\frac{\alpha}{2}\bbbp{\bbp{\frac{2}{4\alpha T}}^{\frac\alpha2}\bbbbr{\sup_{t\in[0,T]}\bp{\abs{W(t,\omega)}^\alpha}}}^{\frac2\alpha}\\
    &=
    \frac{2-\alpha}2\bbp{\frac2{4\alpha T}}^{\frac\alpha{\alpha-2}}c^{\frac2{2-\alpha}}
      +\frac\alpha2\bbp{\frac{2}{4\alpha T}}\bbbbr{\sup_{t\in[0,T]}\bp{\abs{W(t,\omega)}^2}}\\
    &=
    K+(4T)^{-1}\bbbbr{\sup_{t\in[0,T]}\bp{\abs{W(t,\omega)}^2}}
    .
  \ea
  Furthermore, observe that
    Lemma~\ref{lem:supmeas1}
  ensures that for all
    $\beta\in[0,\infty)$
  it holds that
    \beq
    \label{eq:Wpowmeas}
    \Omega\ni\omega\mapsto \sup\nolimits_{t\in[0,T]}(\abs{W(t,\omega)}^\beta)\in\R
    \eeq
    is an $\mc F$/$\mc B(\R)$-measurable function.
  In addition, note that for all 
    $\kappa\in[0,\frac{1}{2T})$ 
  it holds that
  \beq
    \EE\bbr{\exp\bp{\kappa\bbr{ \sup\nolimits_{t\in[0,T]}\bp{\abs{W(t)}^2}}}}<\infty
  \eeq
  (cf., e.g., Ledoux-Talagrand~\cite[Corollary 3.2]{LedouxTalagrand}).
  Combining
    this
    and \eqref{eq:Wpowmeas}
  with
    \eqref{eq:young}
  establishes that
  \ba
    \bEE{\exp\bp{c\bbr{\sup\nolimits_{t\in[0,T]}\bp{\abs{W(t)}^\alpha}}}}
    &\leq
    \bEE{\exp\bp{K+(4T)^{-1}\sup\nolimits_{t\in[0,T]}\bp{\abs{W(t)}^2}}}
    \\&=
    \exp(K)\,\bEE{\exp\bp{(4T)^{-1}\sup\nolimits_{t\in[0,T]}\bp{\abs{W(t)}^2}}}
    <\infty
    .
  \ea
  This completes the proof of Lemma~\ref{lem:expbrownaux}.  
\end{proof}

\subsection{Integrability properties for multi-dimensional Brownian motions}
\label{subsec:intndimbrown}

\begin{lemma}
\label{lem:powsum}
  It holds for all 
  $\beta\in[0,\infty)$,
  $m\in\N$,
  $a_1,a_2,\dots,a_m\in\R$
  that
  \beq
    \label{eq:powsum}
    \bbabs{\sum\nolimits_{i=1}^m a_i}^\beta
    \leq 
    m^{\max\{0,\beta-1\}}\bbbr{\sum\nolimits_{i=1}^m\bp{\abs{a_i}^\beta}}
  .
  \eeq
\end{lemma}
\begin{proof}[Proof of Lemma~\ref{lem:powsum}]
  Throughout this proof
    let $\varphi_\beta\colon \R\to\R$, $\beta\in[1,\infty)$, be the functions 
      which satisfy for all
        $\beta\in[1,\infty)$,
        $x\in\R$
      that
    \beq
      \varphi_\beta(x)=\abs{x}^\beta.
    \eeq
  Note that for all
    $\beta\in[0,1]$,
    $m\in\N$,
    $a_1,a_2,\dots,a_m\in\R$
  it holds that
  \beq
  \label{eq:powsumlt1}
    \bbbr{\sum\nolimits_{i=1}^m\abs{a_i}}^\beta
    \leq 
    2^0\bbbr{\sum\nolimits_{i=1}^m\bp{\abs{a_i}^\beta}}
  .
  \eeq
  Next observe that for all
    $\beta\in[1,\infty)$
  it holds that
    $\varphi_\beta$ is a convex function.
    Jensen's inequality 
    hence 
  establishes that for all
    $m\in\N$,
    $a_1,\dots,a_m\in\R$
  it holds that
  \beq
    \frac{\babs{\sum_{i=1}^m a_i}^\beta}{m^\beta}
    =
    \varphi_\beta\bbbp{\frac{\sum_{i=1}^m a_i}{m}}
    \leq
    \frac{\sum_{i=1}^m \varphi_\beta(a_i)}{m}
    =
    \frac{\sum_{i=1}^m (\abs{a_i}^\beta)}{m}
    .
  \eeq
    This
  implies that for all
    $m\in\N$,
    $a_1,\dots,a_m\in\R$
  it holds that
  \beq
    \bbabs{\sum\nolimits_{i=1}^m a_i}^\beta
    \leq
    m^{\beta-1}\sum\nolimits_{i=1}^m(\abs{a_i}^\beta)
    .
  \eeq
  The proof of Lemma~\ref{lem:powsum} is thus completed.
\end{proof}

\begin{lemma}
\label{lem:expbrown}
  Let $m\in\N$, 
  $T,c\in[0,\infty)$, 
  $\alpha\in[0,2)$,
  let $(\Omega,\mc F,\PP)$ be a probability space,
  let $W\colon[0,T]\times\Omega\to\R^m$ be a standard Brownian motion
    %on $(\Omega,\mc F,\PP)$ 
    with continuous sample paths,
  and let $\norm{\cdot}\colon\R^m\to[0,\infty)$ be a norm.
  Then 
%   \begin{enumerate}[label=(\roman{enumi}),ref=(\roman{enumi})]
%     \item 
%       it holds that
%       $\Omega\ni\omega\mapsto \sup_{t\in[0,T]} (\norm{W(t,\omega)}^\alpha)\in[0,\infty)$
%       is an $\mc F$/$\mc B([0,\infty))$-measurable function and
%     \item 
%       \label{enum:lemexpbrown:2}
      it holds that
      \beq
      \label{eq:expbrownconc}
        \bbbEE{\sup_{t\in[0,T]\cap\Q} \exp\bp{c\,\norm{W(t)}^\alpha}}
%         =
%         \bbbEE{\exp\bbbp{c\bbbbr{\sup_{t\in[0,T]} (\norm{W(t)}^\alpha)}}}
        <
        \infty
        .
      \eeq
%   \end{enumerate}
\end{lemma}
\begin{proof}[Proof of Lemma~\ref{lem:expbrown}]
  Throughout this proof
  let $W_i\colon [0,T]\times\Omega\to\R$, $i\in\{1,2,\dots,m\}$,
  be the functions which satisfy for all $t\in [0,T]$, $\omega\in\Omega$ that
  \beq
  W(t,\omega)=\bp{W_1(t,\omega),W_2(t,\omega),\dots,W_m(t,\omega)}
  \eeq
  and let $K\in[0,\infty]$ satisfy
  \beq
  K
  =
  \sup_{z=(z_1,z_2,\dots,z_m)\in\R^m\setminus\{0\}}\bbbp{\frac{\norm{z}}{\bp{\sum_{i=1}^m\abs{z_i}}}}
  .
  \eeq
  Note that 
    the fact that all norms on $\R^m$ are equivalent
  ensures that $K<\infty$.
  Hence, we obtain that for all
    $\omega\in\Omega$
  it holds that
  \ba
  \sup_{t\in[0,T]}(\norm{W(t,\omega)}^\alpha)
  &\leq
  \sup_{t\in[0,T]}\bp{\bbr{K\bbr{{\textstyle\sum_{i=1}^m \abs{W_i(t,\omega)}}}}^\alpha}
  \\&=
  K^\alpha\bbbbr{\sup_{t\in[0,T]}\bp{\bbr{\textstyle\sum_{i=1}^m \abs{W_i(t,\omega)}}^\alpha}}
  .
  \ea
    This,
    the fact that $\alpha<2$,
    and Lemma~\ref{lem:powsum}
%     and the fact that 
%       for all 
%         $\beta\in[0,2)$,
%         $a_1,a_2,\dots,a_m\in\R$
%       it holds that\T{prove?}
%         $\bbr{\sum_{i=1}^m \abs{a_i}}^\beta\leq 2^m\bbr{\sum_{i=1}^m (\abs{a_i}^\beta)}$
  show that for all
    $\omega\in\Omega$
  it holds that
  \ba
    c\bbbbr{\sup_{t\in[0,T]}(\norm{W(t,\omega)}^\alpha)}
    &\leq
    cK^\alpha\bbbbr{\sup_{t\in[0,T]}\bp{m \bbr{\textstyle\sum_{i=1}^m (\abs{W_i(t,\omega)}^\alpha)}}}
    \\&=
    mcK^\alpha\bbbbr{\sup_{t\in[0,T]}\bp{{\textstyle\sum_{i=1}^m (\abs{W_i(t,\omega)}^\alpha})}}
    \\&\leq
    mcK^\alpha\bbbbbr{\sum_{i=1}^m\bbbp{\sup_{t\in[0,T]}\bp{\abs{W_i(t,\omega)}^\alpha}}}
    \\&=
    \sum_{i=1}^m\bbbp{mcK^\alpha\bbbbr{\sup_{t\in[0,T]}\bp{\abs{W_i(t,\omega)}^\alpha}}}
    .
  \ea
  Hence, we obtain that for all
    $\omega\in\Omega$
  it holds that
  \beq
  \label{eq:expWprod}
    \exp\bbbp{c\bbbbr{\sup_{t\in[0,T]}(\norm{W(t,\omega)}^\alpha)}}
    \leq
    \prod_{i=1}^m\exp\bbbp{mcK^\alpha\bbbbr{\sup_{t\in[0,T]}\bp{\abs{W_i(t,\omega)}^\alpha}}}
    .
  \eeq
  In the next step we note that
    Lemma~\ref{lem:supmeas1}
    (with
      $T\is T$,
      $(\Omega,\mc F,\PP)\is(\Omega,\mc F,\PP)$,
      $Y\is\pp{[0,T]\times\Omega\ni(t,\omega)\mapsto \norm{W(t,\omega)}^\alpha\in[0,\infty)}$
    in the notation of Lemma~\ref{lem:supmeas1})
  ensures 
%   \begin{enumerate}[label=(\alph{enumi})]
%     \item 
      that for all 
        $\omega\in\Omega$ 
      it holds that
      \beq
        \label{eq:WcapQ}
        \sup_{t\in[0,T]}(\norm{W(t,\omega)}^\alpha)
        =
        \sup_{t\in[0,T]\cap\Q}(\norm{W(t,\omega)}^\alpha).
      \eeq
%       and
%     \item
%       that
%       \beq
%         \label{eq:normWmeas}
%         \Omega\ni\omega\mapsto\bbbbr{\sup_{t\in[0,T]}(\norm{W(t,\omega)}^\alpha)}\in[0,\infty)
%       \eeq
%         is an $\mc F$/$\mc B([0,\infty))$-measurable function.
%   \end{enumerate}
  Combining 
    this 
  with 
    \eqref{eq:expWprod}
  shows that for all
    $\omega\in\Omega$
  it holds that
  \ba
  \label{eq:expWprod2}
    \sup_{t\in[0,T]\cap\Q}\exp\bp{c\,\norm{W(t,\omega)}^\alpha}
    &=
    \exp\bbbp{c\bbbbr{\sup_{t\in[0,T]\cap\Q}(\norm{W(t,\omega)}^\alpha)}}
    \\&\leq
    \prod_{i=1}^m\exp\bbbp{mcK^\alpha\bbbbr{\sup_{t\in[0,T]}\bp{\abs{W_i(t,\omega)}^\alpha}}}
    .
  \ea
  In addition, observe that
    Lemma~\ref{lem:expbrownaux}
    (with
      $T\is T$,
      $c\is 2^mcK^\alpha$,
      $\alpha\is\alpha$,
      $(\Omega,\mc F,\PP)\is(\Omega,\mc F,\PP)$,
      $W\is W_i$
      for
      $i\in\{1,2,\dots,m\}$
    in the notation of Lemma~\ref{lem:expbrownaux})
  proves 
  \begin{enumerate}[label=(\Alph{enumi})]
  \item 
    that for all
      $i\in\{1,2,\dots,m\}$
    it holds that
      $\Omega\ni\omega\mapsto \sup_{t\in[0,T]}\bp{\abs{W_i(t,\omega)}^\alpha}\in\R$
      is an $\mc F$/$\mc B(\R)$-measurable function
    and
  \item
    that for all
      $i\in\{1,2,\dots,m\}$
    it holds that
    \beq
    \label{eq:Eexpsup}
      \bbbEE{\exp\bbbp{mcK^\alpha\bbbbr{\sup_{t\in[0,T]}\bp{\abs{W_i(t)}^\alpha}}}}
      <\infty
      .
    \eeq
  \end{enumerate}
%   In addition, note that
%     Lemma~\ref{lem:supmeas1}
%     (with
%       $T\is T$,
%       $(\Omega,\mc F)\is(\Omega,\mc F)$,
%       $Y\is\bp{[0,T]\times\Omega\ni(t,\omega)\mapsto \abs{W_i(t,\omega)}^\beta\in[0,\infty)}$
%       for $i\in\{1,2,\dots,m\}$, $\beta\in[0,\infty)$
%     in the notation of Lemma~\ref{lem:supmeas1})
%   ensures that for all
%     $i\in\{1,2,\dots,m\}$,
%     $\beta\in[0,\infty)$
%   it holds that
%     $
%     \Omega\ni\omega\mapsto\sup_{t\in[0,T]}(\abs{W_i(t,\omega)}^\beta)\in[0,\infty)
%     $
%     is an $\mc F$/$\mc B([0,\infty))$-measurable function.
  Note that
    the fact that
      $W_1,W_2,\dots,W_m$ are independent stochastic processes,
    \eqref{eq:expWprod2},
    and \eqref{eq:Eexpsup}
  establish that
  \beq
  \label{eq:Eexpfin}
    \bbbEE{\sup_{t\in[0,T]\cap\Q}\exp\bp{c\,\norm{W(t)}^\alpha}}
    \leq
    \prod_{i=1}^m\bbbEE{\exp\bbbp{mcK^\alpha\bbbbr{\sup_{t\in[0,T]}\bp{\abs{W_i(t)}^\alpha}}}}
    <\infty
    .
  \eeq
  The proof of Lemma~\ref{lem:expbrown} is thus completed.
\end{proof}

\begin{lemma}
  \label{lem:regcrithelp}
  Let $m\in\N$, 
  $T,C\in[0,\infty)$, 
  $\alpha\in[0,2)$,
  let $\norm{\cdot}\colon\R^m\to[0,\infty)$ be a norm,
  let $(\Omega,\mc F,\PP)$ be a probability space,
  let $W\colon[0,T]\times\Omega\to\R^m$ be a standard Brownian motion
    %on $(\Omega,\mc F,\PP)$ 
    with continuous sample paths,
  and let $\varphi\colon \R^m\to[0,\infty)$ satisfy 
    for all 
      $z\in\R^m$ 
    that
    $
    \varphi(z)\leq C(1+\norm z^\alpha)
    $.
  Then it holds for all 
    $c\in[0,\infty)$ 
  that
  \beq
  \label{eq:expphiWbndc}
    \bbbEE{\sup_{t\in[0,T]\cap\Q}\exp\bp{c\,\varphi(W(t))}}<\infty.
  \eeq
\end{lemma}
\begin{proof}[Proof of Lemma~\ref{lem:regcrithelp}]
  Note that
    the assumption that 
      for all 
        $z\in\R^m$ 
      it holds that
        $\varphi(z)\leq C(1+\norm z^\alpha)$
  implies that for all
    $c\in[0,\infty)$,
    $\omega\in\Omega$
  it holds that
  \ba
  \label{eq:bndexpphiW}
    \sup_{t\in[0,T]\cap\Q}\exp\bp{c\,\varphi(W(t,\omega))}
    &\leq 
    \sup_{t\in[0,T]\cap\Q}\exp\bp{cC(1+\norm{W(t,\omega)}^\alpha)}\\
    &=
    \sup_{t\in[0,T]\cap\Q}\exp\bp{cC+cC\,\norm{W(t,\omega)}^\alpha}
    \\&=
    \exp(cC)\bbbbr{\sup_{t\in[0,T]\cap\Q}\exp\bp{cC\,\norm{W(t,\omega)}^\alpha}}
    .
  \ea
  Combining
    this
  with 
    Lemma~\ref{lem:expbrown}
  establishes that for all
    $c\in[0,\infty)$
  it holds that
  \beq
    \bbbEE{\sup_{t\in[0,T]\cap\Q}\exp\bp{c\,\varphi(W(t))}}
    \leq
    \exp(cC)\,\bbbEE{\sup_{t\in[0,T]\cap\Q}\exp\bp{cC\,\norm{W(t)}^\alpha}}
    <\infty
    .
  \eeq
  This completes the proof of Lemma~\ref{lem:regcrithelp}.
\end{proof}

\begin{lemma}
\label{lem:expWr}
  Let $d,m\in\N$, 
  $T,r\in[0,\infty)$,
  $\sigma\in\R^{d\times m}$,
  let $\norm{\cdot}\colon\R^d\to[0,\infty)$ be a norm,
  let $(\Omega,\mc F,\PP)$ be a probability space,
  and let $W\colon[0,T]\times\Omega\to\R^m$ be a standard Brownian motion
    %on $(\Omega,\mc F,\PP)$ 
    with continuous sample paths.
  Then it holds that
  \beq
  \label{eq:expWr}
    \bbbEE{\sup_{t\in[0,T]\cap\Q} (\norm{\sigma W(t)}^r)}<\infty.
  \eeq
\end{lemma}
\begin{proof}[Proof of Lemma~\ref{lem:expWr}]
  Throughout this proof 
  let $\nnorm{\cdot}\colon\R^m\to[0,\infty)$ be the $m$-dimensional Euclidean norm
  and let $C\in[0,\infty]$ satisfy
  \beq
    C
    =
    \sup_{x\in\R^m\setminus\{0\}}\bbbp{\frac{\norm{\sigma x}}{\nnorm{x}}}
    .
  \eeq
  Note that
  \beq
    C
    =
    \sup_{x\in\R^m\setminus\{0\}}\,\bbbnorm{\sigma\bbbp{\frac{x}{\nnorm{x}}}}
    \leq
    \sup_{y\in\{v\in\R^m\colon \nnorm{v}=1\}}\norm{\sigma y}
    .
  \eeq
    The fact that $\R^m\ni y\mapsto \norm{\sigma y}\in[0,\infty)$ is a continuous function
    and the fact that $\{v\in\R^m\colon \nnorm{v}=1\}$ is a compact set
    hence
  prove that
  \beq
  \label{eq:Cfin}
    C<\infty.
  \eeq
  In the next step we observe that for all
    $t\in[0,T]$,
    $\omega\in\Omega$
  it holds that
  \ba
    \label{eq:Zbnd}
    \norm{\sigma W(t,\omega)}^r
    \leq
    [1+\norm{\sigma W(t,\omega)}]^r
    =
    \exp\bp{r\ln(1+\norm{\sigma W(t,\omega)})}
    .
  \ea
  Furthermore, note that 
    \eqref{eq:Cfin}
    and the fact that
      for all
        $y\in[0,\infty)$
      it holds that
        $\ln(1+y)\leq y$
  ensure that for all
    $z\in\R^m$
  it holds that
  \beq
    \ln(1+\norm{\sigma z})
    \leq 
    \norm{\sigma z}
    \leq 
    C \nnorm{z}
    \leq 
    C(1+\nnorm{z}).
  \eeq
    This,
    \eqref{eq:Cfin},
    and Lemma~\ref{lem:regcrithelp}
      (with
        $m\leftarrow m$,
        $T\leftarrow T$,
        $C\leftarrow C$,
        $\alpha\leftarrow 1$,
        $\norm{\cdot}\is \nnorm\cdot$,
        $(\Omega,\mc F,\PP)\is(\Omega,\mc F,\PP)$,
        $W\is W$,
        $\varphi\is \pp{\R^m\ni z\mapsto \ln(1+\norm{\sigma z})\in[0,\infty)}$
      in the notation of Lemma~\ref{lem:regcrithelp})
  show that
  \beq
    \bbbEE{\sup_{t\in[0,T]\cap\Q} \exp\bp{r\ln(1+\norm{\sigma W(t)})}}<\infty.
  \eeq
  Combining
    this
  with
    \eqref{eq:Zbnd}
  establishes that
  \beq
    \bbbEE{\sup_{t\in[0,T]\cap\Q} (\norm{\sigma W(t)}^r)}
    \leq
    \bbbEE{\sup_{t\in[0,T]\cap\Q} \exp\bp{r\ln(1+\norm{\sigma W(t)})}}
    <\infty
    .
  \eeq
  This completes the proof of Lemma~\ref{lem:expWr}.
\end{proof}

\subsection{Integrability properties for solutions of SDEs}
\label{subsec:intsdes}

\begin{lemma}
  \label{lem:regcrit}
  Let $d,m\in\N$, 
  $T\in[0,\infty)$,
  $\mu\in C^1(\R^d,\R^d)$,
  $\sigma\in\R^{d\times m}$,
  $\varphi\in C(\R^m,[0,\infty))$,
  $V\in C^1(\R^d,[0,\infty))$,
  let $\norm{\cdot}\colon\R^d\to[0,\infty)$ be a norm,
  assume for all $x\in\R^d$, $z\in\R^m$ that
  $
  %\norm{\mu'(x)(h)}\leq \kappa \bp{1+\norm x^p}\norm h,\qquad
  V'(x)\mu(x+\sigma z)\leq \varphi(z)V(x)$ and $\norm x\leq V(x)$,
  let $(\Omega,\mc F,\PP)$ be a probability space,
  let $W\colon[0,T]\times\Omega\to\R^m$ be a stochastic process
    %on $(\Omega,\mc F,\PP)$ 
    with continuous sample paths,
  let $X^x\colon[0,T]\times\Omega\to\R^d$, $x\in\R^d$, be 
  %a family of $(\FF_t)_{t\in[0,T]}$-adapted 
  stochastic processes with continuous sample paths,
  assume for all $c\in[0,\infty)$ that
  $\bEE{
    \sup_{t\in[0,T]\cap\Q}\exp\bp{c\,\varphi(W(t))}
  }
  +\bEE{
    \sup_{t\in[0,T]\cap\Q}(\norm{\sigma W(t)}^c)
  }
  <\infty$,
  and assume for all 
    $x\in\R^d$, 
    $t\in[0,T]$, 
    $\omega\in\Omega$
  that
  \beq
  X^x(t,\omega)=x+\int_0^t\mu(X^x(s,\omega))\,\diff s+\sigma W(t,\omega).
  \eeq
  Then 
  \begin{enumerate}[label=(\roman{enumi}),ref=(\roman{enumi})]
  \item
    it holds for all 
      $R,r\in[0,\infty)$ 
    that
      \beq
      \Omega\ni\omega\mapsto \bbbbr{
        \sup_{x\in\{z\in\R^d\colon\norm z\leq R\}}\,
          \sup_{t\in[0,T]}
            \bp{\norm{X^x(t,\omega)}^r}
      }\in[0,\infty]
      \eeq
        is an $\mc F$/$\mc B([0,\infty])$-measurable function
    and
  \item
    it holds for all 
      $R,r\in[0,\infty)$ 
    that
    \beq
      \bbbEE{ \sup_{x\in\{z\in\R^d\colon \norm{z}\leq R\}}\,\sup_{t\in[0,T]} \bp{\norm{X^x(t)}^{r}} }
      <\infty.
    \eeq
  \end{enumerate}
\end{lemma}
\begin{proof}[Proof of Lemma~\ref{lem:regcrit}]
  Throughout this proof 
    let $Y,\,Z\colon\Omega\to[0,\infty)$
    \intrtype{be the functions which }%
  satisfy for all
    $\omega\in\Omega$
  that
  \beq
    Y(\omega)=\sup_{t\in[0,T]}\exp\bp{\varphi(W(t,\omega))}\qqandqq
    Z(\omega)=\sup_{t\in[0,T]} \norm{\sigma W(t,\omega)}.
  \eeq
  Note that 
    Lemma~\ref{lem:bndy} 
      (with
      $d\is d$,
      $m\is m$,
      $T\is T$,
      $\xi\is x$, 
      $\mu\is\mu$,
      $\sigma\is\sigma$,
      $\varphi\is\varphi$,
      $V\is V$,
      $\norm\cdot\is\norm\cdot$,
      $J\is [0,T]$, 
      $y\is \pp{[0,T]\ni t\mapsto X^x(t,\omega)\in\R^d}$, 
      $w\is \pp{[0,T]\ni t\mapsto W(t,\omega)\in\R^m}$
      for
      $x\in\R^d$, 
      $\omega\in\Omega$
      in the notation of Lemma~\ref{lem:bndy}) 
  ensures that for all 
    $x\in\R^d$, 
    $\omega\in\Omega$
  it holds that
%     $\sup_{t\in[0,T]}[\varphi(W(t,\omega))+\norm{\sigma W(t,\omega)}]$
%     and
  \ba
    \sup_{t\in[0,T]}\norm{X^x(t,\omega)}
    &\leq 
    V(x)\exp\bbbp{ T\bbbbr{\sup_{t\in[0,T]} \varphi(W(t,\omega))}}+\bbbbr{\sup_{t\in[0,T]} \norm{\sigma W(t,\omega)} }
    \\&=
    V(x)\bbbbr{\sup_{t\in[0,T]}\exp\bp{  \varphi(W(t,\omega))}}^T+\bbbbr{\sup_{t\in[0,T]} \norm{\sigma W(t,\omega)} }
    \\&=
    V(x)[Y(\omega)]^T+Z(\omega)
    .
  \ea
    The hypothesis that for all
      $x\in\R^d$,
      $\omega\in\Omega$
    it holds that
      $[0,T]\ni t\mapsto X^x(t,\omega)\in\R^d$ is a continuous function
    and the fact that
      for all
        $a,b\in\R$,
        $r\in[0,\infty)$
      it holds that
        $\abs{a+b}^r\leq 2^r(\abs{a}^r+\abs{b}^r)$
    hence 
  ensure that for all
    $\omega\in\Omega$,
    $R,r\in[0,\infty)$
  it holds that
  \ba
    \label{eq:bndsupsup}
    &\sup_{x\in\{z\in\R^d\colon \norm{z}\leq R\}}\sup_{t\in[0,T]}\bp{\norm{X^x(t,\omega)}^r}\\
    &=
    \sup_{x\in\{z\in\R^d\colon \norm{z}\leq R\}}\bbbp{\bbbbr{\sup_{t\in[0,T]}\norm{X^x(t,\omega)}}^r}\\
    &\leq 
    \sup_{x\in\{z\in\R^d\colon \norm{z}\leq R\}}\bp{2^r\bp{\br{V(x)}^r\br{Y(\omega)}^{Tr}+\br{Z(\omega)}^r}}\\
    &=
    2^r\bp{\bbr{\sup\nolimits_{x\in\{z\in\R^d\colon \norm{z}\leq R\}} V(x)}^r\br{Y(\omega)}^{Tr}+\br{Z(\omega)}^r}
    .
  \ea
  Next we combine
    the assumption that 
      for all 
        $\omega\in\Omega$ 
      it holds that 
        $[0,T]\ni t\mapsto W(t,\omega)\in\R^m$ is a continuous function
    and the assumption that 
      $\varphi$ is a continuous function
    with Lemma~\ref{lem:supmeas1}
  to obtain that 
  \begin{enumerate}[label=(\alph{enumi}),ref=(\alph{enumi})]
  \item
    for all 
      $\omega\in\Omega$ 
    it holds that
    \beq
    \label{eq:YZalt}
      Y(\omega)=\sup_{t\in[0,T]\cap\Q}\exp\bp{\varphi(W(t,\omega))}
      \qqandqq
      Z(\omega)=\sup_{t\in[0,T]\cap \Q} \norm{\sigma W(t,\omega)}
    \eeq
  and
  \item
    it holds that $Y$ and $Z$ are $\mc F$/$\mc B([0,\infty))$-measurable functions.
  \end{enumerate}
  Moreover, note that
    Lemma~\ref{lem:diffy}
    (with
      $d\is d$,
      $m\is m$,
      $T\is T$,
      $\mu\is \mu$,
      $\sigma\is\sigma$,
      $w\is ([0,T]\ni t\mapsto W(t,\omega)\in\R^m)$,
      $(y^x)_{x\in\R^d}\is([0,T]\ni t\mapsto X^x(t,\omega)\in\R^d)_{x\in\R^d}$
      for $\omega\in\Omega$
    in the notation of Lemma~\ref{lem:diffy})
  ensures that for all 
    $\omega\in\Omega$
  it holds that
  \beq
    ([0,T]\times\R^d\ni(t,x)\mapsto X^x(t,\omega)\in\R^d)\in C^{0,1}([0,T]\times\R^d,\R^d)
    .
  \eeq
  Combining
    this
    with Lemma~\ref{lem:supmeas}
  shows that for all
    $R,r\in[0,\infty)$
  it holds that
  \beq
  \label{eq:supsupmeas}
    \Omega\ni\omega\mapsto \bbbbr{
      \sup_{x\in\{z\in\R^d\colon\norm z\leq R\}}\,
        \sup_{t\in[0,T]}\bp{\norm{X^x(t,\omega)}^r}
    }\in[0,\infty]
  \eeq
  is an $\mc F$/$\mc B([0,\infty])$-measurable function.
  In the next step we observe that
    the assumption that 
      for all 
        $c\in[0,\infty)$
      it holds that
        $\bEE{\sup_{t\in[0,T]\cap\Q}\exp\bp{c\,\varphi(W(t))}}<\infty$,
    \eqref{eq:YZalt},
    and the fact that $Y$ is an $\mc F$/$\mc B([0,\infty))$-measurable function
  ensure that for all
    $r\in[0,\infty)$
  it holds that
  \beq
    \label{eq:YTrfin}
    \bEE{ Y^{Tr} }
    =
    \bbbEE{\sup_{t\in[0,T]\cap\Q} \exp\bp{Tr\,\varphi(W(t))} }
    <\infty.
  \eeq
  In addition, note that
    the hypothesis that $V$ is a continuous function
  implies that for all
    $R\in[0,\infty)$
  it holds that
  \beq
    \label{eq:supVfin}
    \sup_{x\in\{z\in\R^d\colon \norm{z}\leq R\}} V(x)<\infty.
  \eeq
  Furthermore, observe that
    \eqref{eq:YZalt},
    the fact that $Z$ is an $\mc F$/$\mc B([0,\infty))$-measurable function,
    and the hypothesis that 
      for all 
        $c\in[0,\infty)$ 
      it holds that
        $\bEE{
          \sup_{t\in[0,T]\cap\Q}(\norm{\sigma W(t)}^c)
        }
        <\infty$
  show that for all
    $r\in[0,\infty)$
  it holds that
  \beq
  \EE[Z^r]
  =\bbbEE{\sup_{t\in[0,T]\cap \Q} \bp{\norm{\sigma W(t)}^r}}
  <\infty
  .
  \eeq
  Combining
    this,
    \eqref{eq:supsupmeas}, 
    \eqref{eq:YTrfin},
    and \eqref{eq:supVfin}
  with
    \eqref{eq:bndsupsup}
  implies that for all
    $R,r\in[0,\infty)$
  it holds that
  \ba
    &\bbbEE{\sup_{x\in\{z\in\R^d\colon \norm{z}\leq R\}}\sup_{t\in[0,T]}\bp{\norm{X^x(t)}^r}}
    \\&\leq 
    2^r\bp{\bbr{\sup\nolimits_{x\in\{z\in\R^d\colon \norm{z}\leq R\}} V(x)}^r\bEE{Y^{Tr}}+\bEE{Z^r}}
    <\infty
    .
  \ea
  This completes the proof of Lemma~\ref{lem:regcrit}.
\end{proof}

\section{Conditional regularity with respect to the initial value for SDEs}\label{section7}

In this section we study in Lemmas~\ref{lem:lem11} and 
\ref{lem:lem1} in Subsection~\ref{subsec:condsubhoelder} below 
regularity properties of solutions of 
certain additive noise driven SDEs with respect to their initial values. 
In particular, in Lemma~\ref{lem:lem1} we establish in inequality \eqref{eq:lem1conc} 
a quantitative 
estimate for the difference of two solutions of certain additive noise driven SDEs. 
Our proof of Lemma~\ref{lem:lem1} is based on an application of Lemma~\ref{lem:lem11} which 
establishes a similar statement in wider generality. 
Our proof of Lemma~\ref{lem:lem11}, in turn, uses, besides other arguments, the auxiliary 
results in Lemma~\ref{lem:distlem} in Subsection~\ref{subsec:loclipde} 
below and in Lemmas~\ref{lem:faux2} and \ref{lem:faux} in Subsection~\ref{subsec:condsubhoelder} 
below.
For completeness we include in this section also the detailed proofs for 
the two elementary results in Lemmas~\ref{lem:faux2} and \ref{lem:faux}.

\subsection{Conditional local Lipschitz continuity for deterministic DEs}
\label{subsec:loclipde}

\begin{lemma}
\label{lem:distlem}
  Let $d\in\N$, 
  $T\in[0,\infty)$,
  $\varphi\in C(\R^d,[0,\infty))$,
  let $\norm{\cdot}\colon\R^d\to[0,\infty)$ be a norm,
  let $z^x\in C([0,T],\R^d)$, $x\in\R^d$, be functions
  which satisfy
    for all 
      $t\in[0,T]$ 
    that 
      $(\R^d\ni x\mapsto z^x(t)\in\R^d)\in C^1(\R^d,\R^d)$,
  and assume for all
    $x\in\R^d$,
    $t\in[0,T]$,
    $h\in\bigl\{v\in \R^d \colon \bp{[0,T]\ni s\mapsto \bp{\tfrac{\partial}{\partial x} z^x(s) }(v)\in\R^d}\text{ is a $\mc B([0,T])$/$\mc B(\R^d)$-mea\-sur\-able function}\bigr\}$
  that
    $\int_0^T \bnorm{\bp{ \tfrac{\partial}{\partial x} z^x(s) }(h)}\,\diff s<\infty$
    and
    \beq
    \label{eq:bndDxz}
      \bnorm{\bp{ \tfrac{\partial}{\partial x} z^x(t) }(h)}
      \leq 
      \norm h+\int_0^t\varphi(z^x(s))\,\bnorm{\bp{\tfrac\partial{\partial x}z^x(s)}(h)}\,\diff s.
    \eeq
  Then it holds for all
    $x,y\in\R^d$,
    $t\in[0,T]$
  that
  \beq
    \norm{z^x(t)-z^y(t)}
    \leq
    \sup_{u\in[0,1]} \bbbbr{\norm{x-y}\exp\bbbp{T\bbbbr{ \sup_{s\in[0,T]}\varphi(z^{(1-u)y+ux}(s)) }}}
    .
  \eeq
\end{lemma}
\begin{proof}[Proof of Lemma~\ref{lem:distlem}]
  Throughout this proof 
    let $D^x\colon [0,T]\to\R^{d\times d}$,
      $x\in\R^d$,
    \intrtype{be the functions which satisfy }%
    \intrtypen{satisfy }%
    for all
      $x\in\R^d$,
      $t\in[0,T]$
    that
      \beq
      D^x(t)=\tfrac{\partial}{\partial x}z^x(t)
      .
      \eeq
  Note that 
    the assumption that
      for all 
        $t\in[0,T]$
      it holds that 
        $(\R^d\ni x\mapsto z^x(t)\in\R^d)\in C^1(\R^d,\R^d)$
  ensures that for all
    $t\in[0,T]$,
    $h\in\R^d$
  it holds that
    $\R^d\ni x\mapsto D^x(t)h\in\R^{d}$ is a continuous function.
    The fundamental theorem of calculus
    hence
  implies that for all 
    $x,y\in\R^d$, 
    $t\in[0,T]$
  it holds that
  \ba
  \label{eq:dist}
    \norm{z^{x}(t)-z^{y}(t)}
    &=
    \bnorm{\bbr{ z^{(1-u)y+ux}(t) }_{u=0}^{u=1}}
    \\&=
    \bbbnorm{\int_0^1 D^{(1-u)y+ux}(t)(x-y)\,\diff u}
    \\&\leq
    \int_0^1 \norm{D^{(1-u)y+ux}(t)(x-y)}\,\diff u
    \\&\leq
    \sup_{u\in[0,1]}\norm{D^{(1-u)y+ux}(t)(x-y)}
    .
  \ea
  Moreover, observe that 
    the fact that 
      for all 
        $x\in\R^d$
      it holds that
        $z^x$ is a continuous function,
    the fact that
      for all
        $t\in[0,T]$
      it holds that
        $(\R^d\ni x\mapsto z^x(t)\in\R^d)\in C^1(\R^d,\R^d)$,
    and Lemma~\ref{lem:diffmeas}
    (with
      $d\is d$,
      $x\is x$,
      $h\is h$,
      $(\Omega,\mc F)\is([0,T],\mc B([0,T]))$,
      $(y^u)_{u\in\R^d}\is (z^u)_{u\in\R^d}$
      for
      $x,h\in\R^d$
    in the notation of Lemma~\ref{lem:diffmeas})
  show that
    for all
      $x,h\in\R^d$
    it holds that
    \beq
    \label{eq:zxtmeas}
      [0,T]\ni t\mapsto\bp{ \tfrac\partial{\partial x} z^x(t) }(h)\in\R^d
    \eeq
    is a $\mc B([0,T])$/$\mc B(\R^d)$-measurable function.
    This
    and the hypothesis that for all
      $x\in\R^d$,
      $h\in\bigl\{v\in \R^d \colon \bp{[0,T]\ni t\mapsto \bp{\tfrac{\partial}{\partial x} z^x(t) }(v)\in\R^d}\text{ is a $\mc B([0,T])$/$\mc B(\R^d)$-mea\-sur\-able function}\bigr\}$
    it holds that
      $\int_0^T \bnorm{\bp{ \tfrac{\partial}{\partial x} z^x(s) }(h)}\,\diff s<\infty$
  implies that 
    for all
      $x,h\in\R^d$
    it holds that
    \beq
    \label{eq:intDxtfin}
      \int_0^T \norm{D^x(s)h}\,\diff s<\infty
      .
    \eeq
  In addition, observe that
    the hypothesis that
      for all
        $w\in\R^d$
      it holds that
        $z^w$ is a continuous function
    and the hypothesis that 
      $\varphi$ is a continuous function
  ensure that for all
    $w\in\R^d$
  it holds that
  \beq
    \label{eq:suppathfin}
    \sup_{s\in[0,T]}\varphi(z^w(s))<\infty.
  \eeq
  This,
    \eqref{eq:bndDxz},
    and \eqref{eq:zxtmeas}
  ensure that for all
    $w,h\in\R^d$,
    $t\in[0,T]$
  it holds that
  \ba
  \label{eq:Dint}
    \norm{D^w(t)(h)}
    &\leq 
    \norm h+\int_0^t\varphi(z^w(s))\,\norm{D^w(s)h}\,\diff s\\
    &\leq 
    \norm h+\bbbbr{\sup_{s\in[0,T]}\varphi(z^w(s))}\int_0^t\norm{D^w(s)h}\,\diff s
    .
  \ea
  Combining
    this,
    \eqref{eq:intDxtfin},
    and \eqref{eq:suppathfin}
  with
    Gronwall's integral inequality
      (see, e.g., Grohs et al.~\cite[Lemma~2.11]{GHJW}
        (with
          $\alpha\is \norm h$,
          $\beta\is \sup_{s\in[0,T]}\varphi(z^w(s))$,
          $T\is T$,
          $f\is ([0,T]\ni s\mapsto \norm{D^w(s)h}\in[0,\infty))$
        for
        $w,h\in\R^d$
        in the notation of Grohs et al.~\cite[Lemma~2.11]{GHJW})%
      )
  shows that for all
    $w,h\in\R^d$,
    $t\in[0,T]$
  it holds that
%   \beq
%     \label{eq:Dbnd}
%     \norm{D^w(t)h}
%     \leq 
%     \norm h\exp\bbbp{\bbbbr{\sup_{s\in[0,T]}\varphi(z^w(s))}t}
%     .
%   \eeq
%   In addition, note that 
%     \eqref{eq:suppathfin}
%     and \eqref{eq:Dint}
%   imply that for all
%     $w,h\in\R^d$
%   it holds that
%   \beq
%     \norm{D^w(0)h}
%     \leq 
%     \norm h.
%   \eeq
%   %
%     This 
%     and \eqref{eq:Dbnd}
%   prove that for all
%     $w,h\in\R^d$,
%     $t\in[0,T]$
%   it holds that
  \beq
    \norm{D^w(t)h}
    \leq 
    \norm h\exp\bbbp{\bbbbr{\sup_{s\in[0,T]}\varphi(z^w(s))}t}
    \leq 
    \norm h\exp\bbbp{T\bbbbr{\sup_{s\in[0,T]}\varphi(z^w(s))}}
    .
  \eeq
  Combining 
    \com{a} this 
  with 
    \eqref{eq:dist}
  shows that for all
    $x,y\in\R^d$,
    $t\in[0,T]$
  it holds that
  \beq
    \norm{z^x(t)-z^y(t)}
    \ann{\leq}{a \eqref{eq:dist}}
    \sup_{u\in[0,1]} \bbbbr{\norm{x-y} \exp\bbbp{T\bbbbr{\sup_{s\in[0,T]}\varphi(z^{(1-u)y+ux}(s)) }}}
    %\\&= 
    %\norm{x-y} \bbbbr{\sup_{u\in[0,1]}\exp\bbbp{T\bbbbr{ \sup_{s\in[0,T]}\varphi(z^{(1-u)y+ux}(s)) }}}
    .
  \eeq
  This completes the proof of Lemma~\ref{lem:distlem}.
\end{proof}

\subsection{Conditional sub-Hoelder continuity for SDEs}
\label{subsec:condsubhoelder}

\begin{lemma}
\label{lem:faux2}
  Let $q\in(0,\infty)$.
  Then it holds for all 
    $a,b\in[e^q,\infty)$ with $a\leq b$
  that
  \beq
  \label{eq:faux2conc}
  \frac{a^2}{\abs{\ln(a)}^{2q}}
  \leq
  \frac{b^2}{\abs{\ln(b)}^{2q}}
  .
  \eeq
\end{lemma}
\begin{proof}[Proof of Lemma~\ref{lem:faux2}]
  Throughout this proof
  let $f\colon (1,\infty)\to[0,\infty)$
    \intrtype{be the function which satisfies }%
    \intrtypen{satisfy }%
    for all 
      $z\in(1,\infty)$ 
    that
    \beq
      f(z)=\frac{z^2}{\abs{\ln(z)}^{2q}}
      .
    \eeq
  Note that 
    $f$ is a continuously differentiable function
    and for all 
      $z\in[e^q,\infty)$ 
    it holds that
    \beq
      f'(z)
      =
      \frac{2z[\ln(z)]^{2q}-2qz^2[\ln(z)]^{2q-1}z^{-1}}{[\ln(z)]^{4q}}
      =
      \frac{2z\ln(z)-2qz}{[\ln(z)]^{2q+1}}\geq 0
      .
    \eeq
  Hence, we obtain that 
    $f|_{[e^q,\infty)}$ is an increasing function.
  This establishes~\eqref{eq:faux2conc}.
  The proof of Lemma~\ref{lem:faux2} is thus completed.
\end{proof}

\begin{lemma}
\label{lem:faux}
  Let $q\in[0,\infty)$.
  Then it holds for all 
    $a,b\in[1,\infty)$ with $a\leq b$ 
  that
  \beq
  \label{eq:fauxconc}
    \frac{(e^qa)^2}{\abs{\ln(e^qa)}^{2q}}
    \leq
    \frac{(e^qb)^2}{\abs{\ln(e^qb)}^{2q}}.
  \eeq
\end{lemma}
\begin{proof}[Proof of Lemma~\ref{lem:faux}]
  Throughout this proof
    assume w.l.o.g.~that $q>0$.
  Observe that 
    Lemma~\ref{lem:faux2}
    (with
      $q\is q$,
      $a\is e^qa$,
      $b\is e^qb$
    in the notation of Lemma~\ref{lem:faux2})
  implies~\eqref{eq:fauxconc}.
  The proof of Lemma~\ref{lem:faux} is thus completed.
\end{proof}

\begin{lemma}
\label{lem:lem11}
  Let $d\in\N$, 
  $T,R,q,K,\mc K\in[0,\infty)$,
  $\varphi\in C(\R^d,[0,\infty))$,
  let $(\Omega,\mc F,\PP)$ be a probability space,
  let $\norm{\cdot}\colon\R^d\to[0,\infty)$ be a norm, 
  let $X^x\colon[0,T]\times\Omega\to\R^d$, $x\in\R^d$, be 
  stochastic processes with continuous sample paths
  which satisfy for all 
    $t\in[0,T]$,
    $\omega\in\Omega$
  that 
  $(\R^d\ni x\mapsto X^x(t,\omega)\in\R^d)\in C^1(\R^d,\R^d)$,
  assume for all 
    $x\in\R^d$,
    $t\in[0,T]$, 
    $\omega\in\Omega$,
    $h\in\bigl\{v\in\R^d\colon\bp{[0,T]\ni s\mapsto \bp{\tfrac{\partial}{\partial x} X^x(s,\omega) }(v)\in\R^d}\text{ is a $\mc B([0,T])$/$\mc B(\R^d)$-measurable function}\bigr\}$
  that
  $
  \int_0^T \bnorm{\bp{ \tfrac{\partial}{\partial x} X^x(s,\omega) }(h)}\,\diff s<\infty
  $ and
  \beq
  \label{eq:bndDxX}
    \bnorm{\bp{ \tfrac{\partial}{\partial x} X^x(t,\omega) }(h)}
    \leq 
    \norm h+\int_0^t\varphi(X^x(s,\omega))\,\bnorm{\bp{\tfrac\partial{\partial x}X^x(s,\omega)}(h)}\,\diff s
    ,
  \eeq
  assume that
    $\bEE{\sup_{x\in\{z\in\Q^d\colon\norm{z}\leq R+1\}}\sup_{t\in[0,T]\cap\Q}\bp{[\varphi(X^x(t))]^{4q+4}}}\leq K$,
  and assume for all
    $x\in\{z\in\R^d\colon\norm{z}\leq R+1\}$,
    $t\in[0,T]$
  that
    $
    %\sup_{x\in\{z\in\R^d\colon\norm{z}\leq R+1\}}\sup_{t\in[0,T]}
    \bEE{\norm{X^x(t)}^2}\leq K$
  and $\mc K=1+2^{4q+4}\bp{\abs{\ln(2+e^q)}^{4q+4}+T^{4q+4}K}$.
  Then it holds
  for all 
    $x\in\{z\in\R^d\colon \norm{z}\leq R\}$,
    $h\in\{v\in \R^d\setminus\{0\}\colon\norm v<1\}$,
    $t\in[0,T]$
  that
  \beq
    \bEE{\norm{X^{x+h}(t)-X^x(t)}}
    \leq 
    2\sqrt{(1+4K)\mc K}\,\babs{\ln(\norm h)}^{-q}
    .
  \eeq
\end{lemma}
\begin{proof}[Proof of Lemma~\ref{lem:lem11}]
%%%%
%%  Definitions
%%%%
Throughout this proof 
let 
  $F,\,G\colon [0,\infty)\to[0,\infty)$
  \intrtype{be the functions which }%
  satisfy for all 
    $y\in[0,\infty)$
  that
  \beq
  \label{eq:defphipsi}
    F(y)=\ln(1+y),
    \qquad
    G(y)=
    \begin{cases}
      0&\colon y=0\\
      [\ln(1+y)]^{-1}y &\colon y\neq 0,
    \end{cases}
  \eeq
let 
  $D^x\colon [0,T]\times\Omega\to\R^{d\times d}$, $x\in\R^d$,
  \intrtype{be the functions which satisfy }%
  \intrtypen{satisfy }%
  for all 
    $x\in\R^d$,
    $t\in[0,T]$, 
    $\omega\in\Omega$
  that
  \beq
  \label{eq:defD}
    D^x(t,\omega)= \tfrac{\partial}{\partial x} X^x(t,\omega)
    ,
  \eeq
let 
  $Y\colon\Omega\to[0,\infty]$ 
  \intrtype{be the function which satisfies }%
  \intrtypen{satisfy }%
  for all
    $\omega\in\Omega$
  that
  \beq
    Y(\omega)
    =
    \bbbbr{
    \sup_{x\in\{z\in\R^d\colon\norm{z}\leq R+1\}}\,
      \sup_{t\in[0,T]}
        \varphi(X^x(t,\omega))}
    ,
  \eeq
let 
  $A\subseteq\Omega$ be the set which satisfies
  \beq
  \label{eq:defA}
    A
    =
    \{\omega\in\Omega\colon Y(\omega)<\infty\}
    ,
  \eeq
and let 
  $Z\colon\Omega\to[0,\infty)$ 
  \intrtype{be the function which satisfies }%
  \intrtypen{satisfy }%
  for all 
    $\omega\in \Omega$ 
  that
  \beq
  \label{eq:defZ}
    Z(\omega)=
    \begin{cases}
      \exp(T\, Y(\omega))&\colon \omega\in A\\
      1 &\colon \omega\in\Omega\setminus A.
    \end{cases}
  \eeq
Note that 
  \eqref{eq:defphipsi} 
implies that for all 
  $y\in[0,\infty)$ 
it holds that
\beq
\label{eq:phipsiid}
  y=G(y)F(y)
  .
\eeq
Hence, we obtain that for all
  $x,h\in\R^d$,
  $t\in[0,T]$
it holds that
\beq
  \label{eq:splitFG}
  \bEE{\norm{X^{x+h}(t)-X^{x}(t)}}
  =
  \bEE{G(\norm{X^{x+h}(t)-X^{x}(t)}) F(\norm{X^{x+h}(t)-X^{x}(t)})}
  .
\eeq
Next observe that 
  the fundamental theorem of calculus
ensures that for all 
  $y\in[0,\infty)$ 
it holds that
\beq
  \ln(1+y)=\bbr{\ln(1+z)}_{z=0}^{z=y}=\int_0^y \frac{1}{1+z}\,\diff z
  \geq y\bbbbr{ \inf_{z\in[0,y]}\frac{1}{1+z} }=\frac{y}{1+y}
  .
\eeq
  This 
  and \eqref{eq:defphipsi} 
show that for all 
  $y\in[0,\infty)$ 
it holds that
\beq
  G(y)
%  =
%  \frac{y}{\ln(1+y)}
  \leq 
  1+y
  .
\eeq
  \com{a} This,
  \com{b} the fact that $G$ is a $\mc B([0,\infty))$/$\mc B([0,\infty))$-measurable function,
  \com{c} the fact that 
    for all 
      $a,b\in\R$ 
    it holds that
      $\abs{a+b}^2\leq 2(\abs a^2+\abs b^2)$,
  and the triangle inequality
show that for all 
  $x,h\in\R^d$, 
  $t\in[0,T]$
it holds that
\ba
  \com{b}\bEE{ \babs{ G(\norm{X^{x+h}(t)-X^{x}(t)}) }^2 }
  &\ann{\leq}{a}
  \bEE{ \bp{ 1+ \norm{X^{x+h}(t)-X^{x}(t)} }^2 }\\
  &\ann{\leq}{c}
  \bEE{2\bp{ 1+ \norm{X^{x+h}(t)-X^{x}(t)}^2 }}\\
  %\bbEE{ \bp{ 1+ \norm{X^{x+h}(t)-X^{x}(t)} }^2 + \bp{ 1- \norm{X^{x+h}(t)-X^{x}(t)} }^2}\\
  %&= 
  %\bbEE{ 2 + 2\norm{X^{x+h}(t)-X^{x}(t)}^2 }\\
  &= 
  2\bp{ 1+\bEE{ \norm{ X^{x+h}(t)-X^{x}(t) }^2 } }
  \\&\leq
  2\bp{ 1+\bEE{ (\norm{ X^{x+h}(t) } + \norm{ X^{x}(t) })^2 } }\\
  &\ann{\leq}b
  2\bp{ 1+\bEE{ 2\bp{ \norm{ X^{x+h}(t) }^2 + \norm{ X^{x}(t) }^2 } } }\\
  &= 
  2\bp{ 1+2\bp{\bEE{ \norm{X^{x+h}(t)}^2}+\bEE{\norm{X^{x}(t)}^2}} }
  .
\ea
  The hypothesis that
    for all
      $x\in\{z\in\R^d\colon\norm{z}\leq R+1\}$,
      $t\in[0,T]$
    it holds that
      $\EE[\norm{X^{x}(t)}^2]\leq K$
  hence
implies that for all 
  $x\in\{z\in \R^d\colon \norm z\leq R\}$, 
  $h\in\{v\in \R^d\colon \norm v<1\}$,
  $t\in[0,T]$
it holds that
\beq
\label{eq:bnd1}
  \bEE{ \babs{ G(\norm{X^{x+h}(t)-X^{x}(t)}) }^2 }
%   \\&\ann{\leq}a
%   2\bp{ 1+\bEE{ (\norm{ X^{x+h}(t) } + \norm{ X^{x}(t) })^2 } }\\
%   &\ann{\leq}b
%   2\bp{ 1+\bEE{ 2\bp{ \norm{ X^{x+h}(t) }^2 + \norm{ X^{x}(t) }^2 } } }\\
%   &= 
%   2\bp{ 1+2\bp{\bEE{ \norm{X^{x+h}(t)}^2}+\bEE{\norm{X^{x}(t)}^2}} }\\
  \ann{\leq}{\eqref{eq:cKrmeas}}
  2(1+2(K+K))
  =
  2+8K
  .
\eeq
% 
% 
% For all $K\geq 0$ and $r\in\N$ we put 
% \[
% c_{K, r}= \EE \bigl[\sup_{\|y\|\leq K}\sup_{s\in[0, T]} \|X^y(s)\|^{r}\bigr].
% \]
% The assumption (A1) implies that $c_{K, r}\in\R$ for all $K\geq 0$ and $r\in\N$. Employing \eqref{lab1}  we obtain
% \begin{align}\label{lab10}
% \EE\bigl[ \psi^2(\|X^{x+h}(t)-X^{x}(t)\|)\bigr]&\leq 2 \, \bigl(1+\EE\bigl[\|X^{x+h}(t)-X^{x}(t)\|^2\bigr]\bigr)\notag\\
% &\leq 2 \,\bigl(1+2\,\EE\bigl[\|X^{x+h}(t)\|^2\bigr]+2\,\EE\bigl[\|X^{x}(t)\|^2\bigr]\bigr)\leq 2+8 c_{R+1, 2 }.
% \end{align}
% 
% Next, we estimate $\EE\bigl[\varphi^2(\|X^{x+h}(t)-X^{x}(t)\|)\bigr]$.
In the next step we note that
  \eqref{eq:bndDxX},
  the hypothesis that
    for all 
      $x\in\R^d$,
      $\omega\in\Omega$,
      $h\in\bigl\{v\in\R^d\colon\bp{[0,T]\ni s\mapsto \bp{\tfrac{\partial}{\partial x} X^x(s,\omega) }(v)\in\R^d}\text{ is a $\mc B([0,T])$/$\mc B(\R^d)$-measurable function}\bigr\}$
    it holds that
    $
    \int_0^T \bnorm{\bp{ \tfrac{\partial}{\partial x} X^x(s,\omega) }(h)}\,\diff s<\infty
    $, 
    and Lemma~\ref{lem:distlem}
    (with
      $d\is d$,
      $T\is T$,
      $\varphi\is\varphi$,
      $\norm{\cdot}\is\norm{\cdot}$,
      $(z^x)_{x\in\R^d}\is \bp{[0,T]\ni t\mapsto X^x(t,\omega)\in\R^d}{}_{x\in\R^d}$
      for
      $\omega\in\Omega$
    in the notation of Lemma~\ref{lem:distlem})
show that for all
  $x,h\in\R^d$,
  $t\in[0,T]$,
  $\omega\in \Omega$
it holds that
\beq
  \norm{X^{x+h}(t,\omega)-X^x(t,\omega)}
  \leq
  \sup_{u\in[0,1]}\bbbbr{\norm{h}\exp\bbbp{T\bbbbr{\sup_{s\in[0,T]}\varphi(X^{x+uh}(s,\omega))}}}
  .
\eeq
Therefore, we obtain that for all
  $x\in\{z\in\R^d\colon\norm z\leq R\}$,
  $h\in\{v\in\R^d\colon\norm v<1\}$,
  $t\in[0,T]$,
  $\omega\in A$
it holds that
\ba
  \label{eq:distbndZ}
  \norm{X^{x+h}(t,\omega)-X^x(t,\omega)}
  &\leq
  \sup_{y\in\{z\in\R^d\colon\norm z\leq R+1\}}\bbbbr{\norm{h}\exp\bbbp{T\bbbbr{\sup_{s\in[0,T]}\varphi(X^{y}(s,\omega)}}}
  \\&=
  \norm{h}\exp(T\,Y(\omega))
  \\&=
  \norm h\,Z(\omega).
\ea
Furthermore, observe that
  Lemma~\ref{lem:supmeas}
  (with
    $d\is d$,
    $T\is T$,
    $R\is R+1$,
    $\norm\cdot\is\norm\cdot$,
    $(\Omega,\mc F,\PP)\is(\Omega,\mc F,\PP)$,
    $(Y^x)_{x\in\R^d}\is \bp{[0,T]\times\Omega\ni (t,\omega)\mapsto \varphi(X^x(t,\omega))\in[0,\infty)}{}_{x\in\R^d}$
  in the notation of Lemma~\ref{lem:supmeas})
ensures
\begin{enumerate}[label=(\alph{enumi})]
\item 
  that for all
    $\omega\in\Omega$
  it holds that
  \beq
  \label{eq:supRQ}
    Y(\omega)
    =
    \bbbbr{
    \sup_{x\in\{z\in\Q^d\colon\norm{z}\leq R+1\}}\,\sup_{t\in[0,T]\cap\Q}\varphi(X^x(t,\omega))
    }
  \eeq
  and
\item
  that $Y$
%   \beq
%   \label{eq:Rsupmeas}
%     \Omega\ni\omega\mapsto \bbbbr{\sup_{x\in\{z\in\R^d\colon\norm{z}\leq K\}}\,
%       \sup_{t\in[0,T]}\bp{[\varphi(X^x(t,\omega))]^r} }\in [0,\infty]
%   \eeq
  is an $\mc F$/$\mc B([0,\infty])$-measurable function.
\end{enumerate}
Observe that
  \eqref{eq:defA}
  and the fact that $Y$ is an $\mc F$/$\mc B([0,\infty])$-measurable function
imply that
\beq
  \label{eq:AinF}
  A\in\mc F.
\eeq
Next note that
  the hypothesis that
    $\bEE{\sup_{x\in\{z\in\Q^d\colon\norm{z}\leq R+1\}}\sup_{t\in[0,T]\cap\Q}\bp{[\varphi(X^x(t))]^{4q+4}}}\leq K$,
  the fact that $Y$ is an $\mc F$/$\mc B([0,\infty])$-measurable function,
  and \eqref{eq:supRQ}
ensure that
\beq
  \EE[Y]
  \leq
  \bbbEE{\sup_{x\in\{z\in\Q^d\colon\norm{z}\leq R+1\}}\,\sup_{t\in[0,T]\cap\Q}\bp{1+[\varphi(X^x(t))]^{4q+4}} }
  \leq 
  1+K
  <\infty
  .
\eeq
Combining
  this
with
  \eqref{eq:AinF}
implies that
\beq
\label{eq:PA1}
  \PP(A)=1.
\eeq
Furthermore, note that
  \eqref{eq:defZ},
  \eqref{eq:AinF},
  and the fact that $Y$ is an $\mc F$/$\mc B([0,\infty])$-measurable function
show that
  $Z$ is an $\mc F$/$\mc B([0,\infty))$-measurable function.
Combining
  this,
  \eqref{eq:distbndZ},
  and \eqref{eq:PA1}
with
  \eqref{eq:defphipsi}
  and the fact that $F$ is a $\mc B([0,\infty))$/$\mc B([0,\infty))$-measurable function
demonstrates that for all
  $x\in\{z\in\R^d\colon \norm{z}\leq R\}$,
  $h\in\{v\in \R^d\colon\norm v<1\}$,
  $t\in[0,T]$
it holds that
\ba
\label{eq:bndphi4}
  \bEE{ \babs{F(\norm{X^{x+h}(t)-X^x(t)})}^2 }
  &\ann{=}{\eqref{eq:defphipsi}}
  \bEE{ \babs{\ln(1+\norm{X^{x+h}(t)-X^x(t)})}^2 }
  \\&\ann{\leq}{\eqref{eq:distbndZ}}
  \bEE{ \babs{\ln(1+\norm h Z)}^2}
  .
\ea
In the next step we observe that for all
  $\omega\in A$
it holds that
\beq
  [Y(\omega)]^{4q+4}
  =
  \bbbbr{\sup_{x\in\{z\in\Q^d\colon\norm{z}\leq R+1\}}\,\sup_{t\in[0,T]\cap\Q}\bp{[\varphi(X^x(t,\omega))]^{4q+4}}}
  .
\eeq
  The hypothesis that
    $\bEE{\sup_{x\in\{z\in\Q^d\colon\norm{z}\leq R+1\}}\sup_{t\in[0,T]\cap\Q}\bp{[\varphi(X^x(t))]^{4q+4}}}\leq K$
  and \eqref{eq:PA1}
  hence
show that
\beq
  \label{eq:powY}
  \bEE{Y^{4q+4}}
  \leq K.
\eeq
Moreover, note that
  the fact that 
    for all 
      $a,b\in\R$,
      $r\in[0,\infty)$
    it holds that
      $\abs{a+b}^r\leq 2^r(\abs a^r+\abs b^r)$
ensures that for all
  $C\in[1,\infty)$,
  $r\in[0,4q+4]$,
  $\omega\in A$
it holds that
\ba
\abs{\ln(CZ(\omega))}^r
&\leq
1+\abs{\ln(CZ(\omega))}^{4q+4}
\\&=
1+\abs{\ln(C)+\ln(Z(\omega))}^{4q+4}
\\&=
1+\abs{\ln(C)+T\, Y(\omega)}^{4q+4}
\\&\leq
1+2^{4q+4}\bp{\abs{\ln(C)}^{4q+4}+T^{4q+4}[Y(\omega)]^{4q+4}}
.
\ea
Combining
  this
  and the fact that $Z$ is an $\mc F$/$\mc B([0,\infty))$-measurable function
with
  \eqref{eq:PA1}
  and \eqref{eq:powY}
proves that for all
  $C\in[1,2+e^q]$,
  $r\in[0,4q+4]$
it holds that
\beq
\label{eq:lnCZ}
  \bEE{\abs{\ln(C Z)}^r}
  \leq 
  1+2^{4q+4}\bp{\abs{\ln(2+e^q)}^{4q+4}+T^{4q+4}K}
  =\mc K.
\eeq
Next note that
  the fact that $Z$ is an $\mc F$/$\mc B([0,\infty))$-measurable function,
  the fact that for all $\omega\in\Omega$ it holds that $Z(\omega)\geq 1$,
  and \com{b} the fact that 
    for all 
      $y\in[0,\infty)$ 
    it holds that
      $\ln(1+y)\leq y$
show that for all 
  $h\in\{v\in\R^d\setminus\{0\}\colon\norm v<1\}$
it holds that
\ba
  \bEE{ \abs{\ln(1+\norm h Z)}^2 \1_{\{Z\leq 1/\norm h\}} }
  &\ann{\leq}{a}
  \bEE{ \norm h^2 Z^2 \1_{\{Z\leq 1/\norm h\}} }\\
  &=
  \norm h^2 e^{-2q}\,\bEE{ (e^q Z)^2\1_{\{Z\leq 1/\norm h\}} }\\
  &=
  \norm h^2 e^{-2q}\,\bbbEE{ \frac{(e^qZ)^2}{\abs{\ln(e^qZ)}^{2q}}\abs{\ln(e^qZ)}^{2q}\1_{\{Z\leq 1/\norm h\}} }
  .
\ea
  The fact that 
    for all 
      $\omega\in\Omega$ 
    it holds that
      $Z(\omega)\geq 1$
  and Lemma~\ref{lem:faux}
  hence 
prove that for all 
  $h\in\{v\in\R^d\setminus\{0\}\colon\norm v<1\}$
it holds that
\ba
\label{eq:bndphi3}
  \bEE{ \abs{\ln(1+\norm h Z)}^2 \1_{\{Z\leq 1/\norm h\}} }
  &\leq 
  \norm h^2 e^{-2q}\,\bbbbEE{ \frac{\bp{\frac{e^q}{\norm{h}}}^2}{\babs{\ln\bp{\frac{e^q}{\norm{h}}}}^{2q}}\abs{\ln(e^qZ)}^{2q}\1_{\{Z\leq 1/\norm h\}} }\\
  &=
  \babs{\ln\bp{ \tfrac{e^q}{\norm h}}}^{-2q}\,\bEE{ \abs{\ln(e^qZ)}^{2q}\1_{\{Z\leq 1/\norm h\}} }\\
  &\leq 
  \babs{\ln\bp{ \tfrac{e^q}{\norm h}}}^{-2q}\,\bEE{ \abs{\ln(e^qZ)}^{2q} }
  \\&=
  \babs{q-\ln(\norm h)}^{-2q}\,\bEE{ \abs{\ln(e^qZ)}^{2q} }
  \\&=
  \bp{q+\babs{\ln(\norm h)}}^{-2q}\,\bEE{ \abs{\ln(e^qZ)}^{2q} }
  \\&\leq
  \babs{\ln(\norm h)}^{-2q}\,\bEE{ \abs{\ln(e^qZ)}^{2q} }
  .
\ea
  This 
  and \eqref{eq:lnCZ} (with $C\is e^q$, $r\is 2q$ in the notation of \eqref{eq:lnCZ})
establish that for all 
  $h\in\{v\in\R^d\setminus\{0\}\colon\norm v<1\}$
it holds that
\ba
\label{eq:bndphi7}
  \bEE{ \abs{\ln(1+\norm h Z)}^2 \1_{\{Z\leq 1/\norm h\}} }
  &\leq
  \mc K\,\babs{\ln(\norm h)}^{-2q}
  .
\ea
In the next step we observe that 
  \eqref{eq:lnCZ} (with $C\is 1$, $r\is 4q$ in the notation of \eqref{eq:lnCZ})
  and the fact that 
    for all
      $h\in\{v\in\R^d\setminus\{0\}\colon\norm v<1\}$,
      $\omega\in\bigl\{Z>\frac1{\norm h}\bigr\}$
    it holds that 
      $\abs{\ln(Z(\omega))}^{4q}\abs{\ln(\nicefrac{1}{\norm h})}^{-4q}\geq 1$
imply that for all
  $h\in\{v\in\R^d\setminus\{0\}\colon\norm v<1\}$ 
it holds that
\ba
\label{eq:bndphi6}
  \bEE{ \1_{\{Z>1/\norm h\}} }
  &\leq 
  \bEE{ \abs{\ln(Z)}^{4q}\abs{\ln(\nicefrac{1}{\norm h})}^{-4q}\1_{\{Z>1/\norm h\}} }\\
  &\leq 
  \bEE{ \abs{\ln(Z)}^{4q}\abs{\ln(\nicefrac{1}{\norm h})}^{-4q} }\\
  &=
  \bEE{ \abs{\ln(Z)}^{4q}}\babs{\ln(\norm h)}^{-4q}\\
  &\leq
  \mc K\,\babs{\ln(\norm h)}^{-4q}
  .
\ea
Furthermore, observe that 
  \eqref{eq:lnCZ} (with $C\is 2$, $r\is 4$ in the notation of \eqref{eq:lnCZ})
shows that
\beq
\label{eq:bndphi1}
  \bEE{ \abs{\ln(2 Z)}^4 }
  \leq
  \mc K
  .
\eeq
  This
  and \eqref{eq:bndphi6}
ensure that for all
  $h\in\{v\in\R^d\setminus\{0\}\colon\norm v<1\}$ 
it holds that
\beq
  \bEE{ \1_{\{Z>1/\norm h\}} }<\infty
  \qqandqq
  \bEE{ \abs{\ln(2 Z)}^4 }<\infty
  .
\eeq
Combining
  this,
  \eqref{eq:bndphi6},
  \eqref{eq:bndphi1},
  and the fact that for all $\omega\in\Omega$ it holds that $Z(\omega)\geq 1$
with
  the Cauchy-Schwarz inequality
establishes that for all 
  $h\in\{v\in\R^d\setminus\{0\}\colon \norm v<1\}$ 
it holds that
\ba
\label{eq:bndphi2}
  \bEE{ \abs{\ln(1+\norm h Z)}^2 \1_{\{Z>1/\norm h\}} }
  &\leq 
  \bEE{ \abs{\ln(Z+Z)}^2 \1_{\{Z>1/\norm h\}} }\\
  &\leq 
  \bbp{ \bEE{ \abs{\ln(2 Z)}^4 } \bEE{ \1_{\{Z>1/\norm h\}} } }^{\!\nicefrac12}\\
  &\leq 
  \bp{\mc K^2\,\babs{\ln(\norm h)}^{-4q}}^{\!\nicefrac12}
  \\&=
  \mc K\,\babs{\ln(\norm h)}^{-2q}
%   &=
%   2^{2q+2}(1+T^4K)^{\nicefrac12}(T^{4q}K)^{\nicefrac12}\,\babs{\ln(\norm h)}^{-2q}
%   \\&\leq
%   2^{2q+2}T^{2q}(1+T^2\sqrt{K})\sqrt{K}\,\babs{\ln(\norm h)}^{-2q}
%   \\&\leq
%   2^{2q+2}T^{2q}(1+T^2)K\,\babs{\ln(\norm h)}^{-2q}
  .
%&= \bbp{ \bEE{ \ln(2 Z)^4 } \PP\bp{ \bigl\{Z>\tfrac{1}{\norm h}\bigr\} } }^{\nicefrac12}.
\ea
Combining 
  this, 
  \eqref{eq:bndphi4}, 
  and \eqref{eq:bndphi7}
proves that for all
  $x\in\{z\in \R^d\colon\norm z\leq R\}$,
  $h\in\{v\in\R^d\setminus\{0\}\colon\norm v<1\}$,
  $t\in[0,T]$
it holds that
\ba
  \label{eq:bndphipart}
  &\bEE{ \babs{F(\norm{X^{x+h}(t)-X^x(t)})}^2 }\\
  &\leq 
  \bEE{ \abs{\ln(1+\norm h Z)}^2(\1_{\{Z\leq 1/\norm h\}}+\1_{\{Z> 1/\norm h\}})}\\
  &=
  \bEE{ \abs{\ln(1+\norm h Z)}^2\1_{\{Z\leq 1/\norm h\}}}
    +\bEE{ \abs{\ln(1+\norm h Z)}^2\1_{\{Z>1/\norm h\}}}\\
  &\leq 
  \mc K\,\babs{\ln(\norm h)}^{-2q}+\mc K\,\babs{\ln(\norm h)}^{-2q}
  \\&=
  2\mc K\,\babs{\ln(\norm h)}^{-2q}
%   \\&\leq 
%   2^{2q+2}\bp{q^{2q}+T^{2q}K+T^{2q}(1+T^2)K}\,\babs{\ln(\norm h)}^{-2q}
%   \\&=
%   2^{2q+2}\bp{q^{2q}+T^{2q}(2+T^2)K}\,\babs{\ln(\norm h)}^{-2q}
  .
\ea
  This
  and \eqref{eq:bnd1}
ensure that for all
  $x\in\{z\in \R^d\colon\norm z\leq R\}$,
  $h\in\{v\in\R^d\setminus\{0\}\colon\norm v<1\}$, 
  $t\in[0,T]$
it holds that
\beq
\bEE{ \babs{ G(\norm{X^{x+h}(t)-X^{x}(t)}) }^2 }<\infty
\qqandqq
\bEE{ \babs{ F(\norm{X^{x+h}(t)-X^{x}(t)}) }^2 }<\infty.
\eeq
Combining
  \com{a} this
  and \com{b} the Cauchy-Schwarz inequality 
with 
  \eqref{eq:splitFG}
proves that for all
  $x\in\{z\in \R^d\colon\norm z\leq R\}$,
  $h\in\{v\in\R^d\setminus\{0\}\colon\norm v<1\}$,
  $t\in[0,T]$
it holds that
\beq
\label{eq:distphipsi}
  \bEE{\norm{X^{x+h}(t)-X^{x}(t)}}
  \leq
  \bbp{ \bEE{\babs{G(\norm{X^{x+h}(t)-X^{x}(t)})}^2} \bEE{\babs{F(\norm{X^{x+h}(t)-X^{x}(t)})}^2} }^{\!\nicefrac12}
  .
\eeq
%
% In addition, note that
%   \eqref{eq:bndphipart}
% shows that
% \ba
%   \bbp{\bEE{ \babs{F(\norm{X^{x+h}(t)-X^x(t)})}^2 } }^{\!\nicefrac12}
%   &\leq
%   \bp{2^{2q+2}\bp{q^{2q}+T^{2q}(2+T^2)K}}^{\nicefrac12}\,\babs{\ln(\norm h)}^{-q}
%   \\&\leq
%   2^{q+1}\bp{q^{q}+T^{q}\sqrt{(2+T^2)K}}\,\babs{\ln(\norm h)}^{-q}
% .
% \ea
% 
  This,
  \eqref{eq:bnd1},
  and \eqref{eq:bndphipart}
prove that for all
  $x\in\{z\in\R^d\colon\norm z\leq R\}$,
  $h\in\{v\in\R^d\setminus\{0\}\colon\norm v<1\}$,
  $t\in[0,T]$
it holds that
\beq
  \bEE{\norm{X^{x+h}(t)-X^{x}(t)}}
  \leq
  \bp{(4+16K)\mc K\,\babs{\ln(\norm h)}^{-2q}}^{\nicefrac12}
  =
  2\sqrt{(1+4K)\mc K}\,\babs{\ln(\norm h)}^{-q}
  .
\eeq
This completes the proof of Lemma~\ref{lem:lem11}.
\end{proof}

\begin{lemma}
\label{lem:lem1}
  Let $d\in\N$, 
  $T,\kappa\in[0,\infty)$,
  let $(\Omega,\mc F,\PP)$ be a probability space,
  let $\norm{\cdot}\colon\R^d\to[0,\infty)$ be a norm, 
  let $X^x\colon[0,T]\times\Omega\to\R^d$, $x\in\R^d$, be 
  stochastic processes
  which satisfy for all
    $\omega\in\Omega$
  that 
  $([0,T]\times\R^d\ni (t,x)\mapsto X^x(t,\omega)\in\R^d)\in C^{0,1}([0,T]\times\R^d,\R^d)$,
  assume for all $R,r\in[0,\infty)$ that
  $\bEE{ \sup_{x\in\{z\in\Q^d\colon \norm z\leq R\}}\sup_{t\in[0,T]\cap\Q}\bp{\norm{X^x(t)}^{r}} }<\infty$,
  and assume for all 
    $x,h\in\R^d$,
    $t\in[0,T]$, 
    $\omega\in\Omega$
  that
  \beq
  \label{eq:diffXxbnd}
    \bnorm{\bp{ \tfrac{\partial}{\partial x} X^x(t,\omega) }(h)}
    \leq 
    \norm h+\kappa\int_0^t(1+\norm{X^x(s,\omega)}^\kappa)\,\bnorm{\bp{\tfrac\partial{\partial x}X^x(s,\omega)}(h)}\,\diff s
    .
  \eeq
  Then it holds for all
    $R,q\in[0,\infty)$
  that there exists $c\in[0,\infty)$
  such that
  for all 
    $h\in\{v\in \R^d\setminus\{0\}\colon\norm v<1\}$
  it holds that
  \beq
  \label{eq:lem1conc}
    \bbbbr{\sup_{x\in\{z\in\R^d\colon \norm{z}\leq R\}}\,\sup_{t\in[0,T]}
      \bEE{\norm{X^{x+h}(t)-X^x(t)}}}
    \leq 
    c\,\babs{\ln(\norm h)}^{-q}
    .
  \eeq
\end{lemma}
\begin{proof}[Proof of Lemma~\ref{lem:lem1}]
  Throughout this proof
  let $\varphi\colon\R^d\to[0,\infty)$ be the function which satisfies 
    for all
      $x\in\R^d$
    that
    \beq
    \label{eq:defphi}
      \varphi(x)=\kappa(1+\norm{x}^\kappa)
    \eeq
  and let $K_{R,q}\in[0,\infty]$, $R,q\in[0,\infty)$, satisfy
    for all $R,q\in[0,\infty)$ that
  \begin{multline}
  K_{R,q}=\max\biggl\{ 
    \bbbEE{\sup_{x\in\{z\in\Q^d\colon \norm{z}\leq R+1\}}\,\sup_{t\in[0,T]\cap\Q}
      \bp{[\varphi(X^x(t))]^{4q+4}}},\\
      \bbbEE{\sup_{x\in\{z\in\Q^d\colon \norm{z}\leq R+1\}}\,\sup_{t\in[0,T]\cap\Q}\bp{\norm{X^x(t)}^2}}
  \biggr\}
  .
  \end{multline}
  Note that
    \eqref{eq:defphi}
    and Lemma~\ref{lem:powsum}
%       the fact that for all $a,b\in\R$, $r\in[0,\infty)$ it holds that
%       $\abs{a+b}^r\leq2^r(\abs{a}^r+\abs{b}^r)$
  show that for all
    $q\in[0,\infty)$,
    $x\in\R^d$,
    $t\in[0,T]$,
    $\omega\in\Omega$
  it holds that
  \ba
  \label{eq:phipow}
    [\varphi(X^x(t,\omega))]^{4q+4}
    &=
    \bbr{\kappa+\kappa\norm{X^x(t,\omega)}^\kappa}^{4q+4}
    \\&\leq
    2^{4q+3}\bp{\kappa^{4q+4}+\kappa^{4q+4}\norm{X^x(t,\omega)}^{\kappa(4q+4)}}
    .
  \ea
  Furthermore, observe that 
    the hypothesis that 
      for all
        $\omega\in\Omega$
      it holds that
        $[0,T]\times\R^d\ni(t,x)\mapsto X^x(t,\omega)\in\R^d$ is a continuous function
  ensures that for all
    $R,q\in[0,\infty)$,
    $\omega\in\Omega$
  it holds that
  \beq
    \bbbbr{\sup_{x\in\{z\in\R^d\colon \norm{z}\leq R+1\}}\,\sup_{t\in[0,T]}
      \norm{X^x(t,\omega)}^{\kappa(4q+4)}}
    <\infty.
  \eeq
  Combining
    this
  with
    \eqref{eq:phipow}
    and the hypothesis that 
      for all 
        $R,r\in[0,\infty)$ 
      it holds that
      $\bEE{ \sup_{x\in\{z\in\Q^d\colon \norm z\leq R\}}\sup_{t\in[0,T]\cap\Q}\bp{\norm{X^x(t)}^{r}} }<\infty$
  demonstrates that for all
    $R,q\in[0,\infty)$
  it holds that 
  \ba
  \label{eq:Kfin11}
    &\bbbEE{\sup_{x\in\{z\in\Q^d\colon \norm{z}\leq R+1\}}\,\sup_{t\in[0,T]\cap\Q}
      \bp{[\varphi(X^x(t))]^{4q+4}}}
    \\&\leq
    2^{4q+3}\bbbp{
      \kappa^{4q+4}
      +
      \kappa^{4q+4}\,
      \bbbEE{
        \sup_{x\in\{z\in\Q^d\colon \norm{z}\leq R+1\}}\,\sup_{t\in[0,T]\cap\Q}
          \bp{\norm{X^x(t)}^{4\kappa(q+1)}}
      }
    }
    <\infty.
  \ea
    This
    and the hypothesis that
      for all 
        $R,r\in[0,\infty)$ 
      it holds that
        $\EE\bigl[ \sup_{x\in\{z\in\Q^d\colon \norm z\leq R\}}$
        $\sup_{t\in[0,T]\cap\Q}\bp{\norm{X^x(t)}^{r}} \bigr]<\infty$
  prove that for all
    $R,q\in[0,\infty)$
  it holds that
  \beq
  \label{eq:KRqfin}
    K_{R,q}
    <\infty
    .
  \eeq
  Furthermore, note that
    Lemma~\ref{lem:supmeas}
    (with
      $d\is d$,
      $T\is T$,
      $R\is R+1$,
      $\norm\cdot\is\norm\cdot$,
      $(\Omega,\mc F,\PP)\is(\Omega,\mc F,\PP)$,
      $(Y^x)_{x\in\R^d}\is\bp{[0,T]\times\Omega\ni(t,\omega)\mapsto \norm{X^x(t,\omega)}^2\in[0,\infty)}{}_{x\in\R^d}$
      for $R\in[0,\infty)$
    in the notation of Lemma~\ref{lem:supmeas})
  shows that for all
    $R\in[0,\infty)$,
    $\omega\in\Omega$
  it holds that
  \beq
    \bbbbr{
      \sup_{x\in\{z\in\R^d\colon \norm{z}\leq R+1\}}\,\sup_{t\in[0,T]}
      \bp{\norm{X^x(t,\omega)}^2}
    }
    =
    \bbbbr{
      \sup_{x\in\{z\in\Q^d\colon \norm{z}\leq R+1\}}\,\sup_{t\in[0,T]\cap\Q}
      \bp{\norm{X^x(t,\omega)}^2}
    }
    .
  \eeq
  Hence, we obtain that for all
    $R,q\in[0,\infty)$,
    $x\in\{z\in\R^d\colon \norm{z}\leq R+1\}$,
    $t\in[0,T]$
  it holds that
  \beq
  \label{eq:hyp2}
    \bEE{\norm{X^x(t)}^2}
    \leq
    \bbbEE{
      \sup_{y\in\{z\in\Q^d\colon \norm{z}\leq R+1\}}\,\sup_{s\in[0,T]\cap\Q}
      \bp{\norm{X^y(s)}^2}
    }
%     \\&\leq
%     \bbbEE{
%       \sup_{x\in\{z\in\Q^d\colon \norm{z}\leq R+1\}}\,\sup_{t\in[0,T]\cap\Q}
%       \bp{1+\norm{X^x(t,\omega)}^{4(p+1)(q+1)}}
%     }
%     \\&=
%     1+\bbbEE{
%       \sup_{x\in\{z\in\Q^d\colon \norm{z}\leq R+1\}}\,\sup_{t\in[0,T]\cap\Q}
%       \bp{\norm{X^x(t,\omega)}^{4(p+1)(q+1)}}
%     }
    \leq
    K_{R,q}.
  \eeq
  Moreover, observe that
    the hypothesis that for all
      $\omega\in\Omega$
    it holds that
      $([0,T]\times\R^d\ni(t,x)\mapsto X^x(t,\omega)\in\R^d)\in C^{0,1}([0,T]\times\R^d,\R^d)$
  implies that for all
    $x,h\in\R^d$,
    $\omega\in\Omega$
  it holds that
    $[0,T]\ni t\mapsto \bp{\tfrac\partial{\partial x}X^x(t,\omega)}(h)\in\R^d$
      is a continuous function.
  This implies that for all
    $x,h\in\R^d$,
    $\omega\in\Omega$
  it holds that
  \beq
  \label{eq:hyp1}
    \int_0^T \bnorm{\bp{ \tfrac{\partial}{\partial x} X^x(s,\omega) }(h)}\,\diff s<\infty
    .
  \eeq
  Furthermore, note that
    \eqref{eq:diffXxbnd}
    and \eqref{eq:defphi}
  demonstrate that for all
    $x,h\in\R^d$,
    $t\in[0,T]$,
    $\omega\in\Omega$
  it holds that
  \beq
     \bnorm{\bp{ \tfrac{\partial}{\partial x} X^x(t,\omega) }(h)}
     \leq 
     \norm h+\int_0^t\varphi(X^x(s,\omega))\,\bnorm{\bp{\tfrac\partial{\partial x}X^x(s,\omega)}(h)}\,\diff s
    .
  \eeq
  Combining 
    this,
    \eqref{eq:KRqfin},
    \eqref{eq:hyp2},
    and \eqref{eq:hyp1}
  with
    Lemma~\ref{lem:lem11}
    (with
      $d\is d$,
      $T\is T$,
      $R\is R$,
      $q\is q$,
      $K\is K_{R,q}$,
      $\mc K\is 1+2^{4q+4}(\abs{\ln(2+e^q)}^{4q+4}+T^{4q+4}K_{R,q})$,
      $\varphi\is\varphi$,
      $(\Omega,\mc F,\PP)\is(\Omega,\mc F,\PP)$,
      $\norm\cdot\is\norm\cdot$,
      $(X^x)_{x\in\R^d}\is(X^x)_{x\in\R^d}$
      for $R,q\in[0,\infty)$
    in the notation of Lemma~\ref{lem:lem11})
    hence
  proves that for all
    $R,q\in[0,\infty)$,
    $x\in\{z\in\R^d\colon \norm{z}\leq R\}$,
    $h\in\{v\in \R^d\setminus\{0\}\colon\norm v<1\}$,
    $t\in[0,T]$
  it holds that
  \ba
    &\bEE{\norm{X^{x+h}(t)-X^x(t)}}
    \\&\leq 
    2\sqrt{(1+4K_{R,q})\bp{1+2^{4q+4}\bp{\abs{\ln(2+e^q)}^{4q+4}+T^{4q+4}K_{R,q}}}}\,\babs{\ln(\norm h)}^{-q}
    .
  \ea
  This completes
  the proof of Lemma~\ref{lem:lem1}.
\end{proof}

\section{Regularity with respect to the initial value for SDEs}\label{section8}

In this section we establish in Theorem~\ref{thm:finalthm} in Subsection~\ref{subsec:subhoelder}
below the main result of this article. 
Theorem~\ref{thm:finalthm} proves that every additive noise driven 
SDE with a drift coefficient function whose derivatives grows at most polynomially 
and which also admits a Lyapunov-type condition 
(which ensures the existence of a unique solution) is at least logarithmically 
Hoelder continuous in the initial value (see \eqref{eq:finalsubhoelder} in 
Theorem~\ref{thm:finalthm}
below for the precise statement). 
Our proof of Theorem~\ref{thm:finalthm} exploits Corollary~\ref{cor:bigcor} in 
Subsection~\ref{subsec:subhoelderpre}
below and the auxiliary continuity-regularity result in Lemma~\ref{lem:lnhoelder} in 
Subsection~\ref{subsec:subhoelder} below. 
Our proof of Corollary~\ref{cor:bigcor} is based on an application of 
Proposition~\ref{prop:bigthm} below. Our proof of Proposition~\ref{prop:bigthm}, in turn, 
uses the regularity result in Lemma~\ref{lem:lem1} in Subsection~\ref{subsec:condsubhoelder} above.

\subsection{Regularity with respect to the initial value for SDEs with general noise}

\begin{prop}
  \label{prop:bigthm}
  Let $d,m\in\N$,
  $T,\kappa\in[0,\infty)$,
  $\mu\in C^1(\R^d,\R^d)$,
  $\sigma\in\R^{d\times m}$,
  $\varphi\in C(\R^m,[0,\infty))$,
  $V\in C^1(\R^d,[0,\infty))$,
  let $\norm{\cdot}\colon\R^d\to[0,\infty)$ be a norm,
  assume 
    for all 
      $x,h\in\R^d$, 
      $z\in\R^m$ 
    that
      $\norm{\mu'(x)h}\leq \kappa \bp{1+\norm x^\kappa} \norm h$,
      $V'(x)\mu(x+\sigma z)\leq \varphi(z)V(x)$,
      and $\norm x\leq V(x)$,
  let $(\Omega,\mc F,\PP)$ be a probability space,
  let $W\colon[0,T]\times\Omega\to\R^m$ be a stochastic process
    %on $(\Omega,\mc F,\PP)$ 
    with continuous sample paths,
  and assume 
    for all
      $c\in[0,\infty)$ 
    that
      $\bEE{
        \sup_{t\in[0,T]\cap\Q}\exp\bp{c\,\varphi(W(t))}
      }+\bEE{
        \sup_{t\in[0,T]\cap\Q}(\norm{\sigma W(t)}^c)
      }
      <\infty$.
%   let $X^x\colon[0,T]\times\Omega\to\R^d$, $x\in\R^d$, be 
%   %a family of $(\FF_t)_{t\in[0,T]}$-adapted 
%   stochastic processes with continuous sample paths,
%   assume for all $\omega\in\Omega$, $x\in\R^d$, $t\in[0,T]$ that
%   \beq
%   X^x(t,\omega)=x+\int_0^t\mu(X^x(s,\omega))\,\diff s+\sigma W(t,\omega),
%   \eeq
%   and assume for all $c\in(0,\infty)$ that
%   $
%   \EE\bbr{\textstyle\exp\bp{c\sup_{t\in[0,T]\cap\Q}\varphi(W(t))}}<\infty$.
  Then
  \begin{enumerate}[label=(\roman{enumi}),ref=(\roman{enumi})]
  \item \label{enum:propbigthm:1}
    there exist unique stochastic processes
      $X^x\colon [0,T]\times\Omega\to\R^d$, $x\in\R^d$, with continuous sample paths
    which satisfy for all
      $x\in\R^d$,
      $t\in[0,T]$,
      $\omega\in\Omega$
    that
    \beq
      X^x(t,\omega)=x+\int_0^t\mu(X^x(s,\omega))\,\diff s+\sigma W(t,\omega),
    \eeq
  \item \label{enum:propbigthm:1b}
    it holds for all
      $R,r\in[0,\infty)$
    that
    $
      \sup_{x\in\{z\in\R^d\colon \norm{z}\leq R\}}\,\sup_{t\in[0,T]} \bEE{ \norm{X^x(t)}^{r} }<\infty
    $,
  and
  \item \label{enum:propbigthm:2}
    it holds for all 
      $R,q\in[0,\infty)$
      that there exists 
        $c\in(0,\infty)$ 
        such that for all 
          $h\in\{v\in\R^d\setminus\{0\}\colon\norm v<1\}$ 
        it holds that
        \beq
        \label{eq:lnbnd}
          \bbbbr{\sup_{x\in\{v\in\R^d\colon\norm v\leq R\}}\;\sup_{t\in[0,T]}\bEE{\norm{X^{x+h}(t)-X^x(t)}}}
          \leq 
          c\,\babs{\ln(\norm h)}^{-q}.
        \eeq
  \end{enumerate}
\end{prop}
\begin{proof}[Proof of Proposition~\ref{prop:bigthm}]
  First, observe that 
    Lemma~\ref{lem:existX} 
  shows
  \begin{enumerate}[label=(\alph{enumi})]
  \item 
    that there exist unique stochastic processes
      $X^x\colon [0,T]\times\Omega\to\R^d$, $x\in\R^d$, with continuous sample paths
    which satisfy for all
      $x\in\R^d$,
      $t\in[0,T]$,
      $\omega\in\Omega$
    that
    \beq
    \label{eq:Xex}
      X^x(t,\omega)=x+\int_0^t\mu(X^x(s,\omega))\,\diff s+\sigma W(t,\omega),
    \eeq
  \item
    that for all $\omega\in\Omega$ it holds that
    \beq
    \label{eq:XxC01}
      \bp{[0,T]\times\R^d\ni (t,x)\mapsto X^x(t,\omega)\in\R^d}
      \in C^{0,1}([0,T]\times\R^d,\R^d),
    \eeq
    and
  \item 
    that for all 
      $x,h\in\R^d$,
      $t\in[0,T]$, 
      $\omega\in\Omega$
    it holds that
    \beq
    \label{eq:Xdiff2}
      \bp{ \tfrac{\partial}{\partial x} X^x(t,\omega) }(h)
      =
      h+\int_0^t \mu'(X^x(s,\omega))\bp{\bp{\tfrac{\partial}{\partial x} X^x(s,\omega) }(h)}\,\diff s
      .
    \eeq
  \end{enumerate}
  In the next step we note that
%     \eqref{eq:Xex}
%     and the assumption that
%       for all
%         $c\in[0,\infty)$ 
%       it holds that
%         $\bEE{\sup_{t\in[0,T]\cap\Q}\exp\bp{c\,\varphi(W(t))}}<\infty$
%   with
    Lemma~\ref{lem:regcrit}
  ensures
  \begin{enumerate}[label=(\Alph{enumi}),ref=(\Alph{enumi})]
  \item
    that for all 
      $R,r\in[0,\infty)$ 
    it holds that
      \beq
      \Omega\ni\omega\mapsto \bbbbr{\sup_{x\in\{z\in\R^d\colon\norm z\leq R\}}\,\sup_{t\in[0,T]}(\norm{X^x(t,\omega)}^r)}\in[0,\infty]
      \eeq
        is an $\mc F$/$\mc B([0,\infty])$-measurable function
    and
  \item
    that for all 
      $R,r\in[0,\infty)$ 
    it holds that
    \beq
    \label{eq:supXfinexp}
      \bbbEE{ \sup_{x\in\{z\in\R^d\colon \norm{z}\leq R\}}\,\sup_{t\in[0,T]} \bp{\norm{X^x(t)}^{r}} }
      <\infty.
    \eeq
  \end{enumerate}
  Furthermore, observe that for all
    $R,r\in[0,\infty)$,
    $y\in\{z\in\R^d\colon \norm{z}\leq R\}$,
    $s\in[0,T]$
  it holds that
  \beq
    \bEE{\norm{X^y(s)}^r}
    \leq
    \bbbEE{ \sup_{x\in\{z\in\R^d\colon \norm{z}\leq R\}}\,\sup_{t\in[0,T]} \bp{\norm{X^x(t)}^{r}} }
    .
  \eeq
  Combining this with
    \eqref{eq:supXfinexp}
  establishes~\ref{enum:propbigthm:1b}.
  Moreover, note that
    \eqref{eq:supXfinexp}
    and Lemma~\ref{lem:supmeas}
  prove that for all
    $R,r\in[0,\infty)$
  it holds that
  \beq
    \label{eq:supsupfin2}
    \bbbEE{ \sup_{x\in\{z\in\Q^d\colon \norm{z}\leq R\}}\,\sup_{t\in[0,T]\cap\Q} \bp{\norm{X^x(t)}^{r}} }<\infty
    .
  \eeq
%   In addition, observe that
%     \eqref{eq:XxC01}
%   implies that for all
%     $x,h\in\R^d$,
%     $\omega\in\Omega$
%   it holds that 
%     $[0,T]\ni s\mapsto \bp{\frac\partial{\partial x} X^x(s,\omega)}(h)\in\R^d$
%       is a continuous function.
%   This ensures that for all
%     $x,h\in\R^d$,
%     $\omega\in\Omega$
%   it holds that
%   \beq
%   \label{eq:intfin}
%     \int_0^T\bnorm{\bp{\tfrac\partial{\partial x} X^x(s,\omega)}(h)}\,\diff s
%     <
%     \infty
%     .
%   \eeq
  In the next step we combine
    \eqref{eq:Xdiff2},
    the triangle inequality,
    and the hypothesis that
      for all
        $x,h\in\R^d$
      it holds that 
        $\norm{\mu'(x)h}\leq \kappa \bp{1+\norm x^\kappa} \norm h$
  to obtain that for all 
    $x,h\in\R^d$,
    $t\in[0,T]$, 
    $\omega\in\Omega$
  it holds that
  \ba
    \bnorm{\bp{ \tfrac{\partial}{\partial x} X^x(t,\omega) }(h)}
    &\leq
    \norm{h}+\bbbnorm{\int_0^t \mu'(X^x(s,\omega))\bp{\bp{\tfrac{\partial}{\partial x} X^x(s,\omega) }(h)}\,\diff s}\\
    &\leq
    \norm{h}+\int_0^t \bnorm{\mu'(X^x(s,\omega))\bp{\bp{\tfrac{\partial}{\partial x} X^x(s,\omega) }(h)}}\,\diff s\\
    &\leq
    \norm{h}+\int_0^t \kappa(1+\norm{X^x(s,\omega)}^\kappa)\,\bnorm{\bp{\tfrac{\partial}{\partial x} X^x(s,\omega) }(h)}\,\diff s\\
    &=
    \norm{h}+\kappa\int_0^t (1+\norm{X^x(s,\omega)}^\kappa)\,\bnorm{\bp{\tfrac{\partial}{\partial x} X^x(s,\omega) }(h)}\,\diff s
    .
  \ea
  Combining 
    this,
    \eqref{eq:XxC01},
    and \eqref{eq:supsupfin2}
  with
    Lemma~\ref{lem:lem1}
    (with
      $d\is d$,
      $T\is T$,
      %$R\is R$,
      %$p\is \kappa$,
      %$q\is q$,
      $\kappa\is\kappa$,
      $(\Omega,\mc F,\PP)\is(\Omega,\mc F,\PP)$,
      $\norm\cdot\is\norm\cdot$,
      $(X^x)_{x\in\R^d}\is(X^x)_{x\in\R^d}$
      %for $R,q\in[0,\infty)$
    in the notation of Lemma~\ref{lem:lem1})
  establishes
    \ref{enum:propbigthm:2}.
  The proof of Proposition~\ref{prop:bigthm} is thus completed.
\end{proof}

\subsection{Regularity with respect to the initial value for SDEs with Wiener noise}
\label{subsec:subhoelderpre}

\begin{cor}
  \label{cor:bigcor}
  Let $d,m\in\N$, 
  $T,\kappa\in[0,\infty)$,
  $\alpha\in[0,2)$,
  $\mu\in C^1(\R^d,\R^d)$,
  $\sigma\in\R^{d\times m}$,
  $V\in C^1(\R^d,[0,\infty))$,
  let $\norm{\cdot}\colon\R^d\to[0,\infty)$ and $\nnorm{\cdot}\colon\R^m\to[0,\infty)$ be norms,
  assume for all 
    $x,h\in\R^d$, 
    $z\in\R^m$ 
  that
    $\norm{\mu'(x)h}\leq \kappa \bp{1+\norm x^\kappa} \norm h$,
    $V'(x)\mu(x+\sigma z)\leq \kappa(1+\nnorm{z}^\alpha)V(x)$,
    and $\norm x\leq V(x)$,
  let $(\Omega,\mc F,\PP)$ be a probability space,
  and let $W\colon[0,T]\times\Omega\to\R^m$ be a standard Brownian motion
    %on $(\Omega,\mc F,\PP)$ 
    with continuous sample paths.
%   let $X^x\colon[0,T]\times\Omega\to\R^d$, $x\in\R^d$, be 
%   %a family of $(\FF_t)_{t\in[0,T]}$-adapted 
%   stochastic processes with continuous sample paths,
%   and assume for all $\omega\in\Omega$, $x\in\R^d$, $t\in[0,T]$ that
%   \beq
%   X^x(t,\omega)=x+\int_0^t\mu(X^x(s,\omega))\,\diff s+\sigma W(t,\omega).
%   \eeq
  Then
  \begin{enumerate}[label=(\roman{enumi}),ref=(\roman{enumi})]
  \item \label{enum:corbigcor:1}
  there exist unique stochastic processes
  $X^x\colon [0,T]\times\Omega\to\R^d$, $x\in\R^d$, with continuous sample paths
  such that for all 
    $x\in\R^d$,
    $t\in[0,T]$,
    $\omega\in\Omega$
  it holds that
  \beq
  X^x(t,\omega)=x+\int_0^t\mu(X^x(s,\omega))\,\diff s+\sigma W(t,\omega),
  \eeq
  \item \label{enum:corbigcor:1b}
    it holds for all
      $R,r\in[0,\infty)$
    that
    $
      \sup_{x\in\{z\in\R^d\colon \norm{z}\leq R\}}\,
        \sup_{t\in[0,T]} \bEE{ \norm{X^x(t)}^{r} }
      <\infty
    $,
    and
  \item \label{enum:corbigcor:2}
    it holds for all 
      $R,q\in[0,\infty)$ 
      that there exists 
        $c\in(0,\infty)$ 
      such that for all 
        $h\in\{v\in\R^d\setminus\{0\}\colon\norm v<1\}$ 
      it holds that
      \beq
        \bbbbr{
          \sup_{x\in\{v\in\R^d\colon\norm v\leq R\}}\;
            \sup_{t\in[0,T]}\bEE{\norm{X^{x+h}(t)-X^x(t)}}
        }
        \leq 
        c\,\babs{\ln(\norm h)}^{-q}.
      \eeq
  \end{enumerate}
\end{cor}
\begin{proof}[Proof of Corollary~\ref{cor:bigcor}]
  Throughout this proof
    let $\varphi\colon\R^m\to[0,\infty)$ 
    \intrtype{be the function which satisfies }%
    \intrtypen{satisfy }%
    for all 
      $z\in\R^m$ 
    that
    \beq
      \varphi(z)=\kappa(1+\nnorm{z}^\alpha).
    \eeq
  Note that 
    Lemma~\ref{lem:regcrithelp}
    (with
      $m\is m$,
      $T\is T$,
      $C\is\kappa$,
      $\alpha\is \alpha$,
      $\norm\cdot\is \nnorm{\cdot}$,
      $(\Omega,\mc F,\PP)\is(\Omega,\mc F,\PP)$,
      $W\is W$,
      $\varphi\is\varphi$
    in the notation of Lemma~\ref{lem:regcrithelp})
  shows that for all
    $c\in[0,\infty)$
  it holds that
  \beq
  \label{eq:EexpW2}
    \bbbEE{\sup_{t\in[0,T]\cap\Q}\exp\bp{c\,\varphi(W(t))}}
    <\infty
    .
  \eeq
  In addition, observe that
    Lemma~\ref{lem:expWr}
    (with
      $d\is d$,
      $m\is m$,
      $T\is T$,
      $r\is c$,
      $\sigma\is\sigma$,
      $\norm\cdot\is\norm\cdot$,
      $(\Omega,\mc F,\PP)\is(\Omega,\mc F,\PP)$,
      $W\is W$
      for $c\in[0,\infty)$
    in the notation of Lemma~\ref{lem:expWr})
  ensures that for all
    $c\in[0,\infty)$
  it holds that
  \beq
    \bbbEE{\sup_{t\in[0,T]\cap\Q}(\norm{\sigma W(t)}^c)}
    <\infty
    .
  \eeq
  Combining
    this
  and
    \eqref{eq:EexpW2}
  with
    Proposition~\ref{prop:bigthm}
  establishes~\ref{enum:corbigcor:1}, \ref{enum:corbigcor:1b},
  and~\ref{enum:corbigcor:2}.
  The proof of Corollary~\ref{cor:bigcor} is thus completed.
\end{proof}

\subsection{Sub-Hoelder continuity with respect to the initial value for SDEs}
\label{subsec:subhoelder}

\begin{lemma}
\label{lem:lnhoelder}
  Let $d\in\N$,
  $T,R,q,c,C\in[0,\infty)$,
  let $(\Omega,\mc F,\PP)$ be a probability space,
  let $X^x\colon[0,T]\times\Omega\to\R^d$, $x\in\R^d$, be stochastic processes,
  let $\norm\cdot\colon\R^d\to[0,\infty)$ be a norm,
  assume 
    for all
      $h\in\{v\in\R^d\setminus\{0\}\colon\norm{v}<1\}$
    that
    \beq
    \label{eq:logcont1}
      \bbbbr{
        \sup_{x\in\{v\in\R^d\colon\norm v\leq R\}}
          \sup_{t\in[0,T]}
            \bEE{\norm{X^{x+h}(t)-X^x(t)}}
      }
      \leq 
      c\,\babs{\ln(\norm h)}^{-q}
      ,
    \eeq
  and assume that
    $
      C
      =
      \sup_{x\in\{v\in\R^d\colon \norm v\leq R\}}\sup_{t\in[0,T]}\bEE{\norm{X^x(t)}}
    $.
  Then it holds for all 
    $x,y\in\{v\in\R^d\colon\norm v\leq R\}$ with $0<\norm{x-y}\neq 1$
  that
  \beq
    \sup_{t\in[0,T]}\bEE{\norm{X^x(t)-X^y(t)}}
    \leq 
    \max\{c,2C\,\abs{\ln(2R+1)}^q\}\, \babs{\ln(\norm{x-y})}^{-q}
    .
  \eeq
\end{lemma}
\begin{proof}[Proof of Lemma~\ref{lem:lnhoelder}]
  First, note that 
    \eqref{eq:logcont1}
  implies that for all
    $x,y\in\{v\in \R^d\colon\norm v\leq R\}$ with $x\neq y$ and $\norm{x-y}<1$
  it holds that
  \beq
    \label{eq:smalldistbnd}
    \sup_{t\in[0,T]}\bEE{\norm{X^x(t)-X^y(t)}}
    =
    \sup_{t\in[0,T]}\bEE{\norm{X^{y+(x-y)}(t)-X^y(t)}}
    \leq
    \frac{c}{\abs{\ln(\norm{x-y})}^{q}}.
  \eeq
  Furthermore, observe that 
    the triangle inequality
    and the hypothesis that $C<\infty$
  show that for all
    $x,y\in\{v\in \R^d\colon\norm v\leq R\}$,
    $t\in[0,T]$
  it holds that
  \beq
    \bEE{\norm{X^x(t)-X^y(t)}}
    \leq
    \bEE{\norm{X^x(t)}+\norm{X^y(t)}}
    =
    \bEE{\norm{X^x(t)}}+\bEE{\norm{X^y(t)}}
    \leq 
    2C
    .
  \eeq
    The fact that
      for all 
        $q\in[0,\infty)$
      it holds that
        $[1,\infty)\ni z\mapsto \abs{\ln(z)}^q\in\R$ is an increasing function
    and the fact that
      for all
        $x,y\in\{v\in \R^d\colon \norm v\leq R\}$
      it holds that
        $\norm{x-y}\leq 2R$
    hence
  show that for all
    $t\in[0,T]$,
    $x,y\in\{v\in \R^d\colon \norm v\leq R\}$ with $\norm{x-y}> 1$
  it holds that
  \beq
    \sup_{t\in[0,T]}\bEE{\norm{X^x(t)-X^y(t)}}
    \leq
    \frac{2C\,\abs{\ln(\norm{x-y})}^q}{\abs{\ln(\norm{x-y})}^q}
    \leq
    \frac{2C\,\abs{\ln(2R+1)}^q}{\abs{\ln(\norm{x-y})}^q}
    .
  \eeq
  Combining 
    this 
  with 
    \eqref{eq:smalldistbnd}
  demonstrates that for all
    $x,y\in\{v\in \R^d\colon \norm v\leq R\}$ with $0<\norm{x-y}\neq 1$
  it holds that
  \ba
    \sup_{t\in[0,T]}\bEE{\norm{X^x(t)-X^y(t)}}
    &\leq 
    \max\biggl\{ \frac{c}{\abs{\ln(\norm{x-y})}^{q}}, \frac{2C\,\abs{\ln(2R+1)}^q}{\abs{\ln(\norm{x-y})}^q} \biggr\}
    \\&=
    \max\{c,2C\,\abs{\ln(2R+1)}^q\}\, \babs{\ln(\norm{x-y})}^{-q}
    .
  \ea
  The proof of Lemma~\ref{lem:lnhoelder} is thus completed.
\end{proof}

\begin{theorem}
  \label{thm:finalthm}
  Let $d,m\in\N$, 
  $T,\kappa\in[0,\infty)$,
  $\alpha\in[0,2)$,
  $\mu\in C^1(\R^d,\R^d)$,
  $\sigma\in\R^{d\times m}$,
  $V\in C^1(\R^d,[0,\infty))$,
  let $\norm{\cdot}\colon\R^d\to[0,\infty)$ and $\nnorm\cdot\colon\R^m\to[0,\infty)$ be norms,
  assume for all 
    $x,h\in\R^d$, 
    $z\in\R^m$ 
  that
    $\norm{\mu'(x)h}\leq \kappa \bp{1+\norm x^\kappa} \norm h$,
    $V'(x)\mu(x+\sigma z)\leq \kappa(1+\nnorm{z}^\alpha)V(x)$,
    and $\norm x\leq V(x)$,
  let $(\Omega,\mc F,\PP)$ be a probability space,
  and let $W\colon[0,T]\times\Omega\to\R^m$ be a standard Brownian motion
    %on $(\Omega,\mc F,\PP)$ 
    with continuous sample paths.
%   let $X^x\colon[0,T]\times\Omega\to\R^d$, $x\in\R^d$, be 
%   %a family of $(\FF_t)_{t\in[0,T]}$-adapted 
%   stochastic processes with continuous sample paths,
%   and assume for all $\omega\in\Omega$, $x\in\R^d$, $t\in[0,T]$ that
%   \beq
%   X^x(t,\omega)=x+\int_0^t\mu(X^x(s,\omega))\,\diff s+\sigma W(t,\omega).
%   \eeq
  Then
  \begin{enumerate}[label=(\roman{enumi}),ref=(\roman{enumi})]
  \item \label{enum:thmfinalthm:1}
    there exist unique stochastic processes
    $X^x\colon [0,T]\times\Omega\to\R^d$, $x\in\R^d$, with continuous sample paths
    such that for all 
      $x\in\R^d$,
      $t\in[0,T]$,
      $\omega\in\Omega$
    it holds that
    \beq
    X^x(t,\omega)=x+\int_0^t\mu(X^x(s,\omega))\,\diff s+\sigma W(t,\omega)
    \eeq
    and
  \item \label{enum:thmfinalthm:2}
    it holds for all 
      $R,q\in[0,\infty)$
      that there exists 
        $c\in(0,\infty)$ 
      such that for all 
        $x,y\in\{v\in \R^d\colon \norm v\leq R\}$ with $0<\norm{x-y}\neq 1$
      it holds that
      \beq
      \label{eq:finalsubhoelder}
        \sup_{t\in[0,T]}\bEE{\norm{X^x(t)-X^y(t)}}\leq c\,\babs{\ln(\norm{x-y})}^{-q}
        .
      \eeq
  \end{enumerate}
\end{theorem}
\begin{proof}[Proof of Theorem~\ref{thm:finalthm}]
  Throughout this proof
    let $\varphi\colon\R^m\to[0,\infty)$ 
    \intrtype{be the function which satisfies }%
    \intrtypen{satisfy }%
    for all 
      $z\in\R^m$ 
    that
    \beq
      \varphi(z)=\kappa(1+\nnorm z^\alpha).
    \eeq
  Note that
    Corollary~\ref{cor:bigcor}
  establishes
  \begin{enumerate}[label=(\alph{enumi})]
  \item
  that there exist unique stochastic processes
  $X^x\colon [0,T]\times\Omega\to\R^d$, $x\in\R^d$, with continuous sample paths
  such that for all 
    $x\in\R^d$,
    $t\in[0,T]$,
    $\omega\in\Omega$
  it holds that
  \beq
  \label{eq:Xdef4}
    X^x(t,\omega)=x+\int_0^t\mu(X^x(s,\omega))\,\diff s+\sigma W(t,\omega),
  \eeq
  \item
    that for all
      $R,r\in[0,\infty)$
    it holds that
    \beq
    \label{eq:supsupE}
      \sup_{x\in\{z\in\R^d\colon \norm{z}\leq R\}}\,\sup_{t\in[0,T]} \bEE{ \norm{X^x(t)}^{r} }<\infty,
    \eeq
    and
  \item
    that there exist $c_{R,q}\in(0,\infty)$, $R,q\in[0,\infty)$,
      such that for all
        $R,q\in[0,\infty)$,
        $h\in\{v\in\R^d\setminus\{0\}\colon\norm v<1\}$ 
      it holds that
      \beq
      \label{eq:lnhoelder2}
        \sup_{x\in\{v\in\R^d\colon\norm v\leq R\}}\;\sup_{t\in[0,T]}\bEE{\norm{X^{x+h}(t)-X^x(t)}}\leq c_{R,q}\,\babs{\ln(\norm h)}^{-q}.
      \eeq
  \end{enumerate}
  Combining
    \eqref{eq:supsupE}
    and \eqref{eq:lnhoelder2}
  with
    Lemma~\ref{lem:lnhoelder}
    (with
      $d\is d$,
      $T\is T$,
      $R\is R$,
      $q\is q$,
      $c\is c_{R,q}$,
      $C\is \sup_{x\in\{z\in\R^d\colon \norm z\leq R\}}\sup_{t\in[0,T]}\bEE{\norm{X^x(t)} }$,
      $(\Omega,\mc F,\PP)\is(\Omega,\mc F,\PP)$,
      $(X^x)_{x\in\R^d}\is (X^x)_{x\in\R^d}$,
      $\norm\cdot\is\norm\cdot$
      for
      $R,q\in[0,\infty)$
    in the notation of Lemma~\ref{lem:lnhoelder})
  establishes~\ref{enum:thmfinalthm:2}.
  The proof of Theorem~\ref{thm:finalthm} is thus completed.
\end{proof}

\bibliographystyle{acm}
\bibliography{mgy}

\end{document}